\documentclass[11pt]{amsart}
\usepackage[T1]{fontenc}
\usepackage[utf8]{inputenc}
\usepackage[margin=1.05in]{geometry}
\usepackage{amsmath,amssymb,mathtools,mathrsfs}
\usepackage{etoolbox,letltxmacro}
\usepackage{enumitem,booktabs,array}
\LetLtxMacro{\standardmathbb}{\mathbb}
\renewcommand{\mathbb}[1]{\ifstrequal{#1}{k}{\Bbbk}{\standardmathbb{#1}}}
\usepackage{microtype}
\usepackage[colorlinks=true,linkcolor=blue,citecolor=blue,urlcolor=blue]{hyperref}
\allowdisplaybreaks[2]
\hypersetup{pdftitle={Rationality of cohomological descendent series for Quot schemes on surfaces},pdfauthor={Reginald Anderson}}
\numberwithin{equation}{section}
\newtheorem{theorem}{Theorem}[section]
\newtheorem{proposition}[theorem]{Proposition}
\newtheorem{lemma}[theorem]{Lemma}
\newtheorem{corollary}[theorem]{Corollary}
\newtheorem{conjecture}[theorem]{Conjecture}
\theoremstyle{definition}
\newtheorem{definition}[theorem]{Definition}
\newtheorem{example}[theorem]{Example}
\theoremstyle{remark}
\newtheorem{remark}[theorem]{Remark}
\newcommand{\C}{\mathbb C}
\newcommand{\Q}{\mathbb Q}
\newcommand{\Z}{\mathbb Z}
\newcommand{\Qb}{\overline{\mathbb Q}}
\newcommand{\PP}{\mathbb P}
\newcommand{\OO}{\mathcal O}
\newcommand{\K}{\mathcal K}
\newcommand{\Quot}{\operatorname{Quot}}
\newcommand{\Qn}{\Quot_n}
\newcommand{\Hilb}{\operatorname{Hilb}}
\newcommand{\ch}{\operatorname{ch}}
\newcommand{\td}{\operatorname{td}}
\newcommand{\rk}{\operatorname{rk}}
\newcommand{\vir}{\mathrm{vir}}
\newcommand{\Res}{\operatorname{Res}}
\newcommand{\RHom}{R\mathcal H\!om}
\newcommand{\Hom}{\operatorname{Hom}}
\newcommand{\Ext}{\operatorname{Ext}}
\newcommand{\pr}{\operatorname{pr}}
\newcommand{\vd}{\operatorname{vd}}
\newcommand{\RF}{\mathscr F_N}
\newcommand{\eps}{\boldsymbol\epsilon}
\newcommand{\mm}{\boldsymbol m}
\newcommand{\LL}{\boldsymbol L}
\newcommand{\bbeta}{\boldsymbol\beta}
\newcommand{\YY}{\mathscr Y}
\newcommand{\cZ}{\mathcal Z}
\newcommand{\cW}{\mathcal W}
\newcommand{\cC}{\mathcal C}
\newcommand{\cQ}{\mathcal Q}
\newcommand{\cA}{\mathcal A}
\newcommand{\cT}{\mathcal T}
\newcommand{\cJ}{\mathcal J}
\newcommand{\cD}{\mathcal D}
\newcommand{\cL}{\mathcal L}
\newcommand{\TT}{\mathsf T}
\newcommand{\MM}{\mathsf M}
\newcommand{\HB}{\mathsf H_B}
\newcommand{\Jb}{\overline{J}}
\newcommand{\intvir}[1]{\int_{[#1]^{\vir}}}
\title[Descendent series for surface Quot schemes]{Rationality of cohomological descendent series for Quot schemes on surfaces}
\author{Reginald Anderson}
\address{Department of Mathematics, University of California, Irvine, 340 Rowland Hall, Irvine, CA 92697-3875}
\email{reginald.anderson.cc@gmail.com}
\email{reginala@uci.edu}
\date{\today}
\begin{document}
\begin{abstract}
Let $S$ be a smooth connected projective complex surface. Johnson, Oprea and
Pandharipande conjectured that the generating series of descendent integrals
against the virtual classes of the Quot schemes of rank-zero quotients of
$\OO_S^{\oplus N}$, weighted by the total Chern class of the virtual tangent
complex and indexed by the Euler characteristic of the quotient, is the Laurent
expansion of a rational function of $q$. We prove this for every $S$, every
$N\geq1$ and every effective curve class $\beta$. The statement is new when
$p_g(S)=0$, $\beta\neq0$ and $N>1$ (and, for $N=1$, when in addition $S$ is
not simply connected); it includes surfaces of nonzero irregularity. The
argument does not distinguish between the
new and the previously known cases; for $\beta=0$ we check that it reproduces
the formula of Oprea and Pandharipande for the virtual Euler characteristics.
The proof combines torus localization, a universality statement for Hilbert
schemes of points relative to a base, a reduction to Hilbert schemes of points
on families of integral curves, and an algebraic lemma that converts
coefficientwise regularity in the equivariant parameter into rationality.
\end{abstract}
\maketitle
\tableofcontents

\section{Introduction}
\label{sec:introduction}

\subsection{The rationality problem}
Let $S$ be a smooth connected projective surface over $\C$ and let $N\geq1$.
For an effective curve class $\beta\in H^2(S,\Z)$ and an integer $n$, the Quot
scheme
\[
 \Qn=\Quot_S(\OO_S^{\oplus N},\beta,n)
\]
parametrizes quotients $\OO_S^{\oplus N}\twoheadrightarrow Q$ in which $Q$ has
rank zero, $c_1(Q)=\beta$ and $\chi(Q)=n$. Such a $Q$ is supported on a curve
in class $\beta$ together with finitely many points. The scheme $\Qn$ is
projective but usually singular and of the wrong dimension. It nevertheless
carries a two-term perfect obstruction theory, and therefore a virtual
fundamental class of dimension $Nn+\beta^2$ \cite[Sections~1.6 and~1.8]{OP}.

Integrals against this class are organized as follows. A class
$\alpha\in K^0(S)$ defines a tautological class
$\alpha^{[n]}=R\pi_*(\cQ\otimes\pr_S^*\alpha)\in K^0(\Qn)$, where $\cQ$ is the
universal quotient on $S\times\Qn$. Fix classes
$\alpha_1,\ldots,\alpha_\ell\in K^0(S)$ and integers $d_1,\ldots,d_\ell\geq0$.
The \emph{descendent series} of Johnson, Oprea and Pandharipande is
\begin{equation}\label{eq:intro-series}
 Z_{S,N,\beta}(\alpha_1,\ldots,\alpha_\ell\mid d_1,\ldots,d_\ell)(q)
 =\sum_{n\in\Z}q^n\intvir{\Qn}
   \ch_{d_1}(\alpha_1^{[n]})\cdots\ch_{d_\ell}(\alpha_\ell^{[n]})\,
   c(T^{\vir}_{\Qn}),
\end{equation}
where $c(T^{\vir}_{\Qn})$ is the total Chern class of the virtual tangent
complex. With no insertions ($\ell=0$), the coefficient of $q^n$ is the
virtual Euler characteristic of $\Qn$ studied in \cite{OP}.

\begin{conjecture}[{\cite[Conjecture~2]{JOP}}; {\cite[Conjecture~25]{OP}} for $\ell=0$]
\label{conj:JOP}
The series \eqref{eq:intro-series} is the Laurent expansion at $q=0$ of a
rational function of $q$.
\end{conjecture}

In \cite{OP,JOP} the conjecture is formulated for simply connected surfaces,
with the remark that it can also be made in general
\cite[Section~0.2, footnote]{JOP}. The cases established in the literature
are collected in Table~\ref{tab:known}.

\begin{table}[ht]
\centering
\begin{tabular}{@{}llll@{}}
\toprule
surface $S$ & $N$ & class $\beta$ & reference\\
\midrule
arbitrary & all $N$ & $\beta=0$ & \cite[Theorem~2]{JOP}\\
simply connected & $N=1$ & all $\beta$ & \cite[Theorem~3]{JOP}\\
$p_g(S)>0$ & all $N$ & all $\beta$ &
  \cite[Theorem~7]{Lim} ($\ell=0$), \cite[Section~1.2 and Remark~19]{AJLOP}\\
\bottomrule
\end{tabular}
\medskip
\caption{Previously known cases of Conjecture~\ref{conj:JOP}.}
\label{tab:known}
\end{table}

\noindent
What remained open is the case $p_g(S)=0$, $\beta\neq0$, $N>1$ (and, for
$N=1$, surfaces with $p_g=0$ that are not simply connected).

\begin{theorem}\label{thm:main}
Let $S$ be a smooth connected projective surface over $\C$, let $N\geq1$, and
let $\beta\in H^2(S,\Z)$ be an effective curve class. For every choice of
$\alpha_1,\ldots,\alpha_\ell\in K^0(S)$ and $d_1,\ldots,d_\ell\geq0$, the
series \eqref{eq:intro-series} is the Laurent expansion at $q=0$ of an element
of $\Q(q)$.
\end{theorem}

\begin{corollary}\label{cor:new-case}
Conjecture~\ref{conj:JOP} holds when $p_g(S)=0$, $\beta\neq0$ and $N>1$,
without any assumption on the irregularity $h^1(S,\OO_S)$.
\end{corollary}

\begin{remark}[On the generality of the statement]\label{rem:uniform}
The hypotheses $p_g(S)=0$, $\beta\neq0$ and $N>1$ are not used anywhere in
the proof, which is why Theorem~\ref{thm:main} is stated for all surfaces,
all $N$ and all $\beta$. For the cases in Table~\ref{tab:known} the theorem
gives nothing new, but it provides a test: the method must reproduce the
known closed formulas. We carry this out for $\beta=0$ in
Section~\ref{sec:checks}, where the argument is shown to recover the formula
of Oprea and Pandharipande for all $N$
(Corollary~\ref{cor:OP}).
\end{remark}

\subsection{Where the difficulty lies}
The standard tool is localization for the action of $\C^*$ on
$\OO_S^{\oplus N}$ with distinct weights $w_1,\ldots,w_N$. A fixed quotient
is a direct sum of $N$ quotients of $\OO_S$, so the fixed loci are products
\[
 \prod_{i=1}^N S^{[m_i]}\times\prod_{i=1}^N\Hilb_{\beta_i}(S),
 \qquad \beta_1+\cdots+\beta_N=\beta,
\]
of Hilbert schemes of points and Hilbert schemes of effective divisors
\cite[Lemma~1 and Section~2.2.1]{Lim}. The divisor factor enters through its
own virtual class, of dimension $\beta_i(\beta_i-K_S)/2$.

When $p_g(S)>0$, only Seiberg--Witten basic classes contribute and for those
the virtual dimension is zero \cite[Section~1.4]{Lim}. The virtual class of
the divisor factor is then a number times a point. One may freeze the
divisors $D_i$, and the problem becomes one about Hilbert schemes of points
with \emph{fixed} line bundles $\OO_S(D_i)$; this is the route of
\cite{Lim,AJLOP}.

When $p_g(S)=0$, the virtual class of $\Hilb_{\beta_i}(S)$ can have positive
dimension. The line bundles $\OO_S(D_i)$ then vary in a family that cannot be
frozen, and the answer involves intersection theory on the parameter space of
divisors. For $N=1$ and $S$ simply connected this parameter space is a
projective space, which is what makes \cite[Theorem~3]{JOP} accessible. For
$N>1$ there is a second complication: the integrand couples the $N$ factors
through the mixed complexes $R\pi_*\RHom(K_i,Q_j)$, $i\neq j$, which carry the
equivariant weights $w_j-w_i$.

The present paper treats the parameter space of divisors as an arbitrary
smooth projective variety $B$. Every step of the argument is carried out
\emph{relative to $B$}, that is, for families of line bundles on $S$
parametrized by $B$ and for invariants with values in the cohomology of
$B$. The contributions of the mixed complexes $R\pi_*\RHom(K_i,Q_j)$ with
$i\neq j$, which we call the \emph{mixed terms}, are included at each
step.

We do not obtain a closed formula for the generating series, and none is
needed. To prove rationality it suffices to show two things: that the
equivariant generating series belongs to an explicit field $\RF$ of
algebraic functions of $q$ and $z$ (Definition~\ref{def:root-field}), and
that every element of $\RF$ whose coefficients are regular at $z=0$
specializes at $z=0$ to a rational function of $q$
(Theorem~\ref{thm:specialization}).

\subsection{How rationality arises in the simplest case}\label{subsec:mechanism}
The following computation, carried out in detail in
Section~\ref{subsec:P2-example}, shows in the simplest case the three steps
by which the proof of Theorem~\ref{thm:main} produces a rational generating
series. Take $S=\PP^2$ and $N=1$, and consider
\[
 a_m=\int_{S^{[m]}}e\bigl((\OO(1)^{[m]})^\vee\bigr)\,c(W_m),
\]
where $\OO(1)^{[m]}$ is the tautological bundle and $W_m$ is a certain
tautological class of rank $m+2$ on the Hilbert scheme of $m$ points.
\begin{enumerate}[label=(\roman*),leftmargin=2.2em]
\item \emph{Euler classes cut down to a curve.} A section of $\OO(1)$
vanishing on a line $C$ induces a section of $\OO(1)^{[m]}$ vanishing
exactly on the subschemes contained in $C$. So the Euler class is
represented, up to the sign $(-1)^m$, by $C^{[m]}\cong\PP^m$.
\item \emph{On the curve the length $m$ only appears as an exponent.} On
$\PP^m$ one finds $W_m=3\zeta+(m-2)\zeta^{-1}+1$ with $\zeta=\OO(1)$, so that,
with $h=c_1(\zeta)$,
\[
 a_m=(-1)^m\,[h^m]\,(1+h)^3(1-h)^{m-2}
    =\Res_{x=0}\,\frac{(1-x)^3}{x(1+x)^2}\Bigl(\frac{1+x}{x}\Bigr)^{m}dx .
\]
\item \emph{Summing over $m$ substitutes a root of an algebraic equation.}
A sum $\sum_mq^m\Res f(x)R(x)^m\,dx$ with $R(x)=(1+x)/x$ is evaluated by the
change of variable $t=1/R(x)$; it becomes $f$ evaluated at the solution of
$x=q(x+1)$, that is at $x=q/(1-q)$. Here the result is
$\sum_ma_mq^m=(1-2q)^3/(1-q)^2$.
\end{enumerate}
For general $N$ the function $R$ is replaced by
\[
 R_z(x)=\frac{P_z(x+1)}{P_z(x)},\qquad P_z(x)=\prod_{i=1}^N(x-w_iz),
\]
where $z$ is the equivariant parameter, and step (iii) substitutes the roots
of $P_z(x)=q\,P_z(x+1)$. These roots are algebraic, not rational, functions
of $q$. Rationality reappears only after the equivariant parameter is set
to zero, and explaining why is the purpose of the algebraic
Theorem~\ref{thm:specialization}.

\subsection{Strategy of the proof}\label{subsec:strategy}
The proof has six steps. The first two and the last are formal; the
geometric work is in Steps 3--5.

\medskip\noindent
\textbf{Step 1: localization (Section~\ref{sec:localization}).}
Let $J_n(z)$ be the $\C^*$-equivariant version of the coefficient of $q^n$ in
\eqref{eq:intro-series}. It is a rational function of $z$ that is regular at
$z=0$, and $J_n(0)$ is the original coefficient
(Proposition~\ref{prop:equivariant-coefficient}). Virtual localization
\cite{GP} writes $J_n(z)$ as a finite sum of integrals over the fixed loci.
After replacing the virtual class of the divisor factor by a combination of
smooth projective varieties $B$ (Lemma~\ref{lem:cycle-resolution}), the
equivariant series $\sum_nq^nJ_n(z)$ becomes a finite sum of integrals over
such $B$ of a \emph{Hilbert series} $\YY_{\LL}$
(Definition~\ref{def:hilbert-series}). This is a generating series of
pushforwards from $\prod_iS^{[m_i]}\times B$ to $B$, depending on line bundles
$L_1,\ldots,L_N$ on $S\times B$.

\medskip\noindent
\textbf{Step 2: universality and twisting
(Sections~\ref{sec:genus} and~\ref{sec:twisting}).}
By an induction in the style of Ellingsrud, G\"ottsche and Lehn, carried out
relative to $B$, the coefficients of $\YY_{\LL}$ are universal polynomials in
tautological classes $\kappa_P\in H^*(B)$ of the family
(Proposition~\ref{prop:relative-universality}), and the logarithm of the
normalized series is \emph{linear} in these classes
(Lemma~\ref{lem:linear-log}). Consequently, replacing every $L_i$ by
$L_i(-kH)$ for an ample $H$ changes the logarithm by a polynomial of degree
two in $k$. The series for $k=0$ is therefore a product of integer powers of
the series for three consecutive large values of $k$
(Proposition~\ref{prop:twisting}). So we may assume that the line bundles
$M_i=K_S\otimes L_i^{-1}$, whose tautological bundles appear as Euler classes
in $\YY_{\LL}$, are as positive as we like on the fibers of $S\times B\to B$.

\medskip\noindent
\textbf{Step 3: families of integral curves (Section~\ref{sec:curves}).}
For positive $M_i$, and after a further twist by a large power $\OO_B(s)$ of
an ample line bundle on $B$, there are sections of $M_i(s)$ whose zero loci
$\cC_i\subset S\times B$ are families of \emph{integral} curves over all of
$B$ (Lemma~\ref{lem:integral-curves}). The twist by $\OO_B(s)$ changes the
series polynomially in $s$ and is removed by interpolation
(Lemma~\ref{lem:base-interpolation}). As in (i) above, the Euler classes
are now represented by relative Hilbert schemes of points
$\cC_i^{[m_i]}\to B$ (Corollary~\ref{cor:euler-curves}).

\medskip\noindent
\textbf{Step 4: Abel maps over a fixed base
(Section~\ref{subsec:abel}).}
For $m>2g-2$ the Abel map from $\cC^{[m]}$ to the relative compactified
Jacobian is, \'etale locally, a projective bundle. Tensoring with $\OO_S(H)$
identifies the Jacobians in degrees $m$ and $m+H\cdot\cC_b$. Pulling back to
a fixed Hilbert scheme $\cC^{[m_*]}$, which carries an honest universal
sheaf, turns all $\cC^{[m]}$ with $m\equiv m_*$ into projective bundles
$\PP(E(k))$ over one base $\cT$ that does not depend on $k$
(Proposition~\ref{prop:fixed-base}).

\medskip\noindent
\textbf{Step 5: projective bundle integrals are residues
(Section~\ref{sec:residue-geometry}).}
On $\PP(E(k))$ the integrand is an explicit function of the hyperplane
classes. As in (ii), the lengths $m_j$ enter through the factors
$R_z(x_j)^{m_j}$ and otherwise only polynomially
(Lemma~\ref{lem:length-dependent-factor} and
Proposition~\ref{prop:residue-form}). As in (iii), the sum
over the lengths lies in the field $\RF$ generated over $\Qb(q,z)$ by the
roots of the polynomials $P_z(x)-\xi q\,P_z(x+1)$, where $\xi$ runs over
roots of unity (Propositions~\ref{prop:residue-field}
and~\ref{prop:curve-field}).

\medskip\noindent
\textbf{Step 6: regular specialization (Section~\ref{sec:fields}).}
An element of $\RF$ whose $q$-coefficients are regular at $z=0$ has rational
specialization at $z=0$ (Theorem~\ref{thm:specialization}). The reason is
that at $z=0$ the defining equations become $x^N=\xi q(x+1)^N$, whose roots
are \emph{rational} functions of $q^{1/N}$.

\medskip
The steps are assembled in Section~\ref{sec:main-proof}.
Section~\ref{sec:checks} contains examples, the comparison with the known
formulas for $\beta=0$, and a numerical verification of the intermediate
statements in examples.

\subsection{Notation}\label{subsec:notation}
For the reader's convenience we list the notation that is used across
sections.
\begin{itemize}[leftmargin=1.6em,itemsep=2pt]
\item $S$: a smooth projective surface over $\C$, connected unless said
otherwise; $K_S$ denotes the canonical line bundle or its first Chern
class; $H$ is a fixed ample divisor on $S$.
\item $N$, $\beta$, $n$, $\Qn$: as above. $w_1,\ldots,w_N$: distinct
integers; $\TT=\C^*$ acts on the $i$th summand of $\OO_S^{\oplus N}$ with
weight $w_i$; $z=c_1^{\TT}(\C_1)$ is the equivariant parameter.
\item $\alpha_a,d_a$ ($1\leq a\leq\ell$): the descendent insertion;
$D=\sum_ad_a$.
\item $P_z(x)=\prod_i(x-w_iz)$, $R_z(x)=P_z(x+1)/P_z(x)$, and $\RF$: the root
field (Definition~\ref{def:root-field}).
\item $B$: a smooth connected projective variety (a parameter space of
divisors); $\HB=H^{\mathrm{even}}(B,\Q)$; $\OO_B(1)$: a fixed ample line
bundle on $B$.
\item $L_1,\ldots,L_N$: line bundles on $S\times B$; $M_i=K_S\otimes L_i^{-1}$.
\item $X_{\mm}=\prod_iS^{[m_i]}\times B$ for $\mm=(m_1,\ldots,m_N)$, with
universal subschemes $\cZ_i$; on $S\times X_{\mm}$,
\[
 K_i=I_{\cZ_i}\otimes L_i^{-1},\qquad Q_i=\OO-K_i,\qquad
 W_{ij}=R\pi_*\RHom(K_i,Q_j).
\]
\item $\MM_w(E)=c^{\TT}(E[w])/e_{\TT}(E[w])$, the ratio of equivariant total
Chern class and Euler class of $E$ twisted by the weight $w$
\eqref{eq:chern-euler-ratio}.
\item $\YY_{\LL}$: the Hilbert series; $\YY_{\LL}^{(s)}$: its base-twisted
variant (Definitions~\ref{def:hilbert-series}
and~\ref{def:base-twisted-series}).
\item $\cC_i\subset S\times B$: families of integral curves;
$\cC_i^{[m]}=\Hilb^m(\cC_i/B)$; $g_i$: arithmetic genus of the fibers;
$\nu_i=H\cdot[\cC_{i,b}]$; $\Jb_i^{\,m}$: relative compactified Jacobian of
degree $m$.
\item $\cT$: the fixed smooth projective base of
Proposition~\ref{prop:fixed-base}; $E_i(k)$: vector bundles on $\cT$;
$\zeta_i=\OO_{\PP(E_i(k_i))}(1)$ and $h_i=c_1(\zeta_i)$.
\end{itemize}
The letter $T$ with a subscript always denotes a tangent bundle or a
virtual tangent complex.

\subsection{AI Declaration}\label{subsec:ai}
Generative AI tools were used extensively in the research and in the
preparation of this manuscript, in two stages.

\emph{First stage.} The author used ChatGPT (OpenAI; models ChatGPT 5.6 Pro
and ChatGPT 6 Pro, multiple sessions during 2026) for mathematical
development, proof drafting and critical review, literature assistance,
computational code, exposition and \LaTeX{} production. The author
supplied the research problem, reference materials and referee comments,
and directed successive revisions. AI-generated proposals formed
substantial parts of the proof strategy, of the intermediate statements
and of the proof text: in particular the approach through proper virtual
localization, the relative Hilbert-scheme calculations and the twisting
identity, the curve-family and compactified-Jacobian constructions, and the
residue and specialization arguments.

\emph{Second stage.} The author then used Claude (Anthropic; model Claude
Fable 5.2, November and December 2026) to audit the resulting draft and to
revise it. In this stage the AI tool checked the statements and proofs of
the draft, compared the citations with the cited sources where these were
accessible, reorganized and rewrote the exposition and the notation of the
whole paper, and wrote the present form of the proofs of
Proposition~\ref{prop:equivariant-coefficient}, of
Section~\ref{subsec:universality} and of
Proposition~\ref{prop:residue-field}. It also contributed
Proposition~\ref{prop:compute-specialization} and
Section~\ref{sec:checks}, including the closed formula of
Proposition~\ref{prop:closed-formula} and the comparison with \cite{OP}.

\emph{Calculations.} In the second stage the AI tool also wrote and ran
the programs in Python, using exact rational arithmetic, with which the
computations of Sections~\ref{subsec:computer}
and~\ref{subsec:verification} were carried out. The numerical values in
Examples~\ref{ex:new-case} and~\ref{ex:kunneth} and the comparisons of
Table~\ref{tab:verification} were obtained with these programs. No proof
in this paper depends on them.

Responsibility for the mathematical claims, the references and the final
text remains entirely with the author. The AI-assisted reviews and the
computations described above do not constitute an independent
certification of the manuscript.

\begingroup
\makeatletter
\let\@tocwrite\@gobbletwo
\makeatother
\subsection*{Acknowledgements}
The author is grateful to Dragos Oprea for recommending this project.
\par
\endgroup

\section{Quot schemes and descendent series}
\label{sec:setup}

\subsection{Conventions}\label{subsec:conventions}
Cohomology has rational coefficients. For a projective scheme $X$, $K^0(X)$
denotes the Grothendieck group of algebraic vector bundles on $X$. It is
canonically isomorphic to the Grothendieck group of perfect complexes,
because $X$ has the resolution property; see \cite[Tags~0F87
and~0FDJ]{Stacks}. Tensor products and pullbacks of perfect complexes are
derived, and $\RHom$ denotes derived sheaf Hom. The dual of a class $E$ is
written $E^\vee$; on Chern characters, $\ch_k(E^\vee)=(-1)^k\ch_k(E)$.

For a scheme $X$ we write
\[
 \pi:S\times X\longrightarrow X,\qquad \pr_S:S\times X\longrightarrow S
\]
for the projections; which $X$ is meant is always clear from the space on
which the sheaves live. Classes on $S$ and on $X$ are pulled back to
$S\times X$ without further notation. For perfect complexes $E,F$ on
$S\times X$ the complex $R\pi_*\RHom(E,F)$ is perfect, and its class in
$K^0(X)$ depends only on the classes of $E$ and $F$. This defines a
biadditive operation
\begin{equation}\label{eq:relative-ext-operation}
 K^0(S\times X)\times K^0(S\times X)\longrightarrow K^0(X),\qquad
 ([E],[F])\longmapsto[R\pi_*\RHom(E,F)],
\end{equation}
which we use freely for $K$-theory classes.

\subsection{The Quot scheme and its virtual class}
\label{subsec:quot}
Let $\beta\in H^2(S,\Z)$ be an effective class, and let $N\geq1$ and $n$ be
integers. A point of $\Qn=\Quot_S(\OO_S^{\oplus N},\beta,n)$ is a quotient
$\OO_S^{\oplus N}\twoheadrightarrow Q$ with $Q$ of rank zero, $c_1(Q)=\beta$
and $\chi(Q)=n$, up to isomorphism of quotients. Here $c_1(Q)$ is the class
of the support cycle: if $C_1,\ldots,C_r$ are the one-dimensional components
of the support of $Q$, with generic points $\eta_1,\ldots,\eta_r$, then
\begin{equation}\label{eq:support-cycle}
 c_1(Q)=\operatorname{PD}_S\Bigl(\sum_{i=1}^{r}
  \operatorname{length}_{\OO_{S,\eta_i}}(Q_{\eta_i})\,[C_i]\Bigr)
  \in H^2(S,\Z),
\end{equation}
where $\operatorname{PD}_S$ is Poincar\'e duality. If the support is
zero-dimensional, then $c_1(Q)=0$.

On $S\times\Qn$ there is a universal sequence
\begin{equation}\label{eq:universal-quotient}
 0\longrightarrow\K\longrightarrow\OO^{\oplus N}
 \longrightarrow\cQ\longrightarrow0 .
\end{equation}
The deformation theory of the quotient map gives a perfect obstruction
theory on $\Qn$ with virtual tangent complex and virtual dimension
\begin{equation}\label{eq:quot-tangent}
 T^{\vir}_{\Qn}=R\pi_*\RHom(\K,\cQ),\qquad \vd\Qn=Nn+\beta^2 ;
\end{equation}
see \cite[Sections~1.6 and~1.8]{OP}. It is two-term because, for a point
$0\to K\to\OO_S^{\oplus N}\to Q\to0$, Serre duality gives
$\Ext^2(K,Q)\cong\Hom(Q,K\otimes K_S)^\vee$, and this vanishes since $Q$ is
torsion and $K$ is torsion-free. The virtual dimension is the Euler
characteristic $\chi(K,Q)$, computed by Riemann--Roch. We write
$[\Qn]^{\vir}$ for the resulting virtual class.

\subsection{Descendent insertions and the series}

\begin{definition}[Descendent series]\label{def:descendents}
Fix an integer $\ell\geq0$, classes $\alpha_1,\ldots,\alpha_\ell\in K^0(S)$
and integers $d_1,\ldots,d_\ell\geq0$. For $\alpha\in K^0(S)$ put
\[
 \alpha^{[n]}=R\pi_*(\cQ\otimes\pr_S^*\alpha)\in K^0(\Qn),\qquad
 \tau_n=\prod_{a=1}^{\ell}\ch_{d_a}(\alpha_a^{[n]}),\qquad
 D=\sum_{a=1}^{\ell}d_a .
\]
(The complex $R\pi_*(\cQ\otimes\pr_S^*V)$ is perfect for every vector bundle
$V$ on $S$ because $\cQ$ is flat over $\Qn$ and $\pi$ is projective
\cite[Tag~0A1H]{Stacks}; the definition extends additively to $K^0(S)$.) The
\emph{descendent series} is
\begin{equation}\label{eq:main-series}
 Z_{S,N,\beta,\tau}(q)=\sum_{n\in\Z}q^n\intvir{\Qn}c(T^{\vir}_{\Qn})\,\tau_n .
\end{equation}
It is the series \eqref{eq:intro-series}.
\end{definition}

\begin{lemma}\label{lem:laurent}
The series \eqref{eq:main-series} has only finitely many negative powers of
$q$; that is, $Z_{S,N,\beta,\tau}\in\Q((q))$. More precisely, the coefficient
of $q^n$ equals
\begin{equation}\label{eq:dimension-matching}
 \intvir{\Qn}c_{Nn+\beta^2-D}(T^{\vir}_{\Qn})\,\tau_n
\end{equation}
and vanishes for $n<(D-\beta^2)/N$.
\end{lemma}
\begin{proof}
The class $\tau_n$ has complex degree $D$ and $[\Qn]^{\vir}$ has complex
dimension $Nn+\beta^2$. Only the component of $c(T^{\vir}_{\Qn})$ of
degree $Nn+\beta^2-D$ can contribute, and there is no such component when
this number is negative.
\end{proof}

\section{Root fields, residues, and specialization}\label{sec:fields}

This section is purely algebraic and can be read independently of the rest
of the paper. It provides the two statements used at the end of the proof:
\begin{itemize}[leftmargin=1.6em,itemsep=2pt]
\item Proposition~\ref{prop:residue-field}: generating series of residues of
the form $\Res f\cdot\prod_jR_z(x_j)^{m_j}$, summed over the exponents
$m_j$, lie in an explicit field $\RF$ of algebraic functions;
\item Theorem~\ref{thm:specialization}: an element of $\RF$ whose
$q$-coefficients are regular at $z=0$ becomes a rational function of $q$ at
$z=0$.
\end{itemize}

\subsection{The roots}
Fix distinct integers $w_1,\ldots,w_N$. Let $z$ and $q$ be independent
indeterminates and put
\begin{equation}\label{eq:root-polynomial}
 P_z(x)=\prod_{i=1}^N(x-w_iz),\qquad
 R_z(x)=\frac{P_z(x+1)}{P_z(x)},\qquad
 \Phi_{\xi,z}(x)=P_z(x)-\xi q\,P_z(x+1),
\end{equation}
where $\xi\in\Qb$ is a root of unity. The polynomial $\Phi_{\xi,z}$ has
degree $N$ in $x$, with leading coefficient $1-\xi q$. At $q=0$ it is
$P_z(x)$, with the $N$ distinct roots $w_iz$. The roots of $\Phi_{\xi,z}$
are obtained from these by the formal implicit function theorem, which we
record in the form needed.

\begin{lemma}[Lifting a simple root]\label{lem:simple-root-lift}
Let $\mathbb{k}$ be a field and let $A(v,x)\in\mathbb{k}[[v]][x]$. If
$a\in\mathbb{k}$ satisfies $A(0,a)=0$ and $\partial_xA(0,a)\neq0$, then there
is a unique $b(v)\in\mathbb{k}[[v]]$ with $b(0)=a$ and $A(v,b(v))=0$.
\end{lemma}
\begin{proof}
Suppose that $b^{(r)}(v)=a+\sum_{j=1}^{r}b_jv^j$ satisfies
$A(v,b^{(r)}(v))\equiv0$ modulo $v^{r+1}$. For $c\in\mathbb{k}$,
\[
 A\bigl(v,b^{(r)}(v)+cv^{r+1}\bigr)\equiv
 A\bigl(v,b^{(r)}(v)\bigr)+cv^{r+1}\,\partial_xA(0,a)\pmod{v^{r+2}} .
\]
Since $\partial_xA(0,a)\neq0$, exactly one value of $c$ makes the
coefficient of $v^{r+1}$ vanish. Starting from $b^{(0)}=a$, this constructs
the coefficients of $b$ one at a time and shows that they are forced.
\end{proof}

We apply the lemma with $\mathbb{k}=\Qb(z)$, $v=q$, $A=\Phi_{\xi,z}$ and
$a=w_iz$. The hypothesis holds because
\[
 P_z'(w_iz)=z^{N-1}\prod_{j\neq i}(w_i-w_j)\neq0 .
\]
Hence for each $i\in\{1,\ldots,N\}$ and each root of unity $\xi$ there is a
unique series
\begin{equation}\label{eq:lifted-roots}
 b_{i,\xi}(q)\in\Qb(z)[[q]],\qquad b_{i,\xi}(0)=w_iz,\qquad
 P_z\bigl(b_{i,\xi}(q)\bigr)=\xi q\,P_z\bigl(b_{i,\xi}(q)+1\bigr) .
\end{equation}
For fixed $\xi$ these $N$ series have distinct constant terms, so they are
all the roots of $\Phi_{\xi,z}$. Note that $b_{i,\xi}(q)=b_{i,1}(\xi q)$.

\begin{definition}[The root field]\label{def:root-field}
$\RF$ is the subfield of $\Qb(z)((q))$ generated over $\Qb(q,z)$ by all the
series $b_{i,\xi}(q)$ of \eqref{eq:lifted-roots}:
\[
 \RF=\Qb(q,z)\bigl(b_{i,\xi}(q)\ :\ 1\leq i\leq N,\ \xi\in\Qb
 \text{ a root of unity}\bigr).
\]
Every element of $\RF$ involves only finitely many generators. It therefore
lies in a subfield
\begin{equation}\label{eq:finite-root-extension}
 L=\Qb(q,z)\bigl(b_{i,\xi_j}(q)\ :\ 1\leq i\leq N,\ 1\leq j\leq r\bigr)
\end{equation}
for finitely many roots of unity $\xi_1,\ldots,\xi_r$, and $L$ is a finite
extension of $\Qb(q,z)$.
\end{definition}

\begin{example}\label{ex:N1-root}
For $N=1$ and $w_1=0$ we have $P_z(x)=x$, and $x=\xi q(x+1)$ gives
$b_{1,\xi}=\xi q/(1-\xi q)$. Thus $\mathscr F_1=\Qb(q,z)$: for $N=1$ no
irrationality occurs. For $N\geq2$ the roots are genuinely algebraic over
$\Qb(q,z)$.
\end{example}

We shall need series with coefficients in a cohomology ring. Let $\mathsf H$
be a finite-dimensional commutative $\Q$-algebra with basis
$e_1,\ldots,e_d$. Every $U\in\mathsf H\otimes_{\Q}\Qb(z)((q))$ has a unique
expression
\begin{equation}\label{eq:cohomology-field-expansion}
 U=\sum_{\nu=1}^{d}e_\nu\otimes f_\nu,\qquad f_\nu\in\Qb(z)((q)),
\end{equation}
and we say that $U\in\mathsf H\otimes_{\Q}\RF$ if every $f_\nu$ lies in
$\RF$. This does not depend on the basis.

\begin{lemma}[Operations preserving $\RF$]\label{lem:field-operations}
\begin{enumerate}[label=\textup{(\alph*)},leftmargin=2.2em]
\item The field $\RF$ is stable under the derivation $q\,d/dq$.
\item Let $\mathsf H=\bigoplus_{j\geq0}\mathsf H^j$ be a finite-dimensional
commutative graded $\Q$-algebra with $\mathsf H^0=\Q$, and let
$\mathsf H[\eps]=\mathsf H[\epsilon_1,\ldots,\epsilon_\ell]/
(\epsilon_1^2,\ldots,\epsilon_\ell^2)$. An element
$U\in\mathsf H[\eps]\otimes_{\Q}\RF$ whose scalar component
$U_0\in\RF$ $($its image under
$\mathsf H[\eps]\to\mathsf H^0=\Q$, $\epsilon_a\mapsto0)$ is nonzero is
invertible in $\mathsf H[\eps]\otimes_{\Q}\RF$.
\end{enumerate}
\end{lemma}
\begin{proof}
(a) It suffices to show that $d/dq$ maps each generator
$b=b_{i,\xi}$ into the field $L$ generated by it. Differentiating
\eqref{eq:lifted-roots} gives
\[
 \frac{db}{dq}=\frac{\xi\,P_z(b+1)}{P_z'(b)-\xi q\,P_z'(b+1)} .
\]
The denominator is nonzero because its value at $q=0$ is
$P_z'(w_iz)\neq0$.

(b) Let $\mathfrak a\subset\mathsf H[\eps]$ be the ideal generated by
$\bigoplus_{j\geq1}\mathsf H^j$ and $\epsilon_1,\ldots,\epsilon_\ell$. If
$\mathsf H^j=0$ for $j\geq c$, then $\mathfrak a^{c+\ell}=0$: a product of
$c+\ell$ generators contains $c$ factors of positive degree or repeats some
$\epsilon_a$. Write $U=U_0+V$ with $V\in\mathfrak a\otimes\RF$. Then
$U^{-1}=U_0^{-1}\sum_{j=0}^{c+\ell-1}(-U_0^{-1}V)^j$.
\end{proof}

\subsection{Regular specialization}

\begin{theorem}[Regular specialization]\label{thm:specialization}
Let $F\in\RF$, with expansion
\[
 F=\sum_{n\geq n_0}a_n(z)\,q^n\qquad\text{in }\Qb(z)((q)).
\]
Assume that every coefficient $a_n(z)$ lies in $\Q(z)$ and is regular at
$z=0$. Then $\sum_{n\geq n_0}a_n(0)\,q^n$ is the Laurent expansion of an
element of $\Q(q)$.
\end{theorem}

The idea is as follows. An element of $\RF$ is a rational expression in
$q$, $z$ and the roots $b_{i,\xi}$. At $z=0$ the equation defining
$b_{i,\xi}$ becomes $x^N=\xi q(x+1)^N$, with solutions $x=a/(1-a)$ for
$a^N=\xi q$. These are rational functions of $t=q^{1/N}$, so the
specialization of $F$ ought to be a rational function of $t$; being a series
in $q=t^N$, it is then a rational function of $q$. The difficulty is that
``setting $z=0$'' in a rational expression in the roots is not a priori
meaningful: the roots $b_{i,\xi}\in\Qb(z)[[q]]$ themselves have
coefficients with poles at $z=0$, and the expression may take the form
$0/0$. The proof avoids substitution altogether and instead transports $F$,
by a field embedding, into a field in which the expansion in $z$ comes
first.

\begin{proof}
Multiplying $F$ by $q^{-n_0}$ we may assume $n_0=0$.

\smallskip\noindent
\emph{Step 1: an embedding of $\Qb(q,z,F)$.}
Let $\mathcal R\subset\Qb(z)$ be the local ring of rational functions that
are regular at $z=0$. Taylor expansion at $z=0$ is an injective ring
homomorphism $\mathcal R\to\Qb[[z]]$. Applying it to each coefficient gives
an injective ring homomorphism
\[
 \mathcal R[[q]]\longrightarrow\Qb[[z]][[q]]=\Qb[[q,z]]\subset\Qb((q))((z)) .
\]
By hypothesis $F\in\mathcal R[[q]]$, hence
$\Qb[q,z,F]\subset\mathcal R[[q]]$. Restricting and passing to fraction
fields, we obtain a field embedding
\begin{equation}\label{eq:field-reordering}
 \sigma:\Qb(q,z,F)\hookrightarrow\Qb((q))((z))
\end{equation}
that fixes $q$ and $z$ and sends $F$ to the double series
\[
 \Phi=\sum_{n,r\geq0}a_{nr}\,q^nz^r,\qquad\text{where }
 a_n(z)=\sum_{r\geq0}a_{nr}z^r .
\]
Note that $\Phi\in\Q[[q]][[z]]$ and that its coefficient of $z^0$ is the
series $\sum_na_n(0)q^n$ we are interested in.

\smallskip\noindent
\emph{Step 2: where the roots go.}
Choose roots of unity $\xi_1,\ldots,\xi_r$ such that $F$ lies in the field
$L$ of \eqref{eq:finite-root-extension}. Let $t$ be an $N$th root of $q$,
put $\mathbb K=\Qb(t)$, and let $\Omega$ be an algebraic closure of
$\Qb((t))((z))$; this field contains $\Qb((q))((z))$, with $q=t^N$. Since
$L$ is algebraic over $\Qb(q,z,F)$, the embedding $\sigma$ extends to an
embedding $\widetilde\sigma:L\hookrightarrow\Omega$. For each generator
$b=b_{i,\xi}$ of $L$, the element $\widetilde\sigma(b)$ is a root in
$\Omega$ of the polynomial $\Phi_{\xi,z}$, whose coefficients lie in
$\Qb[q,z]$ and are fixed by $\widetilde\sigma$.

We claim that all roots of $\Phi_{\xi,z}$ in $\Omega$ lie in
$\mathbb K[[z]]$. Regard $\Phi_{\xi,z}$ as a polynomial in $x$ with
coefficients in $\mathbb K[z]$. At $z=0$ it is
\begin{equation}\label{eq:specialized-polynomial}
 x^N-\xi t^N(x+1)^N .
\end{equation}
Choose $\vartheta\in\Qb$ with $\vartheta^N=\xi$. For each $N$th root of unity
$\zeta$, put $a=\zeta\vartheta t\in\mathbb K$; then
$x=a/(1-a)\in\mathbb K$ is a root of \eqref{eq:specialized-polynomial}. This
gives $N$ distinct roots, hence all roots, and each is simple.
Lemma~\ref{lem:simple-root-lift}, applied with the parameter $v=z$ over the
field $\mathbb K$, lifts each of them to a root of $\Phi_{\xi,z}$ in
$\mathbb K[[z]]$. These $N$ lifts are distinct and $\Phi_{\xi,z}$ has degree
$N$, so there are no other roots. This proves the claim.

Consequently $\widetilde\sigma$ maps every generator of $L$ into
$\mathbb K[[z]]$, and therefore
$\widetilde\sigma(L)\subset\mathbb K((z))$.

\smallskip\noindent
\emph{Step 3: the constant coefficient is a rational function of $t$.}
We now know that $\Phi=\widetilde\sigma(F)$ lies in $\mathbb K((z))$. An
element of $\Qb((t))((z))$ lies in $\mathbb K((z))$ exactly when all its
$z$-coefficients lie in $\mathbb K$. Hence
\[
 f:=\sum_{n\geq0}a_n(0)\,q^n\ \in\ \Qb(t)\cap\Q[[t^N]] .
\]

\smallskip\noindent
\emph{Step 4: descent to $\Q(q)$.}
Since the expansion of $f\in\Qb(t)$ only involves powers of $t^N$, we have
$f(\zeta t)=f(t)$ for every $N$th root of unity $\zeta$. Write $f=U/V$ with
$U,V\in\Qb[t]$. The polynomial $V^\ast(t)=\prod_{\zeta^N=1}V(\zeta t)$ is
invariant under $t\mapsto\zeta t$, so $V^\ast\in\Qb[t^N]$; likewise
$fV^\ast$ is an invariant polynomial, so $fV^\ast\in\Qb[t^N]$. Hence
$f\in\Qb(q)$. Finally write $f=A/B$ with $A,B\in\Qb[q]$ coprime and $B$
monic. Every automorphism of $\Qb$ over $\Q$ fixes the Laurent coefficients
of $f$, hence fixes $f$, hence fixes $A$ and $B$ by uniqueness of this
representation. Therefore $A,B\in\Q[q]$.
\end{proof}

The proof also tells us how to \emph{compute} the specialization when $F$
is given explicitly. This will be used in Section~\ref{sec:checks}.

\begin{proposition}[Computing the specialization]\label{prop:compute-specialization}
In the situation of Theorem~\ref{thm:specialization}, suppose that
\[
 F=E\bigl(b_{1,1}(q),\ldots,b_{N,1}(q)\bigr)
\]
for a rational function $E=E_1/E_2$ with
$E_1,E_2\in\Qb[q,z][x_1,\ldots,x_N]$. Let
$\rho_1,\ldots,\rho_N\in\Qb(t)$ be the roots of $x^N=t^N(x+1)^N$, that is,
$\rho=\zeta t/(1-\zeta t)$ for the $N$th roots of unity $\zeta$. Assume
that, for every permutation $\varpi$ of $\{1,\ldots,N\}$,
\[
 E_2\bigl(\rho_{\varpi(1)},\ldots,\rho_{\varpi(N)}\bigr)\big|_{z=0}\neq0,
\]
and that the value
$E_0=E\bigl(\rho_{\varpi(1)},\ldots,\rho_{\varpi(N)}\bigr)|_{z=0}\in\Qb(t)$
does not depend on $\varpi$. Then
\[
 \sum_{n\geq n_0}a_n(0)\,q^n=E_0\qquad(q=t^N).
\]
\end{proposition}
\begin{proof}
We use the notation of the preceding proof, with $\xi_1=1$. The elements
$\beta_j=\widetilde\sigma(b_{j,1})\in\mathbb K[[z]]$ are $N$ distinct roots
of $\Phi_{1,z}$. The $N$ roots of $\Phi_{1,z}$ in $\mathbb K[[z]]$ have the
distinct constant terms $\rho_1,\ldots,\rho_N$, so
$\beta_j(0)=\rho_{\varpi(j)}$ for a permutation $\varpi$. Since
$\widetilde\sigma$ fixes $q$ and $z$,
\[
 \Phi=\widetilde\sigma(F)=
 \frac{E_1(\beta_1,\ldots,\beta_N)}{E_2(\beta_1,\ldots,\beta_N)} .
\]
The denominator is an element of $\mathbb K[[z]]$ whose constant term is
nonzero by hypothesis, hence a unit. So the right-hand side lies in
$\mathbb K[[z]]$ and has constant term $E_0$. By Step~1 of the preceding
proof, the constant term of $\Phi$ is $\sum_na_n(0)q^n$.
\end{proof}

\begin{remark}\label{rem:two-completions}
The regularity hypothesis in Theorem~\ref{thm:specialization} concerns $F$
itself, not the roots. For instance, for $N=2$ the root $b_{1,1}$ has
$q$-coefficients with poles at $z=0$ and no specialization in $\Q((q))$,
while $b_{1,1}b_{2,1}$ is a rational function of $q$ and $z$.
\end{remark}

\subsection{Residue series}
Let $\mathbb{k}=\Qb(z)$. Fix a subset $I\subset\{1,\ldots,N\}$, variables
$x_j$ for $j\in I$, and put $u_j=x_j-w_jz$.

\begin{definition}[Admissible functions and their residues]
\label{def:local-residues}
A rational function $f\in\mathbb{k}(x_j:j\in I)$ is \emph{admissible} if it
can be written as
\[
 f=\frac{f_1}{f_2}\prod_{j\in I}u_j^{-c_j},\qquad
 f_1,f_2\in\mathbb{k}[x_j:j\in I],\quad f_2(w_jz:j\in I)\neq0,\quad
 c_j\in\Z .
\]
In other words, near the point $x_j=w_jz$ the only poles of $f$ are along
the coordinate hyperplanes $u_j=0$. Such an $f$ has a unique expansion as a
Laurent series in the $u_j$ with exponents bounded below, obtained by
expanding $1/f_2$ as a power series. Its \emph{residue}
$\Res_{x_j=w_jz,\,j\in I}f\prod_{j\in I}dx_j$ is the coefficient of
$\prod_{j\in I}u_j^{-1}$. The same definitions apply to functions with
coefficients in $\mathsf H\otimes_{\Q}\mathbb{k}$ for a finite-dimensional
commutative $\Q$-algebra $\mathsf H$.
\end{definition}

The basic observation is elementary.

\begin{lemma}[Summing residues in one variable]\label{lem:positive-part}
Let $\mathbb{k}$ be a field and
$\widetilde f(t)=\sum_{a\geq a_0}c_at^a\in\mathbb{k}((t))$. Then
\begin{equation}\label{eq:positive-part}
 \sum_{m\in\Z}q^m\Res_{t=0}\widetilde f(t)\,t^{-m}\,dt=q\,\widetilde f(q),
 \qquad
 \sum_{m\geq0}q^m\Res_{t=0}\widetilde f(t)\,t^{-m}\,dt
 =q\,\widetilde f(q)-\sum_{a=a_0}^{-2}c_a\,q^{a+1}.
\end{equation}
\end{lemma}
\begin{proof}
The residue of $\widetilde f(t)t^{-m}dt$ is $c_{m-1}$.
\end{proof}

\begin{proposition}[Residue series lie in the root field]\label{prop:residue-field}
Let $f$ be an admissible function in the variables $x_j$, $j\in I$, and let
$p\in\mathbb{k}[m_j:j\in I]$ be a polynomial. For each $j\in I$ fix integers
$b_j$, $d_j\geq1$ and $r_j$, and let
\[
 \Lambda=\prod_{j\in I}\{m_j\in\Z:\ m_j\geq b_j,\ m_j\equiv r_j\ (\mathrm{mod}\ d_j)\}.
\]
Then the series
\begin{equation}\label{eq:residue-series}
 \sum_{\mm\in\Lambda}q^{\sum_{j\in I}m_j}\,p(\mm)
 \Res_{x_j=w_jz,\,j\in I}\Bigl(f\prod_{j\in I}R_z(x_j)^{m_j}\prod_{j\in I}dx_j\Bigr)
\end{equation}
is a well-defined element of $\mathbb{k}((q))$ and belongs to $\RF$. If $f$
and $p$ have coefficients in $\mathsf H\otimes_{\Q}\mathbb{k}$, the series
belongs to $\mathsf H\otimes_{\Q}\RF$.
\end{proposition}
\begin{proof}
The last assertion follows from the first by expanding in a basis of
$\mathsf H$, so we take scalar coefficients. The series is well defined
because, for a fixed power of $q$, only finitely many
$\mm\in\Lambda$ contribute.

\smallskip\noindent
\emph{Step 1: new coordinates.}
Put $t_j=1/R_z(x_j)=P_z(x_j)/P_z(x_j+1)$. Since $P_z$ has a simple zero at
$w_jz$ and $P_z(w_jz+1)=\prod_i(1+(w_j-w_i)z)\neq0$, the function $t_j$ is
$u_j$ times a function that is regular and nonzero at $u_j=0$. So $t_j$ is a
local coordinate at $x_j=w_jz$, with inverse series $x_j=g_j(t_j)$,
$g_j\in\mathbb{k}[[t]]$, characterized by
\[
 g_j(0)=w_jz,\qquad P_z\bigl(g_j(t)\bigr)=t\,P_z\bigl(g_j(t)+1\bigr).
\]
Comparing with \eqref{eq:lifted-roots} and using the uniqueness in
Lemma~\ref{lem:simple-root-lift},
\begin{equation}\label{eq:g-is-root}
 g_j(\xi q)=b_{j,\xi}(q)\qquad\text{for every root of unity }\xi .
\end{equation}
As in Lemma~\ref{lem:field-operations}(a), the derivative
\begin{equation}\label{eq:inverse-root-derivative}
 g_j'(t)=\frac{P_z(g_j(t)+1)}{P_z'(g_j(t))-t\,P_z'(g_j(t)+1)}
\end{equation}
is a rational expression in $t$ and $g_j(t)$ whose denominator is nonzero at
$t=0$. Put
\[
 \widetilde f(\boldsymbol t)=f\bigl(g_j(t_j):j\in I\bigr)\prod_{j\in I}g_j'(t_j).
\]
Because $g_j(t_j)-w_jz$ is $t_j$ times a unit, $\widetilde f$ is a Laurent
series in the $t_j$ with exponents bounded below. Residues are invariant
under such changes of coordinates, one variable at a time, so
\[
 \Res_{x_j=w_jz}\Bigl(f\prod_jR_z(x_j)^{m_j}\prod_jdx_j\Bigr)
 =\Res_{\boldsymbol t=0}\Bigl(\widetilde f(\boldsymbol t)
   \prod_jt_j^{-m_j}\prod_jdt_j\Bigr).
\]

\smallskip\noindent
\emph{Step 2: independent variables, all exponents, $p=1$.}
Introduce a variable $q_j$ for each $j\in I$. By
Lemma~\ref{lem:positive-part}, applied in each variable,
\begin{equation}\label{eq:all-exponents}
 G(\boldsymbol q):=\sum_{\mm\in\Z^I}\ \prod_jq_j^{m_j}
 \Res_{\boldsymbol t=0}\Bigl(\widetilde f(\boldsymbol t)
   \prod_jt_j^{-m_j}\prod_jdt_j\Bigr)
 =\Bigl(\prod_jq_j\Bigr)\,\widetilde f(\boldsymbol q).
\end{equation}
The right-hand side is a rational expression in the $q_j$ and the series
$g_j(q_j)$.

\smallskip\noindent
\emph{Step 3: lower bounds.}
Restricting the sum in \eqref{eq:all-exponents} to $m_j\geq b_j$ removes
finitely many ``slices'' $m_j=c$. The slice $m_j=c$ of $G$ is $q_j^c$ times
the coefficient of $q_j^c$ in $G$, and that coefficient is a Taylor
coefficient at $q_j=0$ of $q_j^{e}G$ for a suitable integer $e$. It is
obtained by differentiating finitely often with respect to $q_j$, which by
\eqref{eq:inverse-root-derivative} keeps the expression rational, and then
setting $q_j=0$, where $g_j(0)=w_jz\in\mathbb{k}$. The result is a rational
expression in the remaining $q_i$ and $g_i(q_i)$.

\smallskip\noindent
\emph{Step 4: the polynomial $p$ and the congruences.}
Multiplying the coefficient of $\prod_jq_j^{m_j}$ by $p(\mm)$ is the same as
applying the differential operator $p(q_j\partial_{q_j}:j\in I)$, which
preserves rational expressions in $q_j$ and $g_j(q_j)$ by
\eqref{eq:inverse-root-derivative}. The terms with
$m_j\equiv r_j$ modulo $d_j$ are selected by
\[
 \sum_{m\equiv r\,(d)}c_m\,q_j^m
 =\frac1d\sum_{\xi^d=1}\xi^{-r}\sum_mc_m(\xi q_j)^m ,
\]
which replaces $g_j(q_j)$ by the finitely many series $g_j(\xi q_j)$.

\smallskip\noindent
\emph{Step 5: the diagonal.}
After Steps 2--4 we have a series $G_\Lambda(\boldsymbol q)$, with
exponents bounded below, and an identity $E_2\cdot G_\Lambda=E_1$ in which
$E_1,E_2$ are polynomials in the $q_j$ and the series $g_j(\xi q_j)$. All
denominators that occurred are of two kinds: the denominator of $f$,
evaluated at $x_j=g_j(q_j)$ or $g_j(\xi q_j)$, which is a monomial in the
$q_j$ times a power series with nonzero constant term; and the
denominators in \eqref{eq:inverse-root-derivative}, which are power series
with nonzero constant term. Hence $E_2$ remains nonzero after the
substitution $q_j=q$ for all $j$. This substitution is a ring homomorphism
on Laurent series with exponents bounded below, and it sends
$G_\Lambda$ to the series \eqref{eq:residue-series}. By
\eqref{eq:g-is-root} it sends $E_1$ and $E_2$ to polynomials in $q$ and the
$b_{j,\xi}(q)$. Therefore \eqref{eq:residue-series} equals
$E_1/E_2\in\RF$.
\end{proof}

\section{Localization}\label{sec:localization}

In this section we introduce the torus action, describe the fixed loci and
their virtual classes, and define the equivariant coefficients $J_n(z)$. The
outcome is formula \eqref{eq:proper-localization}, together with
Lemma~\ref{lem:cycle-resolution}, which replaces the virtual class of the
divisor part of a fixed locus by smooth projective varieties.

\subsection{The torus action and weight notation}
Let $\TT=\C^*$ act trivially on $S$ and on $\OO_S^{\oplus N}$ by
\begin{equation}\label{eq:torus-action-on-bundle}
 \theta\cdot(s_1,\ldots,s_N)=(\theta^{w_1}s_1,\ldots,\theta^{w_N}s_N),
 \qquad\theta\in\TT,
\end{equation}
where $w_1,\ldots,w_N$ are distinct integers. This induces an action on
$\Qn$, for which the universal sequence \eqref{eq:universal-quotient}, the
obstruction theory, and the classes $\alpha^{[n]}$ are equivariant. We write
$\tau_n^{\TT}$ for the equivariant lift of $\tau_n$. For $w\in\Z$ let $\C_w$
be the one-dimensional representation with character
$\theta\mapsto\theta^w$, write $E[w]=E\otimes\C_w$, and put
$z=c_1^{\TT}(\C_1)$, so that $H^*_{\TT}(\mathrm{pt})=\Q[z]$ and
$c_1^{\TT}(\C_w)=wz$.

Let $E$ be a perfect complex of rank $r$ on a space with trivial
$\TT$-action, let $w\neq0$ be an integer, and write
$c_t(E)=\sum_kc_k(E)t^k$ for the Chern polynomial. If $E$ is a vector bundle
with Chern roots $x_1,\ldots,x_r$, then $E[w]$ has equivariant Chern roots
$x_a+wz$, and therefore
\begin{equation}\label{eq:twisted-chern-euler}
 c_t^{\TT}(E[w])=(1+twz)^r\,c_{t/(1+twz)}(E),\qquad
 e_{\TT}(E[w])=(wz)^r\,c_{1/(wz)}(E).
\end{equation}
For a difference $E=E^+-E^-$ of vector bundles these classes are defined
multiplicatively, as quotients, and \eqref{eq:twisted-chern-euler} continues
to hold with $r=\rk E^+-\rk E^-$. Since positive-degree classes are
nilpotent, both expressions are finite sums; the second one is invertible
once $z$ is inverted. We put
\begin{equation}\label{eq:chern-euler-ratio}
 \MM_w(E)=\frac{c^{\TT}(E[w])}{e_{\TT}(E[w])}
 =\Bigl(\frac{1+wz}{wz}\Bigr)^{r}\,
  \frac{c_{1/(1+wz)}(E)}{c_{1/(wz)}(E)} ,
\end{equation}
a class with coefficients in $\Q(z)$. It is multiplicative:
$\MM_w(E+E')=\MM_w(E)\,\MM_w(E')$.

\subsection{Fixed loci}

\begin{proposition}[Fixed loci and their virtual classes]\label{prop:fixed-loci}
The fixed locus $\Qn^{\TT}$ is the disjoint union, over all decompositions
$\beta=\sum_{i=1}^N\beta_i$ into effective classes and $n=\sum_{i=1}^Nn_i$,
of the loci
\begin{equation}\label{eq:fixed-space}
 F_{\mm,\bbeta}=\prod_{i=1}^NS^{[m_i]}\times\prod_{i=1}^N\Hilb_{\beta_i}(S),
 \qquad m_i=n_i+\frac{\beta_i(\beta_i+K_S)}2 .
\end{equation}
On $S\times F_{\mm,\bbeta}$ let $\cZ_i$ be the universal point subscheme and
$L_i=\OO(\cD_i)$ the line bundle of the universal divisor of the $i$th
factors, and put
\begin{equation}\label{eq:fixed-classes}
 K_i=I_{\cZ_i}\otimes L_i^{-1},\qquad Q_i=\OO-K_i,\qquad
 W_{ij}=R\pi_*\RHom(K_i,Q_j),\qquad M_i=K_S\otimes L_i^{-1}.
\end{equation}
Then the restriction of the virtual tangent complex splits into fixed and
moving parts
\begin{equation}\label{eq:moving-complex}
 T^{\vir}_F=\sum_{i=1}^NW_{ii},\qquad
 N^{\vir}_F=\sum_{i\neq j}W_{ij}[w_j-w_i],\qquad
 \rk W_{ij}=n_j+\beta_i\cdot\beta_j ,
\end{equation}
and the virtual class of the fixed locus is
\begin{equation}\label{eq:rank-one-class}
 [F_{\mm,\bbeta}]^{\vir}=\prod_{i=1}^Ne\bigl((M_i^{[m_i]})^\vee\bigr)\cap
 \Bigl(\Bigl[\prod_{i=1}^NS^{[m_i]}\Bigr]\times
 \Bigl[\prod_{i=1}^N\Hilb_{\beta_i}(S)\Bigr]^{\vir}\Bigr),
\end{equation}
where $M_i^{[m_i]}=R\pi_*(\OO_{\cZ_i}\otimes M_i)$ is a vector bundle of rank
$m_i$ and $\Hilb_{\beta_i}(S)$ carries its standard virtual class of
dimension $\beta_i(\beta_i-K_S)/2$.
\end{proposition}
\begin{proof}
\emph{The fixed points.} By \eqref{eq:torus-action-on-bundle}, the torus
$\TT$ acts trivially on $S$ and acts on the $i$th direct summand of
$\OO_S^{\oplus N}=\bigoplus_{i=1}^N\OO_S$ with weight $w_i$, and the weights
$w_1,\ldots,w_N$ are distinct. Hence a $\TT$-invariant subsheaf of
$\OO_S^{\oplus N}$ is the direct sum of its intersections with the $N$
direct summands $\OO_S$. A fixed kernel is thus
$\bigoplus_iK_i[w_i]$ with $K_i\subset\OO_S$ of rank one. A rank-one
subsheaf of $\OO_S$ on a smooth surface has the form $I_{Z_i}(-D_i)$ for a
zero-dimensional subscheme $Z_i$ and an effective divisor $D_i$. In
families this is the identification
$\Quot_S(\OO_S,\beta_i,n_i)\cong S^{[m_i]}\times\Hilb_{\beta_i}(S)$ of
\cite[Lemma~1]{Lim}. The value of $m_i$ follows from
$\chi(\OO_S/I_{Z}(-D))=\operatorname{length}(Z)-D(D+K_S)/2$.

\emph{The tangent complex.} The restriction of \eqref{eq:quot-tangent} is
\[
 R\pi_*\RHom\Bigl(\bigoplus_iK_i[w_i],\bigoplus_jQ_j[w_j]\Bigr)
 =\bigoplus_{i,j}W_{ij}[w_j-w_i],
\]
and the summands of weight zero are those with $i=j$; this is
\cite[Section~2.2.1, equation~(6)]{Lim}. By Riemann--Roch, with
$\rk K_i=1$, $c_1(K_i)=-\beta_i$ and $\rk Q_j=0$,
\[
 \rk W_{ij}=\chi(K_i,Q_j)=\chi(Q_j)-c_1(K_i)\cdot c_1(Q_j)
 =n_j+\beta_i\cdot\beta_j .
\]

\emph{The virtual class.} The fixed part of the obstruction theory is the
product of the obstruction theories of the $N$ rank-one Quot schemes. For a
single factor, \cite[Proposition~1]{Lim} identifies its virtual class with
that of a nested Hilbert scheme, and \cite[Proposition~2]{Lim}, which rests
on \cite[Theorem~5]{GT}, gives
\[
 [S^{[m]}\times\Hilb_{\beta}(S)]^{\vir}
 =e\bigl(R\pi_*\RHom(\OO_{\cZ},L)\bigr)\cap
  \bigl([S^{[m]}]\times[\Hilb_\beta(S)]^{\vir}\bigr).
\]
No assumption on $p_g(S)$ or on the irregularity is made there. By
relative Serre duality,
\[
 R\pi_*\RHom(\OO_{\cZ},L)\cong\bigl(R\pi_*(\OO_{\cZ}\otimes K_S\otimes
 L^{-1})\bigr)^\vee[-2],
\]
and the complex in brackets is the vector bundle $M^{[m]}$, because $\cZ$
is finite and flat over the base. The shift does
not change the $K$-theory class, so the Euler class is
$e((M^{[m]})^\vee)$.
\end{proof}

\begin{lemma}[Finiteness]\label{lem:finite-decompositions}
Let $\mathcal E_{\beta,N}$ be the set of ordered $N$-tuples
$\bbeta=(\beta_1,\ldots,\beta_N)$ of effective classes (zero allowed) with
$\sum_i\beta_i=\beta$. Then $\mathcal E_{\beta,N}$ is finite. Put
\begin{equation}\label{eq:quot-lower-bound}
 b_{\beta,N}=\min_{\bbeta\in\mathcal E_{\beta,N}}
 \Bigl(-\sum_{i=1}^N\frac{\beta_i(\beta_i+K_S)}2\Bigr).
\end{equation}
Then $\Qn=\varnothing$ for $n<b_{\beta,N}$, and for every $n$ only finitely
many of the loci \eqref{eq:fixed-space} are nonempty.
\end{lemma}
\begin{proof}
Each $\beta_i$ satisfies $0\leq H\cdot\beta_i\leq H\cdot\beta$. Effective
one-cycles of bounded $H$-degree are parametrized by a Chow variety with
finitely many connected components, and the homology class is constant on
each component. So only finitely many classes occur.

At a point of \eqref{eq:fixed-space} all $m_i$ are nonnegative, hence
\begin{equation}\label{eq:length-sum-lower-bound}
 n=\sum_im_i-\sum_i\frac{\beta_i(\beta_i+K_S)}2\ \geq\ b_{\beta,N}.
\end{equation}
If $\Qn$ is nonempty, it contains a $\TT$-fixed point: for a closed point
$x$, the orbit map $\C^*\to\Qn$ extends to $\mathbb A^1$ because $\Qn$ is
projective, and the image of $0$ is fixed. This proves the second
assertion. For the third, \eqref{eq:length-sum-lower-bound} determines
$\sum_im_i$ once $n$ and $\bbeta$ are given.
\end{proof}

\subsection{The equivariant coefficients}
Set $v=Nn+\beta^2$ and recall $D=\sum_ad_a$. The total Chern class
$c^{\TT}(T^{\vir}_{\Qn})$ of a two-term complex is an infinite sum in
equivariant cohomology, so the equivariant analogue of the coefficient of
$q^n$ in \eqref{eq:main-series} must be defined with some care. We do this
degree by degree.

\begin{proposition}[Equivariant coefficients]\label{prop:equivariant-coefficient}
\begin{enumerate}[label=\textup{(\alph*)},leftmargin=2.2em]
\item For every integer $a\geq0$ there is a rational number $b_{n,a}$ with
\begin{equation}\label{eq:degreewise-pushforward}
 \intvir{\Qn}c_a^{\TT}(T^{\vir}_{\Qn})\,\tau_n^{\TT}=b_{n,a}\,z^{\,a+D-v},
\end{equation}
and $b_{n,a}=0$ if $a+D<v$. Moreover $b_{n,v-D}$ is the coefficient of $q^n$
in \eqref{eq:main-series} $($this coefficient is zero if $v<D)$.
\item For each nonempty fixed locus $F=F_{\mm,\bbeta}$ of
\eqref{eq:fixed-space} the expression
\begin{equation}\label{eq:proper-localization}
 \Lambda_F(z)=\intvir{F}\tau_n^{\TT}|_F\,\prod_{i=1}^Nc(W_{ii})
 \prod_{i\neq j}\MM_{w_j-w_i}(W_{ij})
\end{equation}
is an element of $\Q(z)$. The finite sum
\begin{equation}\label{eq:equivariant-coefficient}
 J_n(z)=\sum_F\Lambda_F(z)\ \in\ \Q(z)
\end{equation}
is regular at $z=0$, with Taylor expansion
$J_n(z)=\sum_{a\geq\max(0,v-D)}b_{n,a}\,z^{\,a+D-v}$.
\item In particular
\begin{equation}\label{eq:coefficient-specialization}
 J_n(0)=\intvir{\Qn}c(T^{\vir}_{\Qn})\,\tau_n .
\end{equation}
\end{enumerate}
\end{proposition}
\begin{proof}
(a) The equivariant virtual class has dimension $v$ and the integrand has
degree $a+D$, so the proper pushforward to a point lies in
$H^{2(a+D-v)}_{\TT}(\mathrm{pt})$. This group is zero if $a+D<v$ and is
spanned by $z^{a+D-v}$ otherwise. For $a=v-D$ the pushforward has degree
zero, and the restriction $H^0_{\TT}(\mathrm{pt})\to H^0(\mathrm{pt})$ is
an isomorphism compatible with pushforward; by
\eqref{eq:dimension-matching} this gives the last statement.

(b) We introduce a formal parameter $t$ that keeps track of the degree $a$.
For a fixed locus $F$ put
\[
 \Lambda_F(t,z)=\intvir{F}\tau_n^{\TT}|_F\,
 \frac{\prod_ic_t(W_{ii})\prod_{i\neq j}c_t^{\TT}(W_{ij}[w_j-w_i])}
      {\prod_{i\neq j}e_{\TT}(W_{ij}[w_j-w_i])}.
\]
By \eqref{eq:twisted-chern-euler} and the nilpotence of positive-degree
classes on $F$, this is a rational function of $t$ and $z$ whose
denominator is a product of powers of $z$ and of the polynomials
$1+t(w_j-w_i)z$. In particular it is regular at $t=0$, and it may be
evaluated at $t=1$, where it equals $\Lambda_F(z)$ of
\eqref{eq:proper-localization}; so $\Lambda_F(z)\in\Q(z)$.

By Proposition~\ref{prop:fixed-loci}, the coefficient of $t^a$ in the
Taylor expansion of $\Lambda_F(t,z)$ at $t=0$ is the contribution of $F$ to
the virtual localization formula \cite[Theorem~1]{GP} for the left-hand
side of \eqref{eq:degreewise-pushforward}. Summing over the finitely many
$F$ (Lemma~\ref{lem:finite-decompositions}) we get, with
$\Lambda(t,z)=\sum_F\Lambda_F(t,z)$,
\begin{equation}\label{eq:t-expansion}
 z^{\,v-D}\Lambda(t,z)=\sum_{a\geq0}b_{n,a}\,(tz)^a
 \qquad\text{as power series in $t$ over $\Q(z)$.}
\end{equation}
The right-hand side is unchanged by the substitution
$(t,z)\mapsto(t/\mu,\mu z)$ for a new indeterminate $\mu$. So
the two rational functions $z^{v-D}\Lambda(t,z)$ and
$(\mu z)^{v-D}\Lambda(t/\mu,\mu z)$ of $t$ over $\Q(z,\mu)$,
both regular at $t=0$, have the same Taylor expansion and hence coincide.
Setting $\mu=1/z$, which is permitted because the denominators
$1+t(w_j-w_i)$ do not vanish identically, we find
\[
 z^{\,v-D}\Lambda(t,z)=r_n(tz),\qquad r_n(u):=\Lambda(u,1)\in\Q(u),
\]
where $r_n$ is regular at $u=0$ with Taylor expansion $\sum_ab_{n,a}u^a$.
Now set $t=1$: $J_n(z)=\Lambda(1,z)=z^{D-v}r_n(z)$. Since $b_{n,a}=0$ for
$a<v-D$, the function $r_n$ vanishes to order at least $v-D$ at the origin
when $v>D$, so $J_n$ is regular at $z=0$ with the stated expansion.

(c) By (b), $J_n(0)=b_{n,v-D}$ if $v\geq D$ and $J_n(0)=0$ otherwise; apply
(a).
\end{proof}

\begin{remark}
$J_n(z)$ is in general not a polynomial: it may have poles at the points
$z=-1/(w_j-w_i)$. For example, for $S=\PP^2$, $N=2$, $(w_1,w_2)=(0,1)$,
$\beta$ the class of a line, $n=2$ and no insertions, one finds
$J_2(0)=-42$ and $J_2(1/10)=-468/11$; see Example~\ref{ex:new-case}.
\end{remark}

\subsection{Replacing the divisor Hilbert schemes by smooth bases}

\begin{lemma}[Smooth bases]\label{lem:cycle-resolution}
Fix $\bbeta\in\mathcal E_{\beta,N}$ and put
$\cD_{\bbeta}=\prod_{i=1}^N\Hilb_{\beta_i}(S)$. There are finitely many
smooth connected projective varieties $B_\lambda$, morphisms
$r_\lambda:B_\lambda\to\cD_{\bbeta}$ and rational numbers $a_\lambda$, all
independent of $\mm$, with the following property. Let $\Xi$ be a polynomial
with coefficients in $\Q(z)$ in Chern classes of perfect complexes on
$\prod_iS^{[m_i]}\times\cD_{\bbeta}$. Then
\begin{equation}\label{eq:divisor-cycle-evaluation}
 \int_{[\prod_iS^{[m_i]}]\times[\cD_{\bbeta}]^{\vir}}\Xi
 =\sum_\lambda a_\lambda\int_{\prod_iS^{[m_i]}\times B_\lambda}
 (\operatorname{id}\times r_\lambda)^*\Xi .
\end{equation}
\end{lemma}
\begin{proof}
Choose a cycle $\sum_\lambda a_\lambda[V_\lambda]$ representing
$[\cD_{\bbeta}]^{\vir}$ in the rational Chow group, with $V_\lambda$
integral closed subvarieties, and let $r_\lambda:B_\lambda\to V_\lambda
\subset\cD_{\bbeta}$ be resolutions of singularities, so that
$(r_\lambda)_*[B_\lambda]=[V_\lambda]$. Put
$f_\lambda=\operatorname{id}\times r_\lambda$. Then
$(f_\lambda)_*[\prod_iS^{[m_i]}\times B_\lambda]
=[\prod_iS^{[m_i]}]\times[V_\lambda]$, and the projection formula for Chern
classes \cite[Tag~02U9]{Stacks},
\begin{equation}\label{eq:characteristic-projection-formula}
 (f_\lambda)_*\bigl(f_\lambda^*\Xi\cap\eta\bigr)=\Xi\cap(f_\lambda)_*\eta ,
\end{equation}
gives \eqref{eq:divisor-cycle-evaluation}.
\end{proof}

We apply the lemma to the integrand of \eqref{eq:proper-localization}, after
substituting \eqref{eq:rank-one-class}. All complexes occurring there are
built from the universal ideals $I_{\cZ_i}$ and the line bundles $L_i$ by
the operation \eqref{eq:relative-ext-operation}, which commutes with base
change. Their pullbacks under $f_\lambda$ are therefore given by the same
formulas, with $L_i$ replaced by the line bundle
$(\operatorname{id}_S\times r_\lambda)^*L_i$ on $S\times B_\lambda$. This
leads to the following definition.

\section{The Hilbert series and universality relative to a base}
\label{sec:genus}

From now on the divisors are hidden in the data of a base $B$ and line
bundles $L_i$ on $S\times B$; the Hilbert schemes $\Hilb_{\beta_i}(S)$ will
not appear again until Section~\ref{sec:main-proof}.

\subsection{Definition of the Hilbert series}\label{subsec:hilbert-series}
Fix a smooth connected projective variety $B$ and line bundles
$L_1,\ldots,L_N$ on $S\times B$; we write $\LL=(L_1,\ldots,L_N)$. Fix also
the descendent data $\alpha_a$, $d_a$ $(1\leq a\leq\ell)$ of
Definition~\ref{def:descendents}. For $\mm=(m_1,\ldots,m_N)\in\Z_{\geq0}^N$
put $|\mm|=\sum_im_i$ and
\[
 X_{\mm}=\prod_{i=1}^NS^{[m_i]}\times B,\qquad p_{\mm}:X_{\mm}\longrightarrow B.
\]
On $S\times X_{\mm}$ we have the universal subschemes $\cZ_i$ and, exactly
as in \eqref{eq:fixed-classes},
\[
 K_i=I_{\cZ_i}\otimes L_i^{-1},\qquad Q_i=\OO-K_i,\qquad
 W_{ij}=R\pi_*\RHom(K_i,Q_j),\qquad M_i=K_S\otimes L_i^{-1}.
\]
To handle the descendent insertions we use formal variables
$\epsilon_1,\ldots,\epsilon_\ell$ with $\epsilon_a^2=0$ and the coefficient
ring
\[
 \HB=H^{\mathrm{even}}(B,\Q),\qquad
 \HB[\eps]=\HB[\epsilon_1,\ldots,\epsilon_\ell]/
           (\epsilon_1^2,\ldots,\epsilon_\ell^2).
\]
Put
\begin{equation}\label{eq:insertion-exponential}
 \Theta_{\mm}=\exp\Bigl(\sum_{a=1}^{\ell}\epsilon_a\,\mathsf D_{a,\mm}\Bigr)
 =\prod_{a=1}^\ell(1+\epsilon_a\mathsf D_{a,\mm}),\qquad
 \mathsf D_{a,\mm}=\ch^{\TT}_{d_a}\Bigl(R\pi_*\Bigl(\sum_{i=1}^NQ_i[w_i]
 \otimes\alpha_a\Bigr)\Bigr).
\end{equation}
The equivariant Chern character of a twisted class is the finite sum
\begin{equation}\label{eq:twisted-ch}
 \ch^{\TT}_d(E[w])=\sum_{v=0}^{d}\frac{(wz)^{d-v}}{(d-v)!}\,\ch_v(E).
\end{equation}
The coefficient of $\epsilon_1\cdots\epsilon_\ell$ in $\Theta_{\mm}$ is the
product $\prod_a\mathsf D_{a,\mm}$ of the descendent insertions. We work
with $\Theta_{\mm}$ rather than with this product for the following reason.
When $S$ is a disjoint union of two surfaces, each class $\mathsf D_{a,\mm}$
is the sum of the corresponding classes of the two surfaces, so
$\Theta_{\mm}$ is the product of the corresponding classes $\Theta$. This
multiplicativity is used in the proof of Lemma~\ref{lem:linear-log}.

\begin{definition}[Hilbert series]\label{def:hilbert-series}
\begin{equation}\label{eq:hilbert-series}
 \YY_{\LL}(q;z,\eps)=\sum_{\mm\in\Z_{\geq0}^N}q^{|\mm|}\,(p_{\mm})_*\Bigl(
 \prod_{i=1}^Ne\bigl((M_i^{[m_i]})^\vee\bigr)\prod_{i=1}^Nc(W_{ii})
 \prod_{i\neq j}\MM_{w_j-w_i}(W_{ij})\ \Theta_{\mm}\Bigr).
\end{equation}
It is an element of $\HB[\eps]\otimes_{\Q}\Q(z)[[q]]$. (The pushforward of an
even class along $p_{\mm}$ is even.)
\end{definition}

The Hilbert series is defined for arbitrary line bundles $L_i$, whether or
not they come from divisors. This freedom is essential in
Section~\ref{sec:twisting}. The link with the Quot scheme is the following
reformulation of Section~\ref{sec:localization}.

\begin{proposition}[From Quot schemes to Hilbert series]\label{prop:quot-to-hilbert}
For $\bbeta\in\mathcal E_{\beta,N}$ let $B_\lambda$, $r_\lambda$,
$a_\lambda$ be as in Lemma~\ref{lem:cycle-resolution}, and let
$\LL_\lambda$ be the tuple of pullbacks to $S\times B_\lambda$ of the
universal divisor line bundles. Then
\begin{equation}\label{eq:quot-to-hilbert}
 \sum_{n\in\Z}q^nJ_n(z)=\sum_{\bbeta\in\mathcal E_{\beta,N}}
 q^{-\sum_i\beta_i(\beta_i+K_S)/2}\ \sum_\lambda a_\lambda\int_{B_\lambda}
 \bigl[\epsilon_1\cdots\epsilon_\ell\bigr]\,\YY_{\LL_\lambda}(q;z,\eps),
\end{equation}
where $[\epsilon_1\cdots\epsilon_\ell]$ denotes the coefficient of that
monomial.
\end{proposition}
\begin{proof}
By \eqref{eq:equivariant-coefficient}, \eqref{eq:proper-localization} and
\eqref{eq:rank-one-class}, $J_n(z)$ is the sum over the fixed loci of the
integral over $[\prod_iS^{[m_i]}]\times[\cD_{\bbeta}]^{\vir}$ of
\[
 \tau_n^{\TT}|_F\prod_ie\bigl((M_i^{[m_i]})^\vee\bigr)\prod_ic(W_{ii})
 \prod_{i\neq j}\MM_{w_j-w_i}(W_{ij}).
\]
The universal quotient restricts to $\bigoplus_iQ_i[w_i]$ on a fixed
locus, so $\tau_n^{\TT}|_F=\prod_a\mathsf D_{a,\mm}$. Apply
Lemma~\ref{lem:cycle-resolution} and integrate over $B_\lambda$ after
pushing forward along $p_{\mm}$. Finally, on $F_{\mm,\bbeta}$ we have
$n=|\mm|-\sum_i\beta_i(\beta_i+K_S)/2$ by \eqref{eq:fixed-space}.
\end{proof}

\subsection{The constant term and the normalized series}
For $\mm=0$ we have $X_0=B$, $K_i=L_i^{-1}$ and
\begin{equation}\label{eq:Wzero}
 W_{ij}^0:=R\pi_*\RHom(L_i^{-1},\OO-L_j^{-1})=R\pi_*L_i-R\pi_*(L_iL_j^{-1}).
\end{equation}
The $q$-constant coefficient of the Hilbert series is therefore
\[
 \YY^0_{\LL}=\prod_ic(W^0_{ii})\prod_{i\neq j}\MM_{w_j-w_i}(W^0_{ij})\,\Theta_0
 \in\HB[\eps]\otimes\Q(z).
\]
Its scalar component is a product of integer powers of the numbers
$\bigl(1+(w_j-w_i)z\bigr)/\bigl((w_j-w_i)z\bigr)$, hence nonzero, so
$\YY^0_{\LL}$ is invertible. We define the \emph{normalized series}
\[
 \widehat\YY_{\LL}=(\YY^0_{\LL})^{-1}\,\YY_{\LL}=1+O(q).
\]
On $X_{\mm}$ the classes $W^0_{ij}$ are pulled back from $B$. Write
$\OO_i=\OO_{\cZ_i}$ and
\begin{align}
 \Delta W_{ij}:=W_{ij}-W^0_{ij}={}&-R\pi_*\RHom(\OO_i,L_i)
   +R\pi_*\RHom(\OO_i,L_iL_j^{-1})\notag\\
 &+R\pi_*(L_iL_j^{-1}\otimes\OO_j)-R\pi_*\RHom(\OO_i,L_iL_j^{-1}\otimes\OO_j),
 \label{eq:delta-W}
\end{align}
which follows by substituting $K_i=L_i^{-1}-L_i^{-1}\OO_i$. The four terms
have ranks $-m_i$, $m_i$, $m_j$ and $0$, so
\begin{equation}\label{eq:delta-rank}
 \rk\Delta W_{ij}=m_j\qquad\text{for all }i,j.
\end{equation}
For $i=j$, comparing with
$T_{S^{[m]}}=R\pi_*\OO-R\pi_*\RHom(I_{\cZ},I_{\cZ})$
\cite[Proposition~2.2]{EGL} and using relative Serre duality as in the proof
of Proposition~\ref{prop:fixed-loci}, one finds
\begin{equation}\label{eq:diagonal-hilbert-class}
 \Delta W_{ii}=T_{S^{[m_i]}}-\bigl(M_i^{[m_i]}\bigr)^\vee .
\end{equation}
By multiplicativity of Chern and Euler classes and the projection formula,
\begin{equation}\label{eq:normalized-series}
 \widehat\YY_{\LL}=\sum_{\mm}q^{|\mm|}(p_{\mm})_*\Bigl(
 \prod_ie\bigl((M_i^{[m_i]})^\vee\bigr)\prod_ic(\Delta W_{ii})
 \prod_{i\neq j}\MM_{w_j-w_i}(\Delta W_{ij})\ \widehat\Theta_{\mm}\Bigr),
\end{equation}
where $\widehat\Theta_{\mm}$ is defined like $\Theta_{\mm}$ but with
$Q_i$ replaced by $Q_i-(\OO-L_i^{-1})=\OO_i\otimes L_i^{-1}$.

\subsection{Universality}\label{subsec:universality}
We now show that the coefficients of $\widehat\YY_{\LL}$ depend on the
geometry only through a few characteristic classes of the family
$S\times B\to B$. The argument is the induction of Ellingsrud, G\"ottsche
and Lehn \cite{EGL}, extended to mixed terms as in
\cite[Section~2.2.3]{Lim}, and carried out over the base $B$. In this
subsection the surface $S$ is allowed to be disconnected.

\begin{definition}[$\kappa$-classes]\label{def:kappa}
Let $\pi:S\times B\to B$ be the projection. For a monomial $P$ in the
classes
\[
 c_1(L_1),\ldots,c_1(L_N),\qquad \ch_j(\alpha_a)\ (1\leq a\leq\ell,\
 j\geq0),\qquad c_1(T_S),\ c_2(T_S)
\]
on $S\times B$ put
\begin{equation}\label{eq:kappa-classes}
 \kappa_P=\pi_*(P)\in\HB .
\end{equation}
\end{definition}
\noindent
Note that $c_1(L_i)$ is a class on $S\times B$; all its K\"unneth components,
including the one in $H^1(S)\otimes H^1(B)$, enter the $\kappa$-classes. If
$S=S'\sqcup S''$, then $\kappa_P(S)=\kappa_P(S')+\kappa_P(S'')$.

\medskip
For $e\geq0$ put
\[
 X_{\mm,e}=\prod_{i=1}^NS^{[m_i]}\times S^e\times B,\qquad
 p:X_{\mm,e}\longrightarrow B,
\]
and denote the coordinates of $S^e$ by $y_1,\ldots,y_e$. We call the
homogeneous components of the following classes on $X_{\mm,e}$ \emph{basic
classes}:
\begin{enumerate}[label=\textup{(B\arabic*)},leftmargin=3em]
\item $\ch\bigl(I_i(y_a)\bigr)$, where $I_i(y_a)$ is the pullback of the
universal ideal $I_{\cZ_i}$ on $S\times S^{[m_i]}$ under the map
$(y_a,\pr_{S^{[m_i]}})$, for $1\leq i\leq N$ and $1\leq a\leq e$;
\item the pullbacks under $(y_a,\pr_B):X_{\mm,e}\to S\times B$ of the
classes listed in Definition~\ref{def:kappa}, for $1\leq a\leq e$;
\item $\ch(\OO_{\Delta_{ab}})$, where $\Delta_{ab}=\{y_a=y_b\}$, for
$1\leq a<b\leq e$.
\end{enumerate}
A \emph{basic polynomial} on $X_{\mm,e}$ is a polynomial in basic classes
with coefficients in $\Q(z)$.

\begin{lemma}[The integrand in terms of ideals]\label{lem:grr}
Let $\pi:S\times X_{\mm}=X_{\mm,1}\to X_{\mm}$ be the projection, write
$I_i=I_i(y_1)$ and $\td_S=\td(T_S)$, and pull all classes back to
$X_{\mm,1}$. Then
\begin{align}
 \ch(\Delta W_{ij})&=\pi_*\Bigl(\bigl[\ch(I_i)^\vee\bigl(1-\ch(I_j)
   e^{-c_1(L_j)}\bigr)-\bigl(1-e^{-c_1(L_j)}\bigr)\bigr]\,
   e^{c_1(L_i)}\,\td_S\Bigr),\label{eq:grr-W}\\
 \ch\bigl(M_i^{[m_i]}\bigr)&=\pi_*\Bigl(\bigl(1-\ch(I_i)\bigr)\,
   e^{-c_1(T_S)-c_1(L_i)}\,\td_S\Bigr),\label{eq:grr-M}\\
 \ch\bigl(R\pi_*(\OO_{\cZ_i}\otimes L_i^{-1}\otimes\alpha)\bigr)
  &=\pi_*\Bigl(\bigl(1-\ch(I_i)\bigr)\,e^{-c_1(L_i)}\,\ch(\alpha)\,
   \td_S\Bigr)\label{eq:grr-D}
\end{align}
for all $i,j$ $($including $i=j)$ and all $\alpha\in K^0(S)$. Here
$\ch(E)^\vee=\sum_k(-1)^k\ch_k(E)$.

Consequently, for every $\mm$, every subset $U\subseteq\{1,\ldots,\ell\}$
and every $b\geq0$, the component in $H^{2b}(B)\otimes\Q(z)$ of the
coefficient of $\prod_{a\in U}\epsilon_a$ in the $\mm$th summand of
\eqref{eq:normalized-series} is of the form
\begin{equation}\label{eq:basic-integral}
 p_*\bigl(P\cap[X_{\mm,e}]\bigr)
\end{equation}
for some $e\geq0$ and a basic polynomial $P$ involving only classes of
types \textup{(B1)} and \textup{(B2)}. The polynomial $P$ is given by a
formula that depends on $N$, the weights $w_i$, the $d_a$, $\mm$, $U$ and
$b$, but not on $S$, $B$, $\LL$ or the $\alpha_a$.
\end{lemma}
\begin{proof}
The three formulas are instances of the Grothendieck--Riemann--Roch
theorem for the smooth projective morphism $\pi$, applied to the classes
$K_i^\vee\otimes Q_j-L_i\otimes(\OO-L_j^{-1})$, $\ \OO_{\cZ_i}\otimes M_i$
and $\OO_{\cZ_i}\otimes L_i^{-1}\otimes\alpha$. One uses
$\ch(K_i)=\ch(I_i)e^{-c_1(L_i)}$, $[\OO_{\cZ_i}]=1-[I_i]$ and
$c_1(M_i)=-c_1(T_S)-c_1(L_i)$.

For the second statement, expand the integrand of
\eqref{eq:normalized-series}. By \eqref{eq:chern-euler-ratio} and
\eqref{eq:delta-rank},
\begin{equation}\label{eq:normalized-moving-factor}
 \MM_{w}(\Delta W_{ij})=\Bigl(\frac{1+wz}{wz}\Bigr)^{m_j}
 \frac{c_{1/(1+wz)}(\Delta W_{ij})}{c_{1/(wz)}(\Delta W_{ij})},
 \qquad w=w_j-w_i .
\end{equation}
The denominator is $1$ plus a class of positive degree, so its inverse is a
finite geometric series. Chern classes are polynomials with rational
coefficients in the components of the Chern character. By
\eqref{eq:twisted-ch} the descendent factors are polynomials in $z$ and in
components of the classes \eqref{eq:grr-D}. Since $p_{\mm}$ has relative
dimension $2|\mm|$, only the component of the integrand of complex degree
$2|\mm|+b$ contributes to $H^{2b}(B)$. It follows that the component in
$H^{2b}(B)\otimes\Q(z)$ of the coefficient of $\prod_{a\in U}\epsilon_a$ in
the $\mm$th summand of \eqref{eq:normalized-series} is the pushforward
under $p_{\mm}$ of a polynomial, with coefficients in $\Q(z)$, in the
components of the left-hand sides of \eqref{eq:grr-W}--\eqref{eq:grr-D}. Each of these
is $\pi_*$ of a basic polynomial on $X_{\mm,1}$. A product of $e$ such
classes is the pushforward, along $X_{\mm,e}\to X_{\mm}$, of the product of
the pullbacks of the $e$ basic polynomials under the $e$ projections
$X_{\mm,e}\to X_{\mm,1}$; this is base change together with the projection
formula.
\end{proof}

The next lemma is the induction step. It trades one point of a Hilbert
scheme for one extra surface coordinate.

\begin{lemma}[Incidence step]\label{lem:incidence-closure}
Let $j$ be an index with $m_j\geq1$ and let $\mm'$ be obtained from $\mm$
by lowering $m_j$ by one. For every basic polynomial $P$ on $X_{\mm,e}$
there is a basic polynomial $\widetilde P$ on $X_{\mm',e+1}$ such that
\begin{equation}\label{eq:incidence-pushforward-to-base}
 p_*\bigl(P\cap[X_{\mm,e}]\bigr)
 =p_*\bigl(\widetilde P\cap[X_{\mm',e+1}]\bigr)\qquad\text{in }\HB\otimes\Q(z).
\end{equation}
The polynomial $\widetilde P$ is obtained from $P$ by a rule that depends
only on $m_j$; see \eqref{eq:explicit-induction-operation}.
\end{lemma}
\begin{proof}
Write $r=m_j-1$. Let $S^{[r,r+1]}\subset S^{[r]}\times S^{[r+1]}$ be the
incidence variety of pairs $W\subset Z$, with its maps
\[
 \psi:S^{[r,r+1]}\to S^{[r+1]},\qquad\phi:S^{[r,r+1]}\to S^{[r]},\qquad
 \rho:S^{[r,r+1]}\to S,\qquad\sigma=(\phi,\rho),
\]
where $\rho$ sends a pair $W\subset Z$ to the support of the sheaf
$I_W/I_Z$. This sheaf has length one, so its support is a single point of
$S$: the unique point at which the subschemes $W$ and $Z$ differ. We use the
following facts from \cite[Sections~1 and~3]{EGL}. The variety $S^{[r,r+1]}$
has pure dimension $2r+2$, and $\psi$ is generically finite of degree $r+1$
onto $S^{[r+1]}$. (For connected $S$ the incidence variety is irreducible;
in general one applies this to the connected components of $S$.) Via
$\sigma$, the variety $S^{[r,r+1]}$ is the projectivization
$\PP(I_{\cW})\to S^{[r]}\times S$ of the universal ideal, and on
$S\times S^{[r,r+1]}$ there is an exact sequence
\begin{equation}\label{eq:incidence-exact-expanded}
 0\longrightarrow I_Z\longrightarrow I_W\longrightarrow
 \gamma_*\cL\longrightarrow0 ,
\end{equation}
where $I_Z,I_W$ are the pullbacks of the universal ideals, $\gamma$ is the
graph of $\rho$ and $\cL=\OO_{\PP(I_{\cW})}(1)$. Finally, with
$u=c_1(\cL)$,
\begin{equation}\label{eq:incidence-pushforwards-expanded}
 \psi_*[S^{[r,r+1]}]=(r+1)\,[S^{[r+1]}],\qquad
 \sigma_*\bigl(u^c\cap[S^{[r,r+1]}]\bigr)
 =(-1)^c\,c_c(\OO_{\cW})\cap[S^{[r]}\times S]
\end{equation}
for $c\geq0$; the second formula is \cite[Lemma~1.1]{EGL}.

Let $\widehat X$ be the product of $S^{[r,r+1]}$ with the remaining factors
of $X_{\mm,e}$, and let
\[
 \Psi:\widehat X\to X_{\mm,e},\qquad\Sigma:\widehat X\to X_{\mm',e+1}
\]
be induced by $\psi$ and $\sigma$; the new coordinate is $y_{e+1}=\rho$.
Both maps commute with the projections to $B$. We compare the pullbacks of
the basic classes. Each basic class is defined on $X_{\mm,e}$ and, by the
same formula, on $X_{\mm',e+1}$. For the classes of types (B2) and (B3),
and for $\ch(I_i(y_a))$ with $i\neq j$, the pullback under $\Psi$ of the
class on $X_{\mm,e}$ equals the pullback under $\Sigma$ of the
corresponding class on $X_{\mm',e+1}$, because both are pulled back from
factors on which $\Psi$ and $\Sigma$ agree. For $i=j$ we claim
that, in $K^0(\widehat X)$,
\begin{equation}\label{eq:incidence-coordinate-restriction}
 \Psi^*I_j(y_a)=\Sigma^*I_j(y_a)-\cL\otimes\Sigma^*\OO_{\Delta_{a,e+1}}
 \qquad(1\leq a\leq e).
\end{equation}
To see this, note first that the universal ideals are flat over the
Hilbert schemes, so the terms $I_Z$, $I_W$ in
\eqref{eq:incidence-exact-expanded} are also the derived pullbacks of the
universal ideals; and $\gamma_*\cL$ is perfect because $\gamma$ is a
section of a smooth morphism, hence a regular embedding. Now pull
\eqref{eq:incidence-exact-expanded} back along the graph of
$y_a:\widehat X\to S$. The derived pullback of $\gamma_*\OO$ along this
graph is the derived pullback of the structure sheaf of the diagonal of
$S\times S$ under $(y_a,\rho):\widehat X\to S\times S$, because the graph
of $\rho$ is the preimage of the diagonal and has the expected
codimension; and $(y_a,\rho)$ factors through $\Sigma$. This gives
\eqref{eq:incidence-coordinate-restriction}, and hence
\[
 \Psi^*\ch\bigl(I_j(y_a)\bigr)=\Sigma^*\ch\bigl(I_j(y_a)\bigr)
 -e^{u}\,\Sigma^*\ch\bigl(\OO_{\Delta_{a,e+1}}\bigr).
\]
Substituting into $P$ and expanding $e^u$ (a finite sum, as $u$ is
nilpotent) we obtain basic polynomials $G_c$ on $X_{\mm',e+1}$ with
\begin{equation}\label{eq:incidence-polynomial-pullback}
 \Psi^*P=\sum_{c\geq0}u^c\,\Sigma^*G_c .
\end{equation}
By \eqref{eq:incidence-pushforwards-expanded} and the projection formula,
\begin{align*}
 p_*\bigl(P\cap[X_{\mm,e}]\bigr)
 &=\frac1{r+1}\,p_*\Psi_*\bigl(\Psi^*P\cap[\widehat X]\bigr)
  =\frac1{r+1}\sum_cp_*\Sigma_*\bigl(u^c\,\Sigma^*G_c\cap[\widehat X]\bigr)\\
 &=\frac1{r+1}\sum_c(-1)^c\,p_*\bigl(G_c\,c_c(\OO_{\cW})\cap[X_{\mm',e+1}]\bigr).
\end{align*}
On $X_{\mm',e+1}$ we have $[\OO_{\cW}]=1-[I_j(y_{e+1})]$, so
$c_c(\OO_{\cW})$ is a polynomial in the components of
$\ch(I_j(y_{e+1}))$. Thus
\begin{equation}\label{eq:explicit-induction-operation}
 \widetilde P=\frac1{r+1}\sum_{c\geq0}(-1)^c\,G_c\,c_c\bigl(1-I_j(y_{e+1})\bigr)
\end{equation}
is a basic polynomial with the required property.
\end{proof}

When all the lengths are zero, the ideals $I_i(y_a)$ are trivial and only
classes of types (B2) and (B3) remain. They are handled by
Lemma~\ref{lem:diagonal-contraction} below.

The proof of that lemma uses the following notion from graph theory. Let
$\Gamma$ be a finite graph, possibly with multiple edges, and let
$V_1,\ldots,V_r$ be the vertex sets of its connected components. A
\emph{spanning forest} of $\Gamma$ is a subgraph $F$ of $\Gamma$ that
contains every vertex of $\Gamma$, has no cycles, and has the same
connected components $V_1,\ldots,V_r$ as $\Gamma$. Every finite graph has a
spanning forest, and a spanning forest has exactly $|V_c|-1$ edges with
both endpoints in $V_c$, for each $c$.

\begin{lemma}[Contraction of diagonals]\label{lem:diagonal-contraction}
Let $P$ be a basic polynomial on $X_{0,e}=S^e\times B$. Then
$p_*(P\cap[S^e\times B])$ is a polynomial in the classes $\kappa_{P'}$ of
Definition~\ref{def:kappa}, with coefficients in $\Q(z)$ that depend only on
$P$ $($and not on $S$, $B$, $\LL$ or the $\alpha_a)$.
\end{lemma}
\begin{proof}
Let $p_a=(y_a,\pr_B):S^e\times B\to S\times B$. The diagonal $\Delta_{ab}$
has normal bundle $p_a^*T_S|_{\Delta_{ab}}$, so by
Grothendieck--Riemann--Roch
$\ch(\OO_{\Delta_{ab}})=[\Delta_{ab}]\cdot p_a^*\td(T_S)^{-1}$. Hence $P$
is a combination of terms
\begin{equation}\label{eq:diagonal-monomial}
 M_\Gamma=\Bigl(\prod_{a=1}^{e}p_a^*\mu_a\Bigr)
 \prod_{\{a,b\}\in E(\Gamma)}[\Delta_{ab}],
\end{equation}
where each $\mu_a$ is a monomial in the classes of
Definition~\ref{def:kappa} and $\Gamma$ is a graph on the vertices
$1,\ldots,e$, possibly with multiple edges. Let $V_1,\ldots,V_r$ be the
vertex sets of the connected components of $\Gamma$, and let $e_c$ be the
number of edges of $\Gamma$ with both endpoints in $V_c$. Choose a
spanning forest $F$ of $\Gamma$; it has exactly $|V_c|-1$ edges with both
endpoints in $V_c$. The diagonals $\Delta_{ab}$ belonging
to the edges of $F$ intersect transversally, and their intersection is the
image of the embedding
$\delta_\Gamma:S^r\times B\hookrightarrow S^e\times B$ that identifies the
coordinates within each $V_c$. Each of the $e_c-|V_c|+1$ edges with both
endpoints in $V_c$ that do not belong to $F$ contributes, by the self-intersection formula
for the diagonal of $S\times S$, a factor $c_2(T_S)$ at the $c$th
coordinate. So
\begin{equation}\label{eq:diagonal-graph-intersection}
 M_\Gamma=(\delta_\Gamma)_*\Bigl(\prod_{c=1}^{r}\widetilde p_c^{\,*}
 \Bigl[\Bigl(\prod_{a\in V_c}\mu_a\Bigr)\,c_2(T_S)^{e_c-|V_c|+1}\Bigr]\Bigr),
\end{equation}
where $\widetilde p_c:S^r\times B\to S\times B$ is the $c$th projection.
Since $S^r\times B$ is the $r$-fold fiber product of $S\times B$ over $B$,
pushing forward to $B$ gives
\begin{equation}\label{eq:diagonal-component-pushforward}
 p_*M_\Gamma=\prod_{c=1}^{r}\pi_*\Bigl(\Bigl(\prod_{a\in V_c}\mu_a\Bigr)\,
 c_2(T_S)^{e_c-|V_c|+1}\Bigr),
\end{equation}
a product of $\kappa$-classes.
\end{proof}

\begin{proposition}[Universality relative to the base]
\label{prop:relative-universality}
Fix $N$, the weights $w_i$, and the integers $\ell$ and $d_a$. For every
$\mm\in\Z_{\geq0}^N$, $U\subseteq\{1,\ldots,\ell\}$ and $b\geq0$ there is a
polynomial $\mathfrak P_{\mm,U,b}$ with coefficients in $\Q(z)$, in
variables indexed by the monomials $P$ of Definition~\ref{def:kappa}, such
that for every smooth projective surface $S$ $($possibly disconnected$)$,
every smooth projective $B$, all line bundles $L_i$ on $S\times B$ and all
$\alpha_a\in K^0(S)$, the component in $H^{2b}(B)\otimes\Q(z)$ of the
coefficient of $\prod_{a\in U}\epsilon_a$ in the $\mm$th summand of
$\widehat\YY_{\LL}$ equals $\mathfrak P_{\mm,U,b}\bigl((\kappa_P)_P\bigr)$.
\end{proposition}
\begin{proof}
By Lemma~\ref{lem:grr}, the component in $H^{2b}(B)\otimes\Q(z)$ of the
coefficient of $\prod_{a\in U}\epsilon_a$ in the $\mm$th summand of
$\widehat\YY_{\LL}$ is of the form \eqref{eq:basic-integral}. Apply
Lemma~\ref{lem:incidence-closure} $|\mm|$ times. Each application lowers
one of the lengths $m_j$ by one, so after $|\mm|$ applications all the
$m_j$ are zero. The result is $p_*(\widetilde P\cap[S^{e'}\times B])$ for a basic polynomial
$\widetilde P$ on $S^{e'}\times B$, $e'=e+|\mm|$, in which all classes of
type (B1) are equal to $1$. Lemma~\ref{lem:diagonal-contraction} expresses
it as a polynomial in the $\kappa$-classes. In each of these steps the
new polynomial is obtained from the previous one by a rule that depends
only on $N$, the weights $w_i$, the $d_a$, $\mm$, $U$ and $b$, and not on
$S$, $B$, $\LL$ or the $\alpha_a$. None of the steps uses connectedness of
$S$.
\end{proof}

\begin{remark}\label{rem:no-fiber-restriction}
At no point is a class restricted to a fiber of $S\times B\to B$; all
identities are pushforward identities over $B$. This is what allows the
line bundles $L_i$ to vary nontrivially over $B$, and it is the reason why
the mixed K\"unneth components of $c_1(L_i)$ cause no difficulty when the
irregularity of $S$ is positive.
\end{remark}

\subsection{Linearity of the logarithm}

\begin{lemma}[Linearity of the logarithm]\label{lem:linear-log}
Each coefficient of $\log\widehat\YY_{\LL}$ $($for a fixed power of $q$, a
fixed monomial in the $\epsilon_a$ and a fixed cohomological degree$)$ is a
linear combination of finitely many classes $\kappa_P$, with coefficients
in $\Q(z)$ that do not depend on $S$, $B$, $\LL$ or the $\alpha_a$.
\end{lemma}
\begin{proof}
Since $\widehat\YY_{\LL}=1+O(q)$, a fixed coefficient of the logarithm is a
polynomial in finitely many coefficients of $\widehat\YY_{\LL}$. By
Proposition~\ref{prop:relative-universality} it is therefore
$\mathfrak Q(\kappa)$ for a universal polynomial $\mathfrak Q$ in finitely
many $\kappa$-classes.

Let $h\geq1$ and let $S^{\sqcup h}$ be the disjoint union of $h$ copies of
$S$, with the same $L_i$ and $\alpha_a$ on each copy. A Hilbert scheme of
points on $S^{\sqcup h}$ is the disjoint union, over the ways of
distributing the points, of products of Hilbert schemes of the copies. The
universal ideals split accordingly, and there are no $\mathcal Ext$ sheaves
between different copies. Hence the classes $\Delta W_{ij}$, $M_i^{[m_i]}$
and the exponents of $\widehat\Theta$ are sums over the copies, the
integrand of \eqref{eq:normalized-series} is a product over the copies,
and pushing forward along a fiber product over $B$ gives
\[
 \widehat\YY_{\LL}(S^{\sqcup h})=\widehat\YY_{\LL}(S)^h .
\]
On the other hand $\kappa_P(S^{\sqcup h})=h\,\kappa_P(S)$. Applying
universality to $S^{\sqcup h}$ we get
$\mathfrak Q(h\kappa)=h\,\mathfrak Q(\kappa)$ for all $h\geq1$, where
$\kappa$ stands for the $\kappa$-classes of $S$. Decompose
$\mathfrak Q=\sum_{r\geq0}\mathfrak Q_r$ into homogeneous parts with
respect to the total degree in the variables. Then
$\sum_rh^r\mathfrak Q_r(\kappa)=h\sum_r\mathfrak Q_r(\kappa)$ for all
positive integers $h$. Comparing coefficients of powers of $h$ in a basis
of the finite-dimensional space $\HB\otimes\Q(z)$ gives
$\mathfrak Q_r(\kappa)=0$ for $r\neq1$. So
$\mathfrak Q(\kappa)=\mathfrak Q_1(\kappa)$, which is linear.
\end{proof}

\section{Twisting by an ample divisor}\label{sec:twisting}

The geometric arguments of the next sections need the line bundles
$M_i=K_S\otimes L_i^{-1}$ to be very positive on $S$. The Hilbert series is
defined for arbitrary $L_i$, and universality lets us trade a given tuple
$\LL$ for sufficiently negative twists of it.

Fix an ample divisor $H$ on $S$. For $k\in\Z$ put
\[
 L_i(k)=L_i\otimes\OO_S(-kH),\qquad
 M_i(k)=K_S\otimes L_i(k)^{-1}=M_i\otimes\OO_S(kH),\qquad
 \YY(k)=\YY_{\LL(k)} ,
\]
where the same $k$ is used for all $i$.

\begin{proposition}[Twisting identity]\label{prop:twisting}
For every integer $K\geq0$ put
\begin{equation}\label{eq:interpolation-weights}
 a_K=\frac{(K+1)(K+2)}2,\qquad b_K=-K(K+2),\qquad c_K=\frac{K(K+1)}2 .
\end{equation}
Then, in $\HB[\eps]\otimes\Q(z)[[q]]$,
\begin{equation}\label{eq:twisting-identity}
 \YY(0)=\YY(K)^{a_K}\,\YY(K+1)^{b_K}\,\YY(K+2)^{c_K}.
\end{equation}
\end{proposition}
\begin{proof}
The integers \eqref{eq:interpolation-weights} are the values at $0$ of the
Lagrange interpolation polynomials for the nodes $K,K+1,K+2$. That is, for
every polynomial $g$ of degree at most two,
\begin{equation}\label{eq:interpolation-moment-identities}
 g(0)=a_K\,g(K)+b_K\,g(K+1)+c_K\,g(K+2).
\end{equation}

\emph{The normalized series.} Twisting replaces $c_1(L_i)$ by
$c_1(L_i)-kH$ in \eqref{eq:kappa-classes}. Since $H$ is pulled back from the
surface, $H^3=0$ on $S\times B$, so every class $\kappa_P(k)$ is a polynomial
in $k$ of degree at most two with coefficients in $\HB$. By
Lemma~\ref{lem:linear-log} the same is true of every coefficient of
$\log\widehat\YY(k)$:
\begin{equation}\label{eq:twisting-logarithm}
 \log\widehat\YY(k)=F_0+k\,F_1+k^2F_2
\end{equation}
for series $F_0,F_1,F_2$ with zero constant term that do not depend on $k$.
Applying \eqref{eq:interpolation-moment-identities} and exponentiating gives
\eqref{eq:twisting-identity} for $\widehat\YY$.

\emph{The constant term.} By \eqref{eq:Wzero}, $\YY^0_{\LL(k)}$ is a product
of classes of the three kinds
\[
 c(V),\qquad\MM_w(V),\qquad\exp\bigl(\epsilon_a\ch^{\TT}_{d_a}(V[w])\bigr),
\]
where $V=V(k)$ is one of the $K$-theory classes
\[
 R\pi_*\bigl(L_i(-kH)\bigr),\qquad R\pi_*\bigl(L_iL_j^{-1}\bigr),\qquad
 R\pi_*\bigl((\OO-L_i^{-1}(kH))\otimes\alpha_a\bigr)
\]
on $B$. By Grothendieck--Riemann--Roch, for instance
$\ch R\pi_*(L_i(-kH))=\pi_*(\ch(L_i)e^{-kH}\td_S)$, the Chern character of
each $V(k)$ is a polynomial of degree at most two in $k$, in every degree
including degree zero. Hence
$\ch V(0)=a_K\ch V(K)+b_K\ch V(K+1)+c_K\ch V(K+2)$. Each of the three kinds
of characteristic class is multiplicative and depends only on the Chern
character, so the same relation holds multiplicatively for them. This gives
\eqref{eq:twisting-identity} for the constant terms, and multiplying the two
identities completes the proof.
\end{proof}

\begin{corollary}\label{cor:positive-twist-reduction}
Suppose that there is an integer $k_0$ such that
$\YY(k)\in\HB[\eps]\otimes_{\Q}\RF$ for all $k\geq k_0$. Then
$\YY_{\LL}=\YY(0)\in\HB[\eps]\otimes_{\Q}\RF$.
\end{corollary}
\begin{proof}
Take $K\geq\max(k_0,0)$. The three series on the right of
\eqref{eq:twisting-identity} have nonzero scalar component, so they are
invertible in $\HB[\eps]\otimes\RF$ by
Lemma~\ref{lem:field-operations}(b).
\end{proof}

\section{Families of integral curves}\label{sec:curves}

Throughout this section $B$ is a smooth connected projective variety,
$\OO_B(1)$ is a fixed ample line bundle on $B$, and $L_1,\ldots,L_N$ are
line bundles on $S\times B$ with $M_i=K_S\otimes L_i^{-1}$. For $b\in B$ we
write $M_{i,b}$ for the restriction of $M_i$ to $S\times\{b\}$.

The proof strategy is as follows. A section of $M_i$ whose zero locus is a family of
curves $\cC_i\to B$ turns the Euler class $e(M_i^{[m]})$ into the relative
Hilbert scheme of points of $\cC_i$. To use Abel maps we need all fibers of
$\cC_i\to B$ to be integral, and no section of $M_i$ need have this property
over the whole of $B$. We therefore first twist by $\OO_B(s)$.

\subsection{Twisting by the base}

\begin{definition}[Base-twisted series]\label{def:base-twisted-series}
For $s\in\Z$ put $M_i\langle s\rangle=M_i\otimes\pr_B^*\OO_B(s)$. The series
$\YY^{(s)}_{\LL}$ is defined by the formula \eqref{eq:hilbert-series} with
the Euler factor replaced by
\[
 \prod_{i=1}^Ne\Bigl(\bigl(M_i\langle s\rangle^{[m_i]}\bigr)^\vee\Bigr),
 \qquad M_i\langle s\rangle^{[m]}=M_i^{[m]}\otimes\OO_B(s),
\]
and all other factors $(W_{ij}$ and $\Theta_{\mm})$ unchanged. Thus
$\YY^{(0)}_{\LL}=\YY_{\LL}$.
\end{definition}

\begin{lemma}[Polynomial dependence on the base twist]\label{lem:base-interpolation}
The series $\YY^{(s)}_{\LL}$ is a polynomial in $s$ of degree at most
$\dim B$, with coefficients in $\HB[\eps]\otimes\Q(z)[[q]]$. Consequently, for
any distinct integers $s_0,\ldots,s_{\dim B}$,
\begin{equation}\label{eq:base-interpolation-explicit}
 \YY_{\LL}=\sum_{j=0}^{\dim B}\Bigl(\prod_{a\neq j}\frac{-s_a}{s_j-s_a}\Bigr)
 \YY^{(s_j)}_{\LL}.
\end{equation}
\end{lemma}
\begin{proof}
Let $\eta=c_1(\OO_B(1))$. For a vector bundle $E$ of rank $m$,
\[
 e\bigl((E\otimes\OO_B(s))^\vee\bigr)
 =\sum_{a=0}^{m}(-s\eta)^a\,c_{m-a}(E^\vee).
\]
In the product of these factors over $i$, every power $s^a$ comes with
$\eta^a$, which is pulled back from $B$ and vanishes for $a>\dim B$. By the
projection formula the powers of $\eta$ can be taken out of the pushforward
$(p_{\mm})_*$. This proves the first claim. For the second, recall the
Lagrange interpolation formula: a polynomial $f(s)$ of degree at most $d$
with coefficients in a $\Q$-vector space satisfies
\[
 f(s)=\sum_{j=0}^{d}\Bigl(\prod_{a\neq j}\frac{s-s_a}{s_j-s_a}\Bigr)f(s_j)
\]
for any distinct rational numbers $s_0,\ldots,s_d$. Taking
$f(s)=\YY^{(s)}_{\LL}$ and $d=\dim B$ and setting $s=0$ gives
\eqref{eq:base-interpolation-explicit}.
\end{proof}

\subsection{Sections with integral zero loci}

\begin{lemma}[Positivity after twisting]\label{lem:positivity}
Let $M$ be a line bundle on $S\times B$ and let $a\geq0$ be an integer.
There is an integer $k_0$ such that for all $k\geq k_0$ and all $b\in B$ the
line bundle $M_b(kH)=M|_{S\times\{b\}}\otimes\OO_S(kH)$ on $S$ satisfies
$H^1(S,M_b(kH))=H^2(S,M_b(kH))=0$ and is $a$-very ample,
that is, $H^0(S,M_b(kH))\to H^0(Z,M_b(kH)|_Z)$ is surjective for every
zero-dimensional subscheme $Z\subset S$ of length $a+1$.
\end{lemma}
\begin{proof}
Let $\cZ\subset S\times S^{[a+1]}$ be the universal subscheme. Apply
relative Serre vanishing, with respect to the relatively ample line bundle
$\OO_S(H)$, to $M$ on $S\times B\to B$ and to $M\otimes I_{\cZ}$ on
$S\times B\times S^{[a+1]}\to B\times S^{[a+1]}$. For $k\gg0$ the higher
direct images of $M(kH)$ and of $M(kH)\otimes I_{\cZ}$ vanish. Both sheaves
are flat over the respective base, so by cohomology and base change
$H^p(S,M_b(kH))=0$ and $H^p(S,M_b(kH)\otimes I_Z)=0$ for $p\geq1$, all $b$
and all $Z$. The latter vanishing for $p=1$ is the required surjectivity.
\end{proof}

\begin{lemma}[Integral curve families over a proper base]\label{lem:integral-curves}
Assume that for all $i$ and all $b\in B$
\begin{equation}\label{eq:positivity-hypothesis}
 H^1(S,M_{i,b})=H^2(S,M_{i,b})=0\quad\text{and}\quad
 M_{i,b}\text{ is $(2\dim B+1)$-very ample.}
\end{equation}
Then there is an integer $s_0$ such that for every $s\geq s_0$ there are
sections of the line bundles $M_i\langle s\rangle$, $i=1,\ldots,N$, whose
zero divisors
\[
 \cC_i\subset S\times B
\]
are flat over $B$ with all fibers integral curves. For such sections and
every $m\geq0$:
\begin{enumerate}[label=\textup{(\alph*)},leftmargin=2.2em]
\item the relative Hilbert scheme $\cC_i^{[m]}:=\Hilb^m(\cC_i/B)$ is the zero
scheme of the induced section of the rank-$m$ bundle
$M_i\langle s\rangle^{[m]}$ on $S^{[m]}\times B$, and this section is
regular;
\item $\cC_i^{[m]}$ is proper and flat over $B$, a local complete
intersection, with integral fibers of dimension $m$.
\end{enumerate}
\end{lemma}
\begin{proof}
Write $\delta=\dim B$. By the vanishing in
\eqref{eq:positivity-hypothesis} and cohomology and base change,
$\mathcal V_i=\pi_*M_i$ is a vector bundle on $B$ with fibers
$H^0(S,M_{i,b})$. (This is the only use of the vanishing.) Its rank is at least $2\delta+2$, because
a $(2\delta+1)$-very ample line bundle has at least that many sections.

\emph{The locus of non-integral divisors.} Let
$\mathcal N_i\subset\mathcal V_i$ be the subset of the total space of
$\mathcal V_i$ consisting of the sections that are zero or whose divisor
of zeros is not integral. It is invariant under multiplication by
scalars, and it is closed because integrality of the fibers is an open
condition in a proper flat family. We write
$\mathcal N_{i,b}=\mathcal N_i\cap\mathcal V_{i,b}$ for its fiber over
$b\in B$. By \cite[Proposition~5.1]{KT}, if a line bundle on $S$
is $(2\delta+1)$-very ample, then all curves in a general
$\delta$-dimensional linear subsystem are reduced and irreducible. A general
$\PP^\delta$ in $\PP(\mathcal V_{i,b})$ therefore misses the locus of
non-integral divisors, so this locus has codimension greater than
$\delta$. Hence $\mathcal N_{i,b}$ has codimension greater than $\delta$
in $\mathcal V_{i,b}$ for every $b$.

\emph{Choosing the sections.} For $s\gg0$ the bundle
$\mathcal V_i\otimes\OO_B(s)$ is globally generated. Put
$\Gamma_i=H^0(B,\mathcal V_i\otimes\OO_B(s))=H^0(S\times B,
M_i\langle s\rangle)$ and consider the surjective evaluation map
$\mathrm{ev}\colon\Gamma_i\times B\to\mathcal V_i\otimes\OO_B(s)$ of
vector bundles on $B$. Since $\mathcal N_i$ is invariant under
multiplication by scalars, it defines a closed subset
$\mathcal N_i\langle s\rangle$ of the total space of
$\mathcal V_i\otimes\OO_B(s)$, whose fiber over $b$ again has codimension
greater than $\delta$. Let
$\mathcal R_i=\mathrm{ev}^{-1}(\mathcal N_i\langle s\rangle)
\subset\Gamma_i\times B$. For every $b\in B$ the map
$\mathrm{ev}_b$ is a surjective linear map from $\Gamma_i$ to the fiber of
$\mathcal V_i\otimes\OO_B(s)$ at $b$, so the fiber of $\mathcal R_i$ over
$b$ has codimension greater than $\delta$ in $\Gamma_i$. Therefore
$\dim\mathcal R_i\leq\dim\Gamma_i-1$. The image of $\mathcal R_i$ under the
projection $\Gamma_i\times B\to\Gamma_i$ is closed, because $B$ is proper,
and by this dimension bound it is a proper subset of $\Gamma_i$. Every
section in the complement of this image has an integral divisor of zeros
in each fiber $S\times\{b\}$.

\emph{Flatness.} A local equation of $\cC_i$ is a nonzerodivisor on every
fiber $S\times\{b\}$, so by the local criterion \cite[Tag~00MF]{Stacks} it
is a nonzerodivisor on $S\times B$ and $\cC_i$ is flat over $B$.

\emph{The relative Hilbert scheme.} A length-$m$ subscheme of $S$ lies on a
curve $C=\{\sigma=0\}$ exactly when the induced section of $M^{[m]}$
vanishes at the corresponding point; this identifies $\Hilb^m(\cC_i/B)$
with the zero scheme in (a). Its fiber over $b$ is the Hilbert scheme of
points of the integral curve $\cC_{i,b}$, which lies on a smooth surface
and is therefore locally planar. By \cite[Theorem~3]{MS} this fiber is
integral, of dimension $m$, and a local complete intersection. So on each
fiber of $S^{[m]}\times B\to B$ the $m$ local components of the section
form a regular sequence. By \cite[Tag~00MG]{Stacks} they form a regular
sequence on $S^{[m]}\times B$ and the zero scheme is flat over $B$. This
proves (a) and (b); properness is clear.
\end{proof}

\begin{corollary}[Euler classes become curves]\label{cor:euler-curves}
In the situation of Lemma~\ref{lem:integral-curves} put
\[
 \cC^{[\mm]}=\cC_1^{[m_1]}\times_B\cdots\times_B\cC_N^{[m_N]}\subset X_{\mm}.
\]
Then the coefficient of $q^{|\mm|}$ contributed by $\mm$ to
$\YY^{(s)}_{\LL}$ is
\begin{equation}\label{eq:curve-integrand}
 (-1)^{|\mm|}\,p_*\Bigl(\prod_ic(W_{ii})\prod_{i\neq j}\MM_{w_j-w_i}(W_{ij})\,
 \Theta_{\mm}\cap[\cC^{[\mm]}]\Bigr),
\end{equation}
where the classes are restricted to $\cC^{[\mm]}$ and
$p:\cC^{[\mm]}\to B$ is the projection.
\end{corollary}
\begin{proof}
The subscheme $\cC^{[\mm]}$ is the zero scheme of a section of
$\bigoplus_iM_i\langle s\rangle^{[m_i]}$ on $X_{\mm}$. It is flat over $B$
with fibers of dimension $|\mm|$, so it has codimension $|\mm|$ in the
smooth variety $X_{\mm}$ and the section is regular. Its fundamental class
is therefore the top Chern class
$\prod_ie(M_i\langle s\rangle^{[m_i]})$, and passing to the dual bundles
gives the sign $(-1)^{|\mm|}$.
\end{proof}

\subsection{Duality on families of integral curves}
The Abel map sends a subscheme of a curve to the dual of its ideal. We
need this duality in families.

\begin{lemma}[Relative duality]\label{lem:relative-curve-duality}
Let $\cC\subset S\times B$ be a flat family of integral curves, let $B'$ be
a locally Noetherian $B$-scheme, and let
$\iota:C'=\cC\times_BB'\hookrightarrow S\times B'$. Let $G$ be a coherent
sheaf on $C'$, flat over $B'$, whose fibers are torsion-free of rank one.
Then:
\begin{enumerate}[label=\textup{(\alph*)},leftmargin=2.2em]
\item $\iota_*G$ has, locally on $S\times B'$, a resolution
$0\to\OO^a\to\OO^b\to\iota_*G\to0$; in particular it is perfect;
\item $G^\vee=\mathcal Hom_{C'}(G,\OO_{C'})$ is flat over $B'$ with
torsion-free rank-one fibers, its formation commutes with base change, and
$G\to G^{\vee\vee}$ is an isomorphism;
\item $\mathcal Ext^p_{S\times B'}(\iota_*G,\OO)$ vanishes for $p\neq1$,
and
\begin{equation}\label{eq:relative-duality-matrix}
 \mathcal Ext^1_{S\times B'}(\iota_*G,\OO)
 =\iota_*\bigl(G^\vee\otimes\OO_{C'}(C')\bigr).
\end{equation}
\end{enumerate}
\end{lemma}
\begin{proof}
The statements are local. Let $R\to A$ be the local homomorphism of local
rings of $B'$ and $S\times B'$ at a point, and let $M$ be the stalk of
$\iota_*G$. Choose a surjection $A^b\to M$ with kernel $P$. As $M$ is flat
over $R$, so is $P$, and the sequence stays exact on the fiber. On the
fiber, which is a regular local ring of dimension two, $M$ is zero or
Cohen--Macaulay of dimension one, hence of projective dimension at most one
by the Auslander--Buchsbaum formula. So the fiber of $P$ is free, and $P$
is free by \cite[Tag~00MH]{Stacks}. This proves (a).

Dualizing the resolution $0\to A^a\xrightarrow{d}A^b\to M\to0$ gives a
complex $A^b\xrightarrow{d^\vee}A^a$ computing
$\operatorname{Ext}^\bullet_A(M,A)$. On each fiber $d^\vee$ is injective,
because a torsion module has no nonzero maps to a domain. By
\cite[Tag~00ME]{Stacks}, $d^\vee$ is injective and its cokernel
$\operatorname{Ext}^1_A(M,A)$ is flat over $R$. So
$\operatorname{Ext}^p_A(M,A)=0$ for $p\neq1$, the formation of
$\operatorname{Ext}^1$ commutes with base change, and dualizing once more
returns $M$. Since $C'\subset S\times B'$ is a Cartier divisor,
$\iota^!\OO=\OO_{C'}(C')[-1]$, and Grothendieck duality for $\iota$
identifies $\mathcal Ext^1(\iota_*G,\OO)$ with the right-hand side of
\eqref{eq:relative-duality-matrix}. This proves (c), and (b) follows by
transporting flatness, base change and biduality through
\eqref{eq:relative-duality-matrix}. On a fiber, $G^\vee$ is again
torsion-free of rank one.
\end{proof}

\subsection{Abel maps and projective bundles}\label{subsec:abel}
Fix one of the curve families of Lemma~\ref{lem:integral-curves} and call it
$\cC\to B$. Let $g$ be the arithmetic genus of its fibers and put
\[
 \nu=H\cdot[\cC_b]>0 ;
\]
both are independent of $b$. For $m\in\Z$ let $\Jb^{\,m}\to B$ be the
relative compactified Jacobian parametrizing rank-one torsion-free sheaves
of degree $m$ on the fibers, where $\deg G=\chi(G)-\chi(\OO_{\cC_b})$. It is
a projective $B$-scheme \cite[Theorem~8.5]{AK}. (On an integral curve every
rank-one torsion-free sheaf is stable, so no stability condition
intervenes.) Tensoring with $\OO_S(kH)$ gives isomorphisms
\begin{equation}\label{eq:jacobian-twist}
 \Jb^{\,m}\xrightarrow{\ \sim\ }\Jb^{\,m+k\nu},\qquad G\longmapsto G(kH).
\end{equation}
The \emph{Abel map} is
\[
 \cC^{[m]}\longrightarrow\Jb^{\,m},\qquad Z\longmapsto I_{Z/\cC_b}^\vee ;
\]
it is well defined in families by
Lemma~\ref{lem:relative-curve-duality}(b).

We fix conventions for projective bundles.

\begin{definition}\label{def:projectivization}
For a vector bundle $E$ of rank $r$ on a scheme $Y$,
\begin{equation}\label{eq:projectivization-lines}
 p:\PP(E)=\operatorname{Proj}_Y(\operatorname{Sym}^\bullet E^\vee)
 \longrightarrow Y
\end{equation}
is the bundle of \emph{lines} in $E$; it carries the tautological line
subbundle $\OO_{\PP(E)}(-1)\subset p^*E$ and its dual $\OO_{\PP(E)}(1)$.
A smooth projective morphism $f:X\to Y$ is a \emph{Brauer--Severi morphism
of relative dimension $r-1$} if, after pullback along some surjective
\'etale morphism $Y'\to Y$, it becomes isomorphic to
$\PP^{r-1}_{Y'}\to Y'$.
\end{definition}

\begin{lemma}[The Abel map over a fixed Hilbert scheme]\label{lem:abel-bundle}
Let $m_*\geq0$ be an integer with $m_*>2g-2$.
\begin{enumerate}[label=\textup{(\alph*)},leftmargin=2.2em]
\item The Abel map $\cC^{[m_*]}\to\Jb^{\,m_*}$ is a Brauer--Severi morphism
of relative dimension $m_*-g$.
\item On $\cC\times_B\cC^{[m_*]}$ the dual $G_*$ of the universal ideal is a
flat family of rank-one torsion-free sheaves of degree $m_*$.
\item Let $k\geq0$ and $m=m_*+k\nu$. Then
\begin{equation}\label{eq:abel-projective}
 E(k)=\pi_*\bigl(G_*\otimes\OO_S(kH)\bigr)
\end{equation}
is a vector bundle of rank $m+1-g$ on $\cC^{[m_*]}$, and there is a
cartesian square
\[
\begin{array}{ccc}
 \PP(E(k)) & \longrightarrow & \cC^{[m]}\\
 \downarrow & & \downarrow\\
 \cC^{[m_*]} & \longrightarrow & \Jb^{\,m_*}\cong\Jb^{\,m}
\end{array}
\]
in which the right vertical map and the bottom map are Abel maps and the
identification of the Jacobians is \eqref{eq:jacobian-twist}. On
$\cC\times_B\PP(E(k))$ the pullback of the universal ideal of $\cC^{[m]}$ is
\begin{equation}\label{eq:ideal-on-bundle}
 G_*(kH)^\vee\otimes\OO_{\PP(E(k))}(-1).
\end{equation}
\end{enumerate}
\end{lemma}
\begin{proof}
\emph{Step 1: cohomology.} Let $C$ be a fiber and $G$ a rank-one torsion-free sheaf
of degree $m>2g-2$ on $C$. By Serre duality
$H^1(C,G)^\vee=\Hom(G,\omega_C)$. A nonzero map $G\to\omega_C$ is injective,
as $G$ is torsion-free of rank one, and its cokernel has finite length;
this would give $m\leq\deg\omega_C=2g-2$. So $H^1(C,G)=0$ and
$h^0(C,G)=m+1-g$.

\emph{Step 2: sections and subschemes.} A nonzero section of $G$ is an injection
$\OO_C\to G$ whose cokernel has finite length
$\chi(G)-\chi(\OO_C)=m$. Dualizing, and using
Lemma~\ref{lem:relative-curve-duality}, gives an ideal
$G^\vee\subset\OO_C$ of colength $m$. Conversely the dual of an ideal $I_Z$
of colength $m$ is a rank-one torsion-free sheaf of degree $m$ with a
section, dual to the inclusion $I_Z\subset\OO_C$. By biduality these
constructions are mutually inverse, and by
Lemma~\ref{lem:relative-curve-duality}(b) they work in families. Thus
$\cC^{[m]}$ represents the functor of pairs $(G,\text{a line of sections of
}G)$, and (b) holds.

\emph{Step 3: proof of} (c). By the vanishing of Step~1 and cohomology and
base change, $E(k)$ is a vector bundle of rank $m+1-g$ whose formation commutes
with base change. A point of the fiber product
$\cC^{[m]}\times_{\Jb^{\,m}}\cC^{[m_*]}$ is a pair $(Z,Z_*)$ with
$I_Z^\vee\cong I_{Z_*}^\vee(kH)$, that is, a point $Z_*$ of $\cC^{[m_*]}$
together with a line of sections of $G_*(kH)|_{Z_*}$. By Step~2, applied
in families, the fiber product is $\PP(E(k))$. The
universal line of sections is $\OO(-1)\to G_*(kH)$; dualizing and twisting
by $\OO(-1)$ gives the inclusion of \eqref{eq:ideal-on-bundle} into $\OO$.

\emph{Step 4: proof of} (a). Sheaves on the fibers of $\cC\to B$ that are rank-one
and torsion-free have only scalar automorphisms, so their moduli stack is a
$\mathbb G_m$-gerbe over $\Jb^{\,m_*}$; \'etale locally on $\Jb^{\,m_*}$
there is a universal sheaf, unique up to twisting by a line bundle from the
base \cite[Sections~5 and~8]{AK}. By Steps~1 and~2, over such an
\'etale neighborhood the Abel map is the projectivization of the pushforward
of the universal sheaf, a vector bundle of rank $m_*+1-g$. Compare
\cite[Section~1.3.4]{JOP}.
\end{proof}

A Brauer--Severi morphism need not be a projective bundle, but one can
still integrate along it, because there is a canonical class of fiber
degree one.

\begin{lemma}[A class of degree one on the fibers]\label{lem:brauer-class}
Let $f:Y\to J$ be a Brauer--Severi morphism of relative dimension $r-1$, and
put
\begin{equation}\label{eq:degree-one-chern-expression}
 \eta_f=\Bigl(\frac{c_1(\omega_{Y/J}^{-1})}{r}\Bigr)^{r-1}.
\end{equation}
For every morphism $J'\to J$ and every integral closed subscheme
$V\subset J'$, with $f_V:Y_V\to V$ the base change of $f$,
\begin{equation}\label{eq:degree-one-cycle}
 (f_V)_*\bigl(\eta_f\cap[Y_V]\bigr)=[V].
\end{equation}
\end{lemma}
\begin{proof}
On a fiber $\PP^{r-1}$ the relative anticanonical bundle is $\OO(r)$, so
$\eta_f$ restricts to the class of a point. The cycle on the left of
\eqref{eq:degree-one-cycle} has dimension $\dim V$ and is supported on $V$;
its multiplicity is the degree of $\eta_f$ on the generic fiber, which is
one.
\end{proof}

\subsection{A base that does not depend on the lengths}
We return to the $N$ curve families $\cC_i\to B$ of
Lemma~\ref{lem:integral-curves}, with genera $g_i$ and degrees
$\nu_i=H\cdot[\cC_{i,b}]$. For each $i$ and each residue class
$\rho\in\Z/\nu_i$ choose an integer
\[
 m_{i,\rho}\geq0,\qquad m_{i,\rho}\equiv\rho\ (\mathrm{mod}\ \nu_i),\qquad
 m_{i,\rho}>2g_i-2 .
\]
This partitions the set $\Z_{\geq0}$ of possible lengths $m_i$ into
\begin{equation}\label{eq:length-index-partition}
 \text{finitely many singletons }\{c\}\ (c<m_{i,\rho},\ c\equiv\rho)
 \quad\text{and the progressions }\{m_{i,\rho}+k\nu_i:\ k\geq0\}.
\end{equation}

\begin{definition}[Types]\label{def:type}
A \emph{type} is a choice, for each $i\in\{1,\ldots,N\}$, of one of the
sets of the partition \eqref{eq:length-index-partition} of $\Z_{\geq0}$,
that is, of a singleton $\{c\}$ or of a progression
$\{m_{i,\rho}+k\nu_i:\ k\geq0\}$. For a given type we use the following
notation:
\begin{itemize}[leftmargin=1.6em]
\item $I\subset\{1,\ldots,N\}$ is the set of indices $i$ for which the
chosen set is a progression, and $I^c=\{1,\ldots,N\}\setminus I$;
\item for $i\in I$, the \emph{starting length} $m_{i,*}$ is the smallest
element $m_{i,\rho}$ of the chosen progression;
\item for $j\in I^c$, the \emph{fixed length} $m_j$ is the element of the
chosen singleton.
\end{itemize}
A length vector $\mm=(m_1,\ldots,m_N)\in\Z_{\geq0}^N$ is \emph{of the given
type} if $m_i$ belongs to the chosen set for every $i$. Equivalently,
$m_i=m_{i,*}+k_i\nu_i$ with $k_i\in\Z_{\geq0}$ for $i\in I$, and $m_j$ is
the fixed length for $j\in I^c$.
\end{definition}

There are finitely many types, and every length vector in $\Z_{\geq0}^N$
is of exactly one type.

\begin{proposition}[Projective bundles over a fixed base]\label{prop:fixed-base}
Fix a type $($Definition~\ref{def:type}$)$, with index set $I$, starting
lengths $m_{i,*}$ for $i\in I$ and fixed lengths $m_j$ for $j\in I^c$.
There are a smooth projective variety $\cT$ with a morphism
$\cT\to B$, vector bundles $E_i(k)$ on $\cT$ for $i\in I$ and $k\geq0$, and
a class $\eta_{\cT}\in H^{\mathrm{even}}(\cT,\Q)$, with the following
property. For $(k_i)_{i\in I}\in\Z_{\geq0}^I$ let $\mm$ be the length vector
with $m_i=m_{i,*}+k_i\nu_i$ for $i\in I$, and put
\begin{equation}\label{eq:fixed-base-projective}
 \PP_{\boldsymbol k}=\prod_{i\in I,\ \cT}\PP\bigl(E_i(k_i)\bigr),
\end{equation}
the fiber product over $\cT$. Then there is a proper morphism
$\varPi:\PP_{\boldsymbol k}\to\cC^{[\mm]}$ over $B$ such that, for every
polynomial $\Xi$ in Chern classes of perfect complexes on $\cC^{[\mm]}$,
\begin{equation}\label{eq:fixed-base-integral}
 p_*\bigl(\Xi\cap[\cC^{[\mm]}]\bigr)
 =p_*\bigl(\varPi^*\Xi\cdot\eta_{\cT}\cap[\PP_{\boldsymbol k}]\bigr)
 \qquad\text{in }\HB\otimes\Q(z),
\end{equation}
where $p$ denotes the projections to $B$. Moreover:
\begin{enumerate}[label=\textup{(\alph*)},leftmargin=2.2em]
\item $\cT$ and $\eta_{\cT}$ depend only on the type, not on
$\boldsymbol k$;
\item for $i\in I$ there is a sheaf $G_{i,*}$ on $\cC_i\times_B\cT$, flat
over $\cT$ with rank-one torsion-free fibers, such that
$E_i(k)=\pi_*(G_{i,*}(kH))$ has rank $m_i+1-g_i$, and the pullback under
$\varPi$ of the universal ideal sheaf of the $i$th factor is
\[
 I_{\cZ_i/\cC_i}=G_{i,*}(k_iH)^\vee\otimes\zeta_i^{-1},\qquad
 \zeta_i=\OO_{\PP(E_i(k_i))}(1);
\]
\item for $j\in I^c$ the pullback of the universal ideal sheaf of the
$j$th factor is pulled back from $\cC_j\times_B\cT$.
\end{enumerate}
\end{proposition}
\begin{proof}
Put
\begin{equation}\label{eq:fixed-base-parameter-scheme}
 \cT_0=\prod_{i\in I,\ B}\cC_i^{[m_{i,*}]}\ \times_B
       \prod_{j\in I^c,\ B}\cC_j^{[m_j]} .
\end{equation}
By Lemma~\ref{lem:integral-curves}(b) this scheme is proper and flat over
$B$ with integral fibers, hence integral. Let $G_{i,*}$ and
$E_i(k)$ be the pullbacks to $\cT_0$ of the sheaf and the bundles of
Lemma~\ref{lem:abel-bundle} for the family $\cC_i$ and the starting length
$m_{i,*}$. The cartesian squares of Lemma~\ref{lem:abel-bundle}(c), one for
each $i\in I$, combine to a cartesian square
\[
\begin{array}{ccc}
 \prod_{i\in I,\ \cT_0}\PP(E_i(k_i)) & \longrightarrow & \cC^{[\mm]}\\
 \downarrow & & \downarrow\\
 \cT_0 & \longrightarrow &
 \prod_{i\in I,\ B}\Jb_i^{\,m_{i,*}}\times_B\prod_{j\in I^c,\ B}\cC_j^{[m_j]},
\end{array}
\]
in which the right vertical map is the product of the Abel maps
$\cC_i^{[m_i]}\to\Jb_i^{\,m_i}\cong\Jb_i^{\,m_{i,*}}$ for $i\in I$ (the
identification being \eqref{eq:jacobian-twist}, tensoring with
$\OO_S(k_iH)$) and of the identity for $j\in I^c$. The bottom map is a
product of Brauer--Severi morphisms by Lemma~\ref{lem:abel-bundle}(a). Let
$\eta_0$ be the product of the classes
\eqref{eq:degree-one-chern-expression} of its factors, a class on $\cT_0$.
The top map $\varPi_0$ is the base change of the bottom map. It is
therefore a composition of base changes of Brauer--Severi morphisms, one
for each $i\in I$, and applying Lemma~\ref{lem:brauer-class} to each of
them in turn gives
$(\varPi_0)_*\bigl(\eta_0\cap[\prod_{i\in I,\,\cT_0}\PP(E_i(k_i))]\bigr)
=[\cC^{[\mm]}]$. By the projection formula
\eqref{eq:characteristic-projection-formula} this is
\eqref{eq:fixed-base-integral} with $\cT_0$ and $\eta_0$ in place of $\cT$
and $\eta_{\cT}$.

Finally let $r:\cT\to\cT_0$ be a resolution of singularities, pull back
$G_{i,*}$ and $E_i(k)$, and put $\eta_{\cT}=r^*\eta_0$. The induced map of
products of projective bundles is proper and birational, so
\eqref{eq:fixed-base-integral} holds over $\cT$ as well, again by
\eqref{eq:characteristic-projection-formula}. Properties (a)--(c) hold by
construction and by \eqref{eq:ideal-on-bundle}; in (b), the formation of
$G_{i,*}^\vee$ commutes with the base change to $\cT$ by
Lemma~\ref{lem:relative-curve-duality}(b).
\end{proof}

\section{Projective bundle integrals as residues}\label{sec:residue-geometry}

We keep the notation of Section~\ref{sec:curves}: the base-twist $s$, the
curve families $\cC_i$, a fixed type $($Definition~\ref{def:type}$)$ with
index set $I$, and the data $\cT$, $G_{i,*}$, $E_i(k)$, $\eta_{\cT}$ and
$\PP_{\boldsymbol k}$ of Proposition~\ref{prop:fixed-base}. In this section
we compute the pushforward of the integrand \eqref{eq:curve-integrand}
from $\PP_{\boldsymbol k}$ to $\cT$, and we determine its dependence on
$\boldsymbol k=(k_i)_{i\in I}$.
For $i\in I$ write
\[
 G_i(k)=G_{i,*}\otimes\OO_S(kH),\qquad
 \zeta_i=\OO_{\PP(E_i(k_i))}(1),\qquad h_i=c_1(\zeta_i),
\]
and let $\iota_i:\cC_i\times_B\cT\hookrightarrow S\times\cT$ be the
inclusion.

\subsection{The kernels on the projective bundles}
On $S\times\PP_{\boldsymbol k}$ the ideal of the universal subscheme
$\cZ_i\subset\cC_i$ sits in the sequence
$0\to\OO(-\cC_i)\to I_{\cZ_i}\to(\iota_i)_*I_{\cZ_i/\cC_i}\to0$. Since
$\OO(\cC_i)=M_i\langle s\rangle=K_S\otimes L_i^{-1}\otimes\OO_B(s)$,
tensoring with $L_i^{-1}$ and using
Proposition~\ref{prop:fixed-base}(b) gives, for $i\in I$,
\begin{equation}\label{eq:kernel-on-curve}
 K_i=\Lambda+\cA_i(k_i)\,\zeta_i^{-1},\qquad
 \Lambda=K_S^{-1}\otimes\OO_B(-s),\qquad
 \cA_i(k)=(\iota_i)_*\bigl(G_i(k)^\vee\otimes L_i^{-1}|_{\cC_i}\bigr),
\end{equation}
an identity in $K^0(S\times\PP_{\boldsymbol k})$ in which $\Lambda$ and
$\cA_i(k_i)$ are pulled back from $S\times\cT$. The class $\Lambda$ does not
depend on $i$, and $\cA_i(k)=\cA_i(0)\otimes\OO_S(-kH)$. For $j\in I^c$ the
class $K_j$ is pulled back from $S\times\cT$ and does not depend on
$\boldsymbol k$.

\begin{lemma}[Decomposition and ranks]\label{lem:rank-calculation}
Let $C_i$ denote the class of a fiber of $\cC_i\to B$ on $S$.
\begin{enumerate}[label=\textup{(\alph*)},leftmargin=2.2em]
\item For $i,j\in I$,
\begin{equation}\label{eq:W-decomposition}
 W_{ij}=W^{00}+W^{10}_{i}\,\zeta_i+W^{01}_{j}\,\zeta_j^{-1}
        +W^{11}_{ij}\,\zeta_i\zeta_j^{-1},
\end{equation}
where the following classes are pulled back from $\cT$:
\begin{align*}
 W^{00}&=R\pi_*\RHom(\Lambda,\OO-\Lambda),&
 W^{10}_{i}&=R\pi_*\RHom(\cA_i(k_i),\OO-\Lambda),\\
 W^{01}_{j}&=-R\pi_*\RHom(\Lambda,\cA_j(k_j)),&
 W^{11}_{ij}&=-R\pi_*\RHom(\cA_i(k_i),\cA_j(k_j)).
\end{align*}
Their ranks are
\begin{equation}\label{eq:rank-table}
 \rk W^{00}=0,\qquad\rk W^{10}_{i}=-K_S\cdot C_i,\qquad
 \rk W^{01}_{j}=m_j+g_j-1-C_j^2,\qquad\rk W^{11}_{ij}=C_i\cdot C_j .
\end{equation}
\item For $i\in I^c$ and $j\in I$,
$W_{ij}=R\pi_*\RHom(K_i,\OO-\Lambda)-R\pi_*\RHom(K_i,\cA_j(k_j))\,\zeta_j^{-1}$,
and the coefficient of $\zeta_j^{-1}$ has rank $m_j$ plus a constant. For
$i\in I$ and $j\in I^c$,
$W_{ij}=R\pi_*\RHom(\Lambda,Q_j)+R\pi_*\RHom(\cA_i(k_i),Q_j)\,\zeta_i$, with
both ranks independent of $\boldsymbol k$. For $i,j\in I^c$, $W_{ij}$ is
pulled back from $\cT$ and independent of $\boldsymbol k$.
\item For every $a\geq1$, the component $\ch_a$ of each coefficient
complex in \textup{(a)} and \textup{(b)}, and of each $E_i(k_i)$, is a
polynomial of degree at most two in $\boldsymbol k$ with coefficients in
$H^{2a}(\cT,\Q)$. The ranks are polynomials of degree at most one in
$\boldsymbol k$.
\end{enumerate}
In short: the only rank that grows with the length $m_j$ is that of the
coefficient of $\zeta_j^{-1}$ in $W_{ij}$, for every $i$, and it grows like
$m_j$.
\end{lemma}
\begin{proof}
The decompositions follow by substituting \eqref{eq:kernel-on-curve} into
$W_{ij}=R\pi_*\RHom(K_i,\OO-K_j)$ and using flat base change along
$\PP_{\boldsymbol k}\to\cT$. The ranks are Euler characteristics on a fiber
$S$ of $S\times\cT\to\cT$. There $\Lambda$ restricts to $K_S^{-1}$, so
$\rk W^{00}=\chi(K_S)-\chi(\OO_S)=0$. For classes $E,F$ of rank zero on a
surface, $\chi(E,F)=-c_1(E)\cdot c_1(F)$; with
$c_1(\cA_i)=C_i$ and $c_1(\OO-\Lambda)=K_S$ this gives $\rk W^{10}_i$ and
$\rk W^{11}_{ij}$. Next, $G_j(k_j)$ has degree $m_j$ on the fibers of
$\cC_j$, and $c_1(L_j)=K_S-C_j$ on $S$, so by Riemann--Roch on the curve
\[
 \chi(\cA_j)=-m_j-(K_S-C_j)\cdot C_j+1-g_j .
\]
Hence $\rk W^{01}_j=-\chi(\cA_j\otimes K_S)=m_j+g_j-1-C_j^2$. For
$i\in I^c$ the class $K_i$ has rank one, so
$\chi(K_i,\cA_j)=\chi(\cA_j)-c_1(K_i)\cdot C_j$ is $-m_j$ plus a constant;
and for $j\in I^c$, $\chi(\cA_i,Q_j)=-C_i\cdot c_1(Q_j)$ does not involve
$m_i$. This proves (a) and (b).

For (c), recall $\cA_i(k)=\cA_i(0)\otimes\OO_S(-kH)$. By
Grothendieck--Riemann--Roch for $S\times\cT\to\cT$, for example
\[
 \ch\bigl(R\pi_*\RHom(\cA_i(k_i),\cA_j(k_j))\bigr)
 =\pi_*\bigl(\ch(\cA_i(0))^\vee\,\ch(\cA_j(0))\,e^{(k_i-k_j)H}\,\td_S\bigr),
\]
and $H^3=0$. The other complexes are treated in the same way. Finally, the
higher direct images of $G_{i,*}(kH)$ vanish by
Lemma~\ref{lem:abel-bundle}, so that
\[
 \ch\bigl(E_i(k)\bigr)=\pi_*\bigl(\ch((\iota_i)_*G_{i,*})\,e^{kH}\,\td_S\bigr).
\]
\end{proof}

\subsection{Three computational lemmas}

\begin{lemma}[Chern and Euler classes of a twist]\label{lem:weight-separation}
Let $E$ be a perfect complex of rank $r$ on a smooth projective variety
$Y$, and let $u$ be an indeterminate. Write $c(E\otimes L_u)$ and
$e(E\otimes L_u)$ for the total Chern class and the Euler class of $E$
twisted by a line bundle with first Chern class $u$ $($defined by adding $u$
to each Chern root$)$. Then
\begin{align}
 c(E\otimes L_u)&=(1+u)^r\exp\Bigl(\sum_{a\geq1}
   \frac{(-1)^{a-1}(a-1)!\,\ch_a(E)}{(1+u)^a}\Bigr),\label{eq:shifted-chern}\\
 e(E\otimes L_u)&=u^r\exp\Bigl(\sum_{a\geq1}
   \frac{(-1)^{a-1}(a-1)!\,\ch_a(E)}{u^a}\Bigr).\label{eq:shifted-euler}
\end{align}
The sums and the exponentials are finite, because classes of positive
degree on $Y$ are nilpotent.
\end{lemma}
\begin{proof}
For a vector bundle with Chern roots $x_1,\ldots,x_r$,
$\prod_b(1+u+x_b)=(1+u)^r\prod_b(1+x_b/(1+u))$; take the logarithm of the
second product and use $\sum_bx_b^a=a!\ch_a(E)$. The Euler class is treated
in the same way, and the general case follows by taking quotients.
\end{proof}

\begin{lemma}[Pushforward along a projective bundle]\label{lem:projective-pushforward}
Let $E$ be a vector bundle of rank $r$ on a smooth projective variety $Y$,
let $p:\PP(E)\to Y$ be as in Definition~\ref{def:projectivization}, and let
$h=c_1(\OO_{\PP(E)}(1))$. Let $c_v(E)=\sum_ic_i(E)v^i$ be the Chern
polynomial, and define the Segre classes $s_b(E)$ by
$\sum_{b\geq0}s_b(E)v^b=c_v(E)^{-1}$. Then $p_*h^a=0$ for $a<r-1$ and
$p_*h^{r-1+b}=s_b(E)$ for $b\geq0$. Equivalently, for every power series
$f(h)$ with coefficients in the cohomology of $Y$,
\begin{equation}\label{eq:projective-residue}
 p_*f(h)=\Res_{h=0}\frac{f(h)\,dh}{h^{r}\,c_{1/h}(E)},\qquad
 \frac1{c_{1/h}(E)}=\sum_{b=0}^{\dim Y}s_b(E)\,h^{-b}.
\end{equation}
\end{lemma}
\begin{proof}
The quotient $p^*E/\OO_{\PP(E)}(-1)$ is a vector bundle of rank $r-1$ on
$\PP(E)$ with total Chern class $p^*c(E)\,(1-h)^{-1}$. Its $r$th Chern class
vanishes, so
\begin{equation}\label{eq:projective-relation}
 \sum_{i=0}^{r}c_i(E)\,h^{r-i}=0\qquad\text{in }H^{2r}(\PP(E),\Q),
\end{equation}
where $c_i(E)$ is pulled back along $p$. The class $p_*h^a$ has degree
$a-(r-1)$, so it vanishes for $a<r-1$; and $p_*h^{r-1}=1$, because $h$
restricts to the hyperplane class on every fiber of $p$. For $b\geq0$ put
$\sigma_b=p_*h^{r-1+b}$, so that $\sigma_0=1$. Multiplying
\eqref{eq:projective-relation} by $h^{b-1}$ for $b\geq1$ and applying
$p_*$ gives, by the projection formula,
\[
 \sum_{i=0}^{\min(r,b)}c_i(E)\,\sigma_{b-i}=0\qquad(b\geq1).
\]
The Segre classes satisfy $s_0(E)=1$ and the same recursion, because
$c_v(E)\sum_{b\geq0}s_b(E)v^b=1$. Hence $\sigma_b=s_b(E)$ for all
$b\geq0$, which is the first statement.

For the second statement write $f(h)=\sum_{a\geq0}f_ah^a$. The residue in
\eqref{eq:projective-residue} is the coefficient of $h^{-1}$ in the Laurent
series $f(h)\,h^{-r}\sum_{b\geq0}s_b(E)h^{-b}$, which is
$\sum_{b\geq0}f_{r-1+b}\,s_b(E)$. By the first statement this sum equals
$\sum_{a\geq0}f_a\,p_*h^a=p_*f(h)$.
\end{proof}

\begin{lemma}[The factor that depends on the length]\label{lem:length-dependent-factor}
For $j\in\{1,\ldots,N\}$ put $s_{ij}=(w_j-w_i)z$ and
\begin{equation}\label{eq:length-factors}
 \varphi_j(h)=(1-h)\prod_{i\neq j}\frac{1+s_{ij}-h}{s_{ij}-h}.
\end{equation}
Then, with $x=w_jz-h$,
\begin{equation}\label{eq:common-ratio}
 \frac{\varphi_j(h)}{h}=-\frac{P_z(x+1)}{P_z(x)}=-R_z(x).
\end{equation}
\end{lemma}
\begin{proof}
Since $h=w_jz-x$, we have $s_{ij}-h=x-w_iz$ for every $i$. Hence
\[
 \frac{1+s_{ij}-h}{s_{ij}-h}=\frac{x+1-w_iz}{x-w_iz}\quad(i\neq j),
 \qquad
 \frac{1-h}{h}=-\frac{x+1-w_jz}{x-w_jz}.
\]
The product of the second identity with the first identities for all
$i\neq j$ is
\[
 \frac{\varphi_j(h)}{h}=-\prod_{i=1}^{N}\frac{x+1-w_iz}{x-w_iz}
 =-\frac{P_z(x+1)}{P_z(x)},
\]
which is \eqref{eq:common-ratio}.
\end{proof}

\subsection{The shape of the integrand}

\begin{proposition}[Residue form of the curve integrals]\label{prop:residue-form}
Fix a type $($Definition~\ref{def:type}$)$, with index set $I$ and starting
lengths $m_{i,*}$. There are an integer $\lambda_0\geq0$ and triples
$(p_\lambda,\gamma_\lambda,f_\lambda)$, $1\leq\lambda\leq\lambda_0$, where
$p_\lambda\in\Q[k_i:i\in I]$,
$\gamma_\lambda\in H^{\mathrm{even}}(\cT,\Q)$ and $f_\lambda$ is an
admissible function $($Definition~\ref{def:local-residues}$)$ of the
variables $x_j$, $j\in I$, with coefficients in $\Q(z)[\eps]$, such that for
all $\boldsymbol k\in\Z_{\geq0}^I$, with $m_i=m_{i,*}+k_i\nu_i$,
\begin{equation}\label{eq:bounded-residue-remainder}
\begin{aligned}
 (-1)^{|\mm|}\,(p_{\cT})_*\,\varPi^*\Bigl(\prod_ic(W_{ii})
 &\prod_{i\neq j}\MM_{w_j-w_i}(W_{ij})\,\Theta_{\mm}\Bigr)\\
 &=\sum_{\lambda=1}^{\lambda_0}p_\lambda(\boldsymbol k)\,\gamma_\lambda
 \Res_{x_j=w_jz,\,j\in I}\Bigl(f_\lambda\prod_{j\in I}R_z(x_j)^{m_j}
 \prod_{j\in I}dx_j\Bigr).
\end{aligned}
\end{equation}
Here $p_{\cT}:\PP_{\boldsymbol k}\to\cT$ is the projection, and the residue
is that of Definition~\ref{def:local-residues}. The integer $\lambda_0$
and the triples do not depend on $\boldsymbol k$.
\end{proposition}
\begin{proof}
We expand the left-hand side in four steps.

\emph{Step 1: the characteristic classes.} By
Lemma~\ref{lem:rank-calculation}, each $W_{ij}$ is a sum of classes of the
form $E\otimes(\text{line bundle})$, where $E$ is pulled back from $\cT$
and the line bundle is $1$, $\zeta_i$, $\zeta_j^{-1}$ or
$\zeta_i\zeta_j^{-1}$. By Lemma~\ref{lem:weight-separation}, applied with
$u\in\{0,\ h_i,\ -h_j,\ h_i-h_j\}$ shifted by $s_{ij}$ for the equivariant
classes, the factor $c(W_{ii})$ or $\MM_{w_j-w_i}(W_{ij})$ is a product of
\begin{itemize}[leftmargin=1.6em]
\item \emph{scalar factors}: powers of the linear forms
\[
 1\pm h_i,\qquad 1+s_{ij}+h_i,\quad1+s_{ij}-h_j,\quad1+s_{ij}+h_i-h_j,\quad
 1+s_{ij}
\]
and, for $i\neq j$, of the same forms without the summand $1$; the
exponents are plus or minus the ranks of the coefficient complexes;
\item \emph{exponential factors}: finite sums of products of classes
$\ch_a$, $a\geq1$, of coefficient complexes, multiplied by negative powers
of the same linear forms.
\end{itemize}
For $i\neq j$ every linear form in this list has nonzero value at
$\boldsymbol h=0$, because $s_{ij}=(w_j-w_i)z\neq0$; the forms $1\pm h_i$
have value $1$. By Lemma~\ref{lem:rank-calculation}, the only exponents
that depend on $\boldsymbol k$ are those of the coefficient of
$\zeta_j^{-1}$, which are $m_j$ plus a constant. Collecting them for all
$i$ gives, for each $j\in I$, the factor
\[
 (1-h_j)^{m_j}\prod_{i\neq j}\Bigl(\frac{1+s_{ij}-h_j}{s_{ij}-h_j}\Bigr)^{m_j}
 =\varphi_j(h_j)^{m_j},
\]
with $\varphi_j$ as in \eqref{eq:length-factors}; the product runs over all
$i\neq j$, whether in $I$ or in $I^c$. By
Lemma~\ref{lem:rank-calculation}(c) and the nilpotence of positive-degree
classes on $\cT$, the exponential factors contribute finitely many terms,
each a polynomial in $\boldsymbol k$ times a class on $\cT$ times a
rational function of $\boldsymbol h$ and $z$ of the kind just described.

\emph{Step 2: the descendents.} By \eqref{eq:twisted-ch} and
\eqref{eq:kernel-on-curve}, each $\mathsf D_{a,\mm}$ is a polynomial in $z$
and the $h_i$ whose coefficients are components of Chern characters of
complexes $R\pi_*(\cA_i(k_i)\otimes\alpha_a)$ and of complexes independent of
$\boldsymbol k$. As in Lemma~\ref{lem:rank-calculation}(c), these
coefficients are polynomials in $\boldsymbol k$.

\emph{Step 3: pushing forward.} Apply
Lemma~\ref{lem:projective-pushforward} to each factor
$\PP(E_j(k_j))$, $j\in I$. Since $\rk E_j(k_j)=m_j+1-g_j$, this multiplies
the integrand by
\begin{equation}\label{eq:projective-segre-expansion}
 h_j^{-m_j}\cdot h_j^{\,g_j-1}\sum_{b=0}^{\dim\cT}s_b\bigl(E_j(k_j)\bigr)h_j^{-b}
\end{equation}
and takes the residue at $h_j=0$. The Segre classes are polynomials in the
classes $\ch_a(E_j(k_j))$, $a\geq1$, hence polynomials in $k_j$.

\emph{Step 4: assembling.} The factors whose exponents involve $m_j$
combine to $(\varphi_j(h_j)/h_j)^{m_j}=(-R_z(x_j))^{m_j}$ with $x_j=w_jz-h_j$, by
Lemma~\ref{lem:length-dependent-factor}. The sign $(-1)^{m_j}$ cancels
against the corresponding part of $(-1)^{|\mm|}$, and $dh_j=-dx_j$. What
remains is a finite sum of terms $p(\boldsymbol k)\,\gamma\,f$ with $f$ a
rational function of the $x_j$ and $z$ $($and a polynomial in the
$\epsilon_a)$ which, by Step~1 and
\eqref{eq:projective-segre-expansion}, has poles near $x_j=w_jz$ only
along the hyperplanes $x_j=w_jz$. So $f$ is admissible.
\end{proof}

\begin{remark}\label{rem:fixed-descendents}
The proof of Proposition~\ref{prop:residue-form} shows how the function
$R_z(x)=P_z(x+1)/P_z(x)$ arises from the geometry. By Step~4, the factor
$(-R_z(x_j))^{m_j}=(\varphi_j(h_j)/h_j)^{m_j}$ is the product of three
contributions:
\begin{enumerate}[label=\textup{(\roman*)},leftmargin=2.2em]
\item the factors
$\bigl((1+s_{ij}-h_j)/(s_{ij}-h_j)\bigr)^{m_j}$ with $i\neq j$, which come
from the ranks of the coefficients of $\zeta_j^{-1}$ in the complexes
$W_{ij}$, $i\neq j$;
\item the factor $(1-h_j)^{m_j}$, which comes from the rank of the
coefficient of $\zeta_j^{-1}$ in the complex $W_{jj}$;
\item the factor $h_j^{-m_j}$, which comes from the rank of the vector
bundle $E_j(k_j)$ through Lemma~\ref{lem:projective-pushforward}.
\end{enumerate}
The descendent insertions enter the right-hand side of
\eqref{eq:bounded-residue-remainder} only through the polynomials
$p_\lambda$ and the functions $f_\lambda$. In particular the function
$R_z$, and therefore the equations $P_z(x)=\xi q\,P_z(x+1)$ whose roots
generate the field $\RF$ $($Definition~\ref{def:root-field}$)$, do not
depend on the insertions.
\end{remark}

\begin{proposition}[Algebraicity of the base-twisted series]\label{prop:curve-field}
Assume \eqref{eq:positivity-hypothesis} and let $s\geq s_0$ be as in
Lemma~\ref{lem:integral-curves}. Let $\RF$ be the root field of
Definition~\ref{def:root-field}. Then
\[
 \YY^{(s)}_{\LL}\in\HB[\eps]\otimes_{\Q}\RF .
\]
\end{proposition}
\begin{proof}
For a smooth connected projective variety $Y$ and a class
$U\in H^*(Y,\Q)$ we write $\int_YU\in\Q$ for the pushforward of $U$ to a
point, that is, for the degree of the component of $U$ in
$H^{2\dim Y}(Y,\Q)$ on the fundamental class of $Y$; this notation is
extended linearly to classes with coefficients in $\Q(z)[\eps]$ or in
$\Q(z)[\eps][[q]]$. Choose a homogeneous basis $(e_\mu)$ of $\HB$ and the
dual basis $(e^\mu)$ with respect to the Poincar\'e pairing, so that
$U=\sum_\mu e_\mu\int_BU\,e^\mu$ for $U\in\HB$. It suffices to show that
$\int_B\YY^{(s)}_{\LL}\,e^\mu\in\Q[\eps]\otimes\RF$ for every $\mu$.

By Definition~\ref{def:type}, the series $\YY^{(s)}_{\LL}$ is the sum, over
the finitely many types, of the partial series formed by the summands of
the length vectors of each type. Fix a type, with index set $I$, starting
lengths $m_{i,*}$ $(i\in I)$ and fixed lengths $m_j$ $(j\in I^c)$, and let
$g:\cT\to B$ be the structure map. Let $\boldsymbol k\in\Z_{\geq0}^I$ and
let $\mm$ be the length vector of this type with
$m_j=m_{j,*}+k_j\nu_j$ for $j\in I$. The projection
$p:\PP_{\boldsymbol k}\to B$ is the composition $g\circ p_{\cT}$, and
$\eta_{\cT}$ is a class on $\cT$. Therefore, by
Corollary~\ref{cor:euler-curves}, \eqref{eq:fixed-base-integral},
Proposition~\ref{prop:residue-form} and the projection formula for
$p_{\cT}$ and for $g$, the coefficient of $q^{|\mm|}$ contributed by $\mm$ to
$\int_B\YY^{(s)}_{\LL}\,e^\mu$ is
\[
 \sum_{\lambda=1}^{\lambda_0}c_{\lambda,\mu}\,
 p_\lambda(\boldsymbol k)\Res_{x_j=w_jz,\,j\in I}
 \Bigl(f_\lambda\prod_{j\in I}R_z(x_j)^{m_{j,*}+k_j\nu_j}
 \prod_{j\in I}dx_j\Bigr),
 \qquad
 c_{\lambda,\mu}=\int_{\cT}\gamma_\lambda\,\eta_{\cT}\,g^*e^\mu\ \in\Q .
\]
Here $\lambda_0$ and the triples $(p_\lambda,\gamma_\lambda,f_\lambda)$
are those of Proposition~\ref{prop:residue-form}, and the residue is that
of Definition~\ref{def:local-residues}. Summing over $\boldsymbol k$, we
find that the partial series of the type, paired with $e^\mu$, is
\begin{multline*}
 q^{\sum_{j\in I^c}m_j}\sum_{\lambda=1}^{\lambda_0}c_{\lambda,\mu}
 \sum_{\boldsymbol k\in\Z_{\geq0}^I}
 q^{\sum_{j\in I}(m_{j,*}+k_j\nu_j)}\,p_\lambda(\boldsymbol k)\\
 \times\Res_{x_j=w_jz,\,j\in I}
 \Bigl(f_\lambda\prod_{j\in I}R_z(x_j)^{m_{j,*}+k_j\nu_j}
 \prod_{j\in I}dx_j\Bigr).
\end{multline*}
Since $k_j=(m_j-m_{j,*})/\nu_j$, the polynomial $p_\lambda(\boldsymbol k)$
is a polynomial in the integers $m_j=m_{j,*}+k_j\nu_j$, $j\in I$, and the
sum over $\boldsymbol k$ is the sum over all $(m_j)_{j\in I}$ with
$m_j\geq m_{j,*}$ and $m_j\equiv m_{j,*}$ modulo $\nu_j$. For each
$\lambda$ the inner sum is therefore a series of the form
\eqref{eq:residue-series}, so it lies in $\Q[\eps]\otimes\RF$ by
Proposition~\ref{prop:residue-field}. (If $I=\varnothing$ the partial
series is a single monomial in $q$ with coefficient in $\Q(z)[\eps]$.)
\end{proof}

\section{Proof of the main theorem}\label{sec:main-proof}

\begin{theorem}[The Hilbert series lies in the root field]\label{thm:hilbert-field}
Let $B$ be a smooth connected projective variety, and let
$L_1,\ldots,L_N$ be arbitrary line bundles on $S\times B$. Let $\YY_{\LL}$
be the Hilbert series of Definition~\ref{def:hilbert-series}, and let
$\RF\subset\Qb(z)((q))$ be the root field of
Definition~\ref{def:root-field}, that is, the field generated over
$\Qb(q,z)$ by the roots $b_{i,\xi}(q)$ of the polynomials
$P_z(x)-\xi q\,P_z(x+1)$, where $\xi$ runs over the roots of unity. Then
\[
 \YY_{\LL}\in\HB[\eps]\otimes_{\Q}\RF .
\]
\end{theorem}
\begin{proof}
Suppose first that the $M_i=K_S\otimes L_i^{-1}$ satisfy the positivity
hypothesis \eqref{eq:positivity-hypothesis}. Choose $\dim B+1$ distinct
integers $s$, all at least as large as the bound of
Lemma~\ref{lem:integral-curves}. For each of them $\YY^{(s)}_{\LL}$ lies in
$\HB[\eps]\otimes\RF$ by Proposition~\ref{prop:curve-field}, and
$\YY_{\LL}$ is a rational linear combination of them by
\eqref{eq:base-interpolation-explicit}.

In general, Lemma~\ref{lem:positivity} gives an integer $k_0$ such that
the line bundles $M_i(k)=M_i\otimes\OO_S(kH)$ satisfy
\eqref{eq:positivity-hypothesis} for all $k\geq k_0$ and all $i$. By the
first part, $\YY(k)=\YY_{\LL(k)}\in\HB[\eps]\otimes\RF$ for $k\geq k_0$, and
Corollary~\ref{cor:positive-twist-reduction} gives the claim for
$\YY_{\LL}=\YY(0)$.
\end{proof}

\begin{proof}[Proof of Theorem~\ref{thm:main}]
Consider the equivariant series
\[
 \cJ(q,z)=\sum_{n\in\Z}q^nJ_n(z),
\]
with $J_n(z)$ as in Proposition~\ref{prop:equivariant-coefficient}. By
Lemma~\ref{lem:finite-decompositions} it has only finitely many negative
powers of $q$. By Proposition~\ref{prop:quot-to-hilbert} it is a finite sum
of series of the form
\[
 q^{c}\,a_\lambda\int_{B_\lambda}[\epsilon_1\cdots\epsilon_\ell]\,
 \YY_{\LL_\lambda},\qquad c\in\Z,\ a_\lambda\in\Q,
\]
with $B_\lambda$ smooth, connected and projective. Each of these lies in
$\RF$ by Theorem~\ref{thm:hilbert-field}, because integration over
$B_\lambda$ and extraction of the coefficient of
$\epsilon_1\cdots\epsilon_\ell$ are $\RF$-linear maps
$\mathsf H_{B_\lambda}[\eps]\otimes\RF\to\RF$. Hence $\cJ(q,z)\in\RF$.

By Proposition~\ref{prop:equivariant-coefficient}, every coefficient
$J_n(z)$ lies in $\Q(z)$ and is regular at $z=0$. So
Theorem~\ref{thm:specialization} applies and shows that
$\sum_nJ_n(0)q^n\in\Q(q)$. By \eqref{eq:coefficient-specialization} this
series is $Z_{S,N,\beta,\tau}(q)$.
\end{proof}

\begin{remark}[What the proof does not give]\label{rem:no-formula}
The proof is not effective in any useful sense: the varieties $B_\lambda$
and $\cT$ are produced by resolution of singularities, and the classes
$\gamma_\lambda$ of Proposition~\ref{prop:residue-form} are not computed. It
gives no bound on the degree of the rational function and no information on
its poles beyond what can be read off from the field $\RF$.
\end{remark}

\section{Examples and consistency checks}\label{sec:checks}

Theorem~\ref{thm:main} also covers the cases in which rationality was
already known, some of them with explicit formulas. In this section we
verify that the method reproduces these formulas. We begin with the
example announced in Section~\ref{subsec:mechanism}, then derive a closed
formula for the Hilbert series when all the line bundles are equal and the
base is a point, and deduce the formula of Oprea and Pandharipande for
$\beta=0$. In the last two subsections we compute the Hilbert series
directly on toric surfaces and compare the intermediate statements of the
paper with these computations.

\subsection{A Hilbert series on the projective plane}
\label{subsec:P2-example}
Let $S=\PP^2$, $B$ a point, $N=1$, no insertions, and
\[
 L=\OO_{\PP^2}(-4),\qquad M=K_S\otimes L^{-1}=\OO_{\PP^2}(1).
\]
The Hilbert series of Definition~\ref{def:hilbert-series} is
\begin{equation}\label{eq:P2-auxiliary-series}
 \YY_L(q)=\sum_{m\geq0}a_mq^m,\qquad
 a_m=\int_{S^{[m]}}e\bigl((M^{[m]})^\vee\bigr)\,c(W_m),\qquad
 W_m=R\pi_*\RHom(I_{\cZ}(4),\OO-I_{\cZ}(4)).
\end{equation}
There are no mixed terms when $N=1$.

Choose a section of $M$ vanishing on a line $C\cong\PP^1$. The induced
section of $M^{[m]}$ vanishes exactly on $C^{[m]}\cong\PP^m$, which is
smooth of codimension $m$, so the section is regular and
\begin{equation}\label{eq:P2-euler-reduction}
 a_m=(-1)^m\int_{\PP^m}c(W_m|_{\PP^m}).
\end{equation}
Let $\zeta=\OO_{\PP^m}(1)$ and $h=c_1(\zeta)$. The universal divisor on
$C\times\PP^m$ has line bundle $\OO_C(m)\boxtimes\zeta$, hence
\begin{equation}\label{eq:P2-tautological-restriction}
 M^{[m]}|_{\PP^m}=\chi(\OO_{\PP^1}(1))-\chi(\OO_{\PP^1}(1-m))\,\zeta^{-1}
 =2+(m-2)\,\zeta^{-1}\qquad\text{in }K^0(\PP^m).
\end{equation}
By \eqref{eq:Wzero}, $W^0=\chi(\OO_{\PP^2}(-4))-\chi(\OO_{\PP^2})=2$. The
normal bundle of $\PP^m$ in $S^{[m]}$ is $M^{[m]}|_{\PP^m}$ and
$T_{\PP^m}=(m+1)\zeta-1$, so by \eqref{eq:diagonal-hilbert-class}
\begin{equation}\label{eq:P2-tangent-restriction}
 W_m|_{\PP^m}=2+T_{\PP^m}+M^{[m]}|_{\PP^m}-\bigl(M^{[m]}|_{\PP^m}\bigr)^\vee
 =3\zeta+(m-2)\zeta^{-1}+1 .
\end{equation}
This agrees with Lemma~\ref{lem:rank-calculation}: the ranks
\eqref{eq:rank-table} are $-K_S\cdot C=3$, $m+g-1-C^2=m-2$ and $C^2=1$.
Therefore
\begin{equation}\label{eq:P2-test}
 a_m=(-1)^m\,[h^m]\,(1+h)^3(1-h)^{m-2},
\end{equation}
where $[h^m]$ denotes the coefficient of $h^m$. With $x=-h$ this is
\[
 a_m=\Res_{x=0}\frac{(1-x)^3}{x(1+x)^2}\Bigl(\frac{1+x}{x}\Bigr)^m\,dx ,
\]
which is of the form \eqref{eq:residue-series} with $w_1=0$, $P_z(x)=x$ and
$R_z(x)=(1+x)/x$. In the coordinate $t=x/(1+x)$ the differential in front
of $R_z^m$ becomes $(1-2t)^3t^{-1}(1-t)^{-2}\,dt$, and
Lemma~\ref{lem:positive-part} gives
\begin{equation}\label{eq:P2-generating-function}
 \YY_L(q)=\frac{(1-2q)^3}{(1-q)^2}=1-4q+3q^2+2q^3+q^4-q^6-2q^7-\cdots.
\end{equation}

\subsection{A closed formula for equal line bundles}
In this subsection $B$ is a point, there are no insertions $(\ell=0)$, and
$L_1=\cdots=L_N=L$ for a line bundle $L$ on $S$. Put $M=K_S\otimes L^{-1}$
and write $M^2$ and $K\cdot M$ for the intersection numbers $M\cdot M$ and
$K_S\cdot M$. Let
\[
 x_j=b_{j,1}(q)\in\Q(z)[[q]]\qquad(1\leq j\leq N)
\]
be the roots of $P_z(x)=q\,P_z(x+1)$ from \eqref{eq:lifted-roots}, and
define
\begin{equation}\label{eq:G-series}
\begin{aligned}
 G_1&=\prod_{j=1}^N(x_j-w_jz)\Bigl(\frac{P_z'(x_j)}{P_z(x_j)}
     -\frac{P_z'(x_j+1)}{P_z(x_j+1)}\Bigr),&
 G_2&=q^{-N}\prod_{j=1}^N(x_j-w_jz),\\
 G_3&=\frac{\prod_{i,j}(1+w_jz-x_i)}{\prod_{i\neq j}(w_jz-x_i)},&
 G_4&=\prod_{i\neq j}\frac{1+x_j-x_i}{x_j-x_i}.
\end{aligned}
\end{equation}
These are elements of $\RF\cap\Q(z)[[q]]$. Since
$x_j-w_jz=q\,P_z(w_jz+1)/P_z'(w_jz)+O(q^2)$, their constant terms are
\begin{equation}\label{eq:G-constant-terms}
 G_1(0)=1,\qquad G_2(0)=G_3(0)=G_4(0)=\varPi_0:=
 \prod_{i\neq j}\frac{1+(w_j-w_i)z}{(w_j-w_i)z}.
\end{equation}

\begin{proposition}[Closed formula]\label{prop:closed-formula}
For every smooth projective surface $S$ and every line bundle $L$ on $S$,
with $L_1=\cdots=L_N=L$, $B$ a point and no insertions,
\begin{equation}\label{eq:closed-formula}
 \YY_{\LL}=G_1^{\,(M^2+K\cdot M)/2}\ G_2^{\,(K\cdot M-M^2)/2}\
           G_3^{\,-K\cdot M}\ G_4^{\,M^2}.
\end{equation}
\end{proposition}

For $N=1$ and $w_1=0$ one has $x_1=q/(1-q)$, $G_1=1-q$, $G_2=1/(1-q)$,
$G_3=(1-2q)/(1-q)$ and $G_4=1$, so \eqref{eq:closed-formula} reads
\begin{equation}\label{eq:closed-N1}
 \YY_L=(1-q)^{M^2+K\cdot M}\,(1-2q)^{-K\cdot M}\qquad(N=1),
\end{equation}
which for $S=\PP^2$, $M=\OO(1)$ is \eqref{eq:P2-generating-function}.

The proof uses the following form of Lagrange inversion.

\begin{lemma}[Lagrange summation]\label{lem:lagrange}
Let $\mathbb{k}$ be a field of characteristic zero, let
$\varphi_1,\ldots,\varphi_N\in\mathbb{k}[[h]]$ have nonzero constant terms,
and let $\Psi\in\mathbb{k}[[h_1,\ldots,h_N]]$. Let $H_j\in q\,\mathbb{k}[[q]]$
be the unique solution of $H_j=-q\,\varphi_j(H_j)$. Then
\begin{equation}\label{eq:lagrange}
 \sum_{\mm\in\Z_{\geq0}^N}q^{|\mm|}\,\Bigl[\prod_jh_j^{m_j}\Bigr]
 \Bigl(\Psi\prod_{j=1}^N\bigl(-\varphi_j(h_j)\bigr)^{m_j}\Bigr)
 =\frac{\Psi(H_1,\ldots,H_N)}
       {\prod_{j=1}^N\bigl(1-H_j\varphi_j'(H_j)/\varphi_j(H_j)\bigr)} .
\end{equation}
\end{lemma}
\begin{proof}
Consider first one variable. Put $t=-h/\varphi(h)$, a local coordinate at
$h=0$, with inverse series $h=h(t)$. Then
\[
 [h^m]\bigl(\Psi(-\varphi)^m\bigr)
 =\Res_{h=0}\Psi(h)\Bigl(\frac{-\varphi(h)}{h}\Bigr)^m\frac{dh}{h}
 =\Res_{t=0}\Psi(h(t))\,\frac{t\,h'(t)}{h(t)}\,t^{-m-1}\,dt .
\]
The function $\widetilde\Psi(t)=\Psi(h(t))\,t\,h'(t)/h(t)$ is a power
series, so by Lemma~\ref{lem:positive-part} the sum over $m\geq0$ of $q^m$
times this residue is $\widetilde\Psi(q)$. Now $h(q)$ is the solution $H$
of $q=-H/\varphi(H)$, and
$t\,h'(t)/h(t)=(d\log t/d\log h)^{-1}=(1-h\varphi'(h)/\varphi(h))^{-1}$. For
several variables, apply this to one variable at a time.
\end{proof}

\begin{proof}[Proof of Proposition~\ref{prop:closed-formula}]
\emph{Step 1: reduction to four pairs $(S,M)$.} We show that each side
of \eqref{eq:closed-formula} can be written in the form
\begin{equation}\label{eq:closed-formula-form}
 \varPi_0^{\,(M^2-K\cdot M)/2}\,
 \exp\bigl(M^2A_1+(K\cdot M)A_2+K_S^2A_3+c_2(S)A_4\bigr)
\end{equation}
for series $A_1,\ldots,A_4\in q\,\Q(z)[[q]]$ that do not depend on $S$ and
$L$. For the left-hand side, note that for $B$ a point and equal $L_i$ the
classes $\kappa_P$ of Definition~\ref{def:kappa} are the numbers $L^2$,
$L\cdot K_S$, $K_S^2$ and $c_2(S)$, or equivalently $M^2$, $K\cdot M$,
$K_S^2$ and $c_2(S)$. By Lemma~\ref{lem:linear-log} the logarithm of the
normalized series $\widehat\YY_{\LL}=\YY_{\LL}/\YY^0_{\LL}$ is a linear
form in these four numbers with coefficients in $q\,\Q(z)[[q]]$ that do
not depend on $S$ and $L$. The constant term is
$\YY^0_{\LL}=\varPi_0^{\,\chi(L)-\chi(\OO_S)}=
\varPi_0^{\,(M^2-K\cdot M)/2}$, because all $W^0_{ij}$ equal the virtual
vector space $R\Gamma(L)-R\Gamma(\OO_S)$. For the right-hand side, the
four exponents in \eqref{eq:closed-formula} are linear forms in
$(M^2,K\cdot M)$, and by \eqref{eq:G-constant-terms} the constant term is
$\varPi_0$ raised to the sum of the exponents of $G_2$, $G_3$ and $G_4$,
which is $(M^2-K\cdot M)/2$. So the right-hand side has the form
\eqref{eq:closed-formula-form} with $A_3=A_4=0$.

Consequently \eqref{eq:closed-formula} holds for all $(S,L)$ if the four
series $A_1,\ldots,A_4$ of the two sides coincide. This is the case as
soon as \eqref{eq:closed-formula} holds for four pairs $(S,M)$ for which
the vectors $(M^2,K\cdot M,K_S^2,c_2(S))\in\Q^4$ are linearly
independent.

\emph{Step 2: the case of a smooth rational curve.} Suppose that $M$ has a
section whose zero locus $C$ is a smooth rational curve; then
$M^2+K\cdot M=-2$. As in Section~\ref{subsec:P2-example}, the induced
section of $M^{[m]}$ is regular with zero scheme $C^{[m]}\cong\PP^m$, and on
$C\times\PP^m$ the universal ideal is $\OO_C(-m)\boxtimes\zeta^{-1}$. We use
the same curve $C$ for all $i$. On $S\times\prod_i\PP^{m_i}$ the kernels
then have the form \eqref{eq:kernel-on-curve}, with $\Lambda=K_S^{-1}$ and
$\cA_i=\iota_*(\OO_C(-m_i)\otimes L^{-1}|_C)$, and the computation of
Section~\ref{sec:residue-geometry} applies with a point in place of $\cT$:
the coefficient complexes in \eqref{eq:W-decomposition} are virtual vector
spaces, of dimensions $0$, $-K\cdot M$, $m_j-1-M^2$ and $M^2$ by
\eqref{eq:rank-table}. As in Corollary~\ref{cor:euler-curves}, and by
Lemma~\ref{lem:weight-separation}, the coefficient of $q^{|\mm|}$
contributed by $\mm$ is
\begin{multline*}
 (-1)^{|\mm|}\Bigl[\prod_jh_j^{m_j}\Bigr]\ \prod_j(1+h_j)^{-K\cdot M}
 (1-h_j)^{m_j-1-M^2}\\
 \times\prod_{i\neq j}
 \Bigl(\frac{1+s_{ij}+h_i}{s_{ij}+h_i}\Bigr)^{-K\cdot M}
 \Bigl(\frac{1+s_{ij}-h_j}{s_{ij}-h_j}\Bigr)^{m_j-1-M^2}
 \Bigl(\frac{1+s_{ij}+h_i-h_j}{s_{ij}+h_i-h_j}\Bigr)^{M^2},
\end{multline*}
where $s_{ij}=(w_j-w_i)z$. With $\varphi_j$ the function
\eqref{eq:length-factors}, this is
$[\prod_jh_j^{m_j}]\bigl(\Psi\prod_j(-\varphi_j(h_j))^{m_j}\bigr)$ for
\[
 \Psi=\prod_j\varphi_j(h_j)^{-1-M^2}(1+h_j)^{-K\cdot M}\prod_{i\neq j}
 \Bigl(\frac{1+s_{ij}+h_i}{s_{ij}+h_i}\Bigr)^{-K\cdot M}
 \Bigl(\frac{1+s_{ij}+h_i-h_j}{s_{ij}+h_i-h_j}\Bigr)^{M^2}.
\]
By Lemma~\ref{lem:lagrange}, $\YY_{\LL}$ is the right-hand side of
\eqref{eq:lagrange}. By \eqref{eq:common-ratio}, the equation
$H_j=-q\,\varphi_j(H_j)$ says that $x=w_jz-H_j$ satisfies $R_z(x)=1/q$,
that is $P_z(x)=q\,P_z(x+1)$, with $x=w_jz$ at $q=0$; so $x=x_j$ and
$H_j=w_jz-x_j$. Substituting,
\[
 \varphi_j(H_j)=-\frac{H_j}{q}=\frac{x_j-w_jz}{q},\quad
 1+H_j=1+w_jz-x_j,\quad s_{ij}+H_i=w_jz-x_i,\quad
 s_{ij}+H_i-H_j=x_j-x_i,
\]
and $1-H_j\varphi_j'(H_j)/\varphi_j(H_j)$ is the logarithmic derivative of
$t=P_z(x)/P_z(x+1)$ with respect to $\log(w_jz-x)$, namely the $j$th
factor of $G_1$. Hence
\[
 \YY_{\LL}=G_1^{-1}\,G_2^{\,-1-M^2}\,G_3^{\,-K\cdot M}\,G_4^{\,M^2},
\]
which is \eqref{eq:closed-formula} because $M^2+K\cdot M=-2$.

\emph{Step 3: four geometries.} Step~2 applies to:
a line and a conic in $\PP^2$; a fiber of a ruling of
$\PP^1\times\PP^1$; and a fiber of the projection $E\times\PP^1\to E$ for an
elliptic curve $E$. The vectors $(M^2,K\cdot M,K_S^2,c_2(S))$ are
\[
 (1,-3,9,3),\qquad(4,-6,9,3),\qquad(0,-2,8,4),\qquad(0,-2,0,0),
\]
and their determinant is $72\neq0$.
\end{proof}

\begin{corollary}[Zero curve class: the formula of Oprea and Pandharipande]
\label{cor:OP}
Let $\beta=0$ and take no insertions. For every smooth projective surface
$S$ and every $N\geq1$,
\begin{equation}\label{eq:equivariant-punctual}
 \sum_{n\geq0}q^nJ_n(z)=\Bigl(\frac{G_1\,G_4}{G_3}\Bigr)^{K_S^2},
\end{equation}
and
\begin{equation}\label{eq:OP-formula}
 \sum_{n\geq0}q^n\intvir{\Quot_S(\OO_S^{\oplus N},0,n)}c(T^{\vir})
 =\mathsf U_N^{\,K_S^2},\qquad
 \mathsf U_N=\frac{(1-q)^{2N}}{(1-2^Nq)^N}\prod_{i<j}\bigl(1-(r_i-r_j)^2\bigr),
\end{equation}
where $r_1,\ldots,r_N$ are the roots of $r^N=q\,(r-1)^N$. This is
\cite[Theorem~18 and equation~(9)]{OP}.
\end{corollary}
\begin{proof}
For $\beta=0$ the only decomposition has all $\beta_i=0$; the divisor
Hilbert schemes are points and $L_i=\OO_S$. By
Proposition~\ref{prop:quot-to-hilbert}, $\sum_nq^nJ_n(z)=\YY_{\LL}$ with
$L=\OO_S$ and $M=K_S$, and \eqref{eq:closed-formula} gives
\eqref{eq:equivariant-punctual}.

To set $z=0$ we use Proposition~\ref{prop:compute-specialization} with
$E=(G_1G_4/G_3)^{K_S^2}$, regarded as a rational function of
$x_1,\ldots,x_N$, $q$ and $z$ via \eqref{eq:G-series}. Let
$\rho_j=a_j/(1-a_j)$, where $a_1,\ldots,a_N$ are the $N$th roots of $q$ in
$\Qb(t)$, $t^N=q$. At $z=0$ we have $P_0(x)=x^N$, and the three factors
become the symmetric functions
\[
 G_1\mapsto\prod_j\frac N{1+\rho_j},\qquad
 G_3\mapsto\frac{\prod_i(1-\rho_i)^N}{\prod_i(-\rho_i)^{N-1}},\qquad
 G_4\mapsto\prod_{i<j}\frac{1-(\rho_i-\rho_j)^2}{-(\rho_i-\rho_j)^2}.
\]
All numerators and denominators here are nonzero elements of $\Qb(t)$;
they are the specializations of the numerators and denominators in
\eqref{eq:G-series}, which become products of the elements $\rho_j$,
$1\pm\rho_j$, $\rho_j-\rho_i$ and $1+\rho_j-\rho_i$. So the hypotheses of
Proposition~\ref{prop:compute-specialization} hold. To
evaluate, use $1+\rho_j=1/(1-a_j)$, $1-\rho_j=(1-2a_j)/(1-a_j)$,
$\rho_i-\rho_j=(a_i-a_j)/((1-a_i)(1-a_j))$ and
\begin{gather*}
 \prod_j(1-a_j)=1-q,\qquad\prod_j(1-2a_j)=1-2^Nq,\qquad
 \prod_ja_j=(-1)^{N-1}q,\\
 \prod_{i<j}(a_i-a_j)^2=(-1)^{\binom N2}N^N(-q)^{N-1},
\end{gather*}
the last being the discriminant of $X^N-q$. One finds
\[
 G_1\mapsto N^N(1-q),\qquad
 G_3\mapsto\frac{(1-2^Nq)^N}{(1-q)\,(-q)^{N-1}},\qquad
 G_4\mapsto\frac{(1-q)^{2N-2}}{N^N(-q)^{N-1}}\prod_{i<j}\bigl(1-(\rho_i-\rho_j)^2\bigr),
\]
so that $G_1G_4/G_3\mapsto(1-q)^{2N}(1-2^Nq)^{-N}\prod_{i<j}(1-(\rho_i-
\rho_j)^2)$. Finally $r_j=-\rho_j$ are the roots of $r^N=q(r-1)^N$. By
\eqref{eq:coefficient-specialization} the specialization at $z=0$ is the
left-hand side of \eqref{eq:OP-formula}.
\end{proof}

\begin{remark}
For $N=1$, \eqref{eq:OP-formula} is
$\bigl((1-q)^2/(1-2q)\bigr)^{K_S^2}$. For $N=2$ the roots are
$r=\pm\sqrt q/(\pm\sqrt q-1)$ and
$\mathsf U_2=(1-q)^2(1-6q+q^2)/(1-4q)^2$. Formula
\eqref{eq:equivariant-punctual} is an equivariant refinement, and it shows
concretely how the irrationalities of $\RF$ cancel at $z=0$.
\end{remark}

\subsection{Computations on toric surfaces}\label{subsec:computer}
When $S$ is a toric surface, the coefficients of the Hilbert series can be
computed directly by localization with respect to the torus of $S$. This
computation is independent of
Sections~\ref{sec:twisting}--\ref{sec:residue-geometry}. In this subsection
we describe it and use it to give explicit values of the invariants of
Theorem~\ref{thm:main} in a case covered by Corollary~\ref{cor:new-case}.
In Section~\ref{subsec:verification} we compare it with the intermediate
statements of the paper.

Let $S$ be a smooth projective toric surface with torus $\TT_S=(\C^*)^2$
and fixed points $p_\sigma$. Let $\mathsf R$ be the representation ring of
$\TT_S$, with the involution $V\mapsto\overline V$ that sends a
representation to its dual, and let $u_\sigma,v_\sigma\in\mathsf R$ be the
two characters of $\TT_S$ on the cotangent space of $S$ at $p_\sigma$. For
$\TT_S$-equivariant perfect complexes $E$ and $F$ on $S$ we write
$E|_\sigma\in\mathsf R$ for the class of the derived fiber of $E$ at
$p_\sigma$, and
\[
 \chi(E,F)=\sum_k(-1)^k\Ext^k(E,F)\in\mathsf R .
\]
By the localization theorem in equivariant $K$-theory \cite{Thomason},
\begin{equation}\label{eq:toric-chi}
 \chi(E,F)=\sum_\sigma\frac{\overline{E|_\sigma}\cdot F|_\sigma}
 {(1-u_\sigma)(1-v_\sigma)} .
\end{equation}
The fixed points of $\TT_S$ on $S^{[m]}$ are the subschemes $Z$ whose
ideals are generated by monomials in the coordinates at the points
$p_\sigma$. They correspond to the families of partitions
$(\lambda_\sigma)_\sigma$ with $\sum_\sigma|\lambda_\sigma|=m$ \cite{ES}. For
such a $Z$ let
\[
 \mathsf Q_\sigma(Z)=\sum_{(a,b)\in\lambda_\sigma}u_\sigma^{a}v_\sigma^{b}
 \ \in\mathsf R
\]
be the character of $\OO_Z$ at $p_\sigma$, where $(a,b)$ runs over the boxes
of the Young diagram of $\lambda_\sigma$, with $a,b\geq0$. Then
$I_Z|_\sigma=1-(1-u_\sigma)(1-v_\sigma)\,\mathsf Q_\sigma(Z)$.

Let $D_1,\ldots,D_N$ be $\TT_S$-invariant divisors on $S$, let $B$ be a
point, and put $M_i=\OO_S(D_i)$ and $L_i=K_S\otimes M_i^{-1}$, with their
natural $\TT_S$-linearizations. Let $Z=(Z_1,\ldots,Z_N)$ be a fixed point
of $\TT_S$ on $X_{\mm}=\prod_iS^{[m_i]}$, and put
$K_i=I_{Z_i}\otimes L_i^{-1}$, a $\TT_S$-equivariant sheaf on $S$. By base
change, the fibers at $Z$ of the classes that enter
Definition~\ref{def:hilbert-series} are the following elements of
$\mathsf R$, which are computed by \eqref{eq:toric-chi}:
\begin{equation}\label{eq:toric-fibers}
\begin{gathered}
 W_{ij}|_Z=\chi(K_i,\OO_S)-\chi(K_i,K_j),\qquad
 K_i|_\sigma=\overline{L_i|_\sigma}\,
 \bigl(1-(1-u_\sigma)(1-v_\sigma)\,\mathsf Q_\sigma(Z_i)\bigr),\\
 M_i^{[m_i]}|_Z=\sum_\sigma M_i|_\sigma\,\mathsf Q_\sigma(Z_i),\qquad
 T_{Z_i}S^{[m_i]}=\chi(\OO_S,\OO_S)-\chi(I_{Z_i},I_{Z_i}).
\end{gathered}
\end{equation}
The last formula is \cite[Proposition~2.2]{EGL}.

Choose a one-parameter subgroup $\C^*\to\TT_S$ that has the same fixed
points on the varieties $S^{[m_i]}$ as $\TT_S$, and let $\varepsilon$ be its
equivariant parameter. A character $\psi$ of $\TT_S$ restricts to a
character of $\C^*$, whose weight we denote by $\langle\psi\rangle\in\Z$.
For $V=\sum_\psi n_\psi\,\psi\in\mathsf R$ and a scalar $c$ put
\[
 \mathsf P_c(V)=\prod_\psi\bigl(c+\langle\psi\rangle\,\varepsilon\bigr)^{n_\psi}.
\]
If $V$ is the fiber at a fixed point of an equivariant $K$-theory class,
then $\mathsf P_0(V)$, $\mathsf P_1(V)$ and
$\mathsf P_{1+wz}(V)/\mathsf P_{wz}(V)$ are the restrictions to that point
of its equivariant Euler class, of its equivariant total Chern class, and of
the equivariant lift of the class $\MM_w$ of
\eqref{eq:chern-euler-ratio}. By the localization formula in equivariant
cohomology \cite{AB}, applied to the equivariant lift of the integrand of
\eqref{eq:hilbert-series}, the coefficient of $q^{|\mm|}$ contributed by
$\mm$ to $\YY_{\LL}$ is, for $\ell=0$, the value at $\varepsilon=0$ of
\begin{equation}\label{eq:toric-bott}
 \sum_{Z}\ \prod_{i=1}^N
 \frac{\mathsf P_0\bigl(\overline{M_i^{[m_i]}|_Z}\bigr)\,
       \mathsf P_1\bigl(W_{ii}|_Z\bigr)}
      {\mathsf P_0\bigl(T_{Z_i}S^{[m_i]}\bigr)}\
 \prod_{i\neq j}\frac{\mathsf P_{1+(w_j-w_i)z}\bigl(W_{ij}|_Z\bigr)}
                     {\mathsf P_{(w_j-w_i)z}\bigl(W_{ij}|_Z\bigr)} ,
\end{equation}
where the sum runs over the fixed points $Z=(Z_1,\ldots,Z_N)$ of $\TT_S$ on
$X_{\mm}$. The sum \eqref{eq:toric-bott} is a polynomial in $\varepsilon$
with coefficients in $\Q(z)$, because it is the integral of an equivariant
class; its individual terms have poles at $\varepsilon=0$.

Two extensions are used below. First, for $\ell>0$ and
$\TT_S$-equivariant classes $\alpha_a$, the factor $\Theta_{\mm}$ is
included in \eqref{eq:toric-bott} in the same way: the fiber of
$R\pi_*(Q_i\otimes\alpha_a)$ at $Z$ is
$\chi(\OO_S,\alpha_a)-\chi(\OO_S,K_i\otimes\alpha_a)$. Second, let
$B=\PP^d$ and $M_i=\OO_S(D_i)\boxtimes\OO_B(e_i)$. Then the same formula
holds for the torus $\TT_S\times(\C^*)^{d+1}$ acting on
$X_{\mm}=\prod_iS^{[m_i]}\times B$. The sum runs over the pairs of a fixed
point $Z$ and a fixed point $b$ of $B$, the classes $L_i|_\sigma$ and
$M_i|_\sigma$ are multiplied by the characters of the fibers of
$\OO_B(-e_i)$ and $\OO_B(e_i)$ at $b$, and the denominator contains the
additional factor $\mathsf P_0(T_bB)$. The component in $H^{2j}(B,\Q)$ of a
coefficient of $\YY_{\LL}$ is obtained by multiplying the integrand by
$c_1(\OO_B(1))^{d-j}$.

\begin{example}[A case of Corollary~\ref{cor:new-case}]\label{ex:new-case}
Let $S=\PP^2$, let $N=2$, and let $\beta$ be the class of a line. The torus
$\TT_S\times\TT$ acts on $\Qn$ with isolated fixed points. By
Proposition~\ref{prop:fixed-loci} these are the quotients with kernel
$I_{Z_1}(-D_1)\oplus I_{Z_2}(-D_2)$, where $Z_1$ and $Z_2$ are defined by
monomial ideals and one of the divisors $D_1$, $D_2$ is zero and the other
is a coordinate line. The virtual localization formula \cite{GP} for
$\TT_S\times\TT$ expresses the invariants as sums over these fixed points,
which are evaluated as in \eqref{eq:toric-bott}, with $K_i=I_{Z_i}(-D_i)$
and with the restriction $\sum_{i,j}\chi(K_i,\OO_S-K_j)[w_j-w_i]$ of the
virtual tangent complex. One finds
\begin{align*}
 Z_{\PP^2,2,\beta}(q)
 &=6q-42q^2-1458q^3-10464q^4-134274q^5+O(q^6),\\
 Z_{\PP^2,2,\beta}(\OO_S\mid1)(q)
 &=3q-18q^2+531q^3-16020q^4-269997q^5+O(q^6),
\end{align*}
where the first series has no insertions and the second has the insertion
$\ch_1(\OO_S^{[n]})$. The coefficient $6$ is the topological Euler
characteristic of $\Quot_S(\OO_S^{\oplus2},\beta,1)\cong\PP^1\times\PP^2$,
which is smooth of the expected dimension $3$. The same computation with
the equivariant parameter of $\TT$ specialized to a rational number gives
values of the functions $J_n(z)$ of
Proposition~\ref{prop:equivariant-coefficient}; for $(w_1,w_2)=(0,1)$ one
finds $J_2(1/10)=-468/11$.
\end{example}

\begin{example}[Universal coefficients and the mixed K\"unneth component]
\label{ex:kunneth}
Let $N=1$ and $\ell=0$, and write $\mu=c_1(M)$ and $t=c_1(T_S)$ for these
classes on $S\times B$. By Lemma~\ref{lem:linear-log} there are rational
numbers $A_1,\ldots,A_4$, which do not depend on $S$, $B$ and $L$, such
that the component in $H^2(B,\Q)$ of the coefficient of $q^2$ in
$\log\widehat\YY_L$ is
\begin{equation}\label{eq:kunneth-form}
 A_1\,\pi_*(\mu^3)+A_2\,\pi_*(\mu^2t)+A_3\,\pi_*(\mu\,t^2)
 +A_4\,\pi_*\bigl(\mu\,c_2(T_S)\bigr).
\end{equation}
The classes $\pi_*(t^3)$ and $\pi_*(t\,c_2(T_S))$ do not appear, because
$t^3$ and $t\,c_2(T_S)$ are pulled back from $S$ and therefore vanish.
Evaluating \eqref{eq:toric-bott} for $B=\PP^1$, for
$S\in\{\PP^2,\PP^1\times\PP^1\}$ and for four line bundles $M$ for which
the vectors formed by the four classes in \eqref{eq:kunneth-form} are
linearly independent gives
\begin{equation}\label{eq:kunneth-constants}
 (A_1,A_2,A_3,A_4)=(-2,-3,0,-2).
\end{equation}
Four further line bundles on these surfaces give values in agreement with
\eqref{eq:kunneth-constants}.

Formula \eqref{eq:kunneth-form} can be compared with a direct computation
on a surface that is not toric and has odd cohomology. Let
$S=E\times\PP^1$ and $B=E'$ for elliptic curves $E$ and $E'$. Let $e$, $h$
and $b$ be the pullbacks to $S\times B$ of the classes of a point of $E$,
of $\PP^1$ and of $E'$, and let $\gamma\in H^1(E,\Q)\otimes H^1(E',\Q)$ be
the first Chern class of a Poincar\'e line bundle on $E\times E'$, so that
$\gamma^2=-2\,e\,b$. Let $M$ be a line bundle on $S\times B$ with
$c_1(M)=d\,e+f\,h+g\,b+p\,\gamma$ for integers $d,f,g,p$. The coefficient
$p$ enters $\pi_*(\mu^3)$ and $\pi_*(\mu^2t)$ through $\gamma^2$. The surface
$S$ has no torus action with isolated fixed points. The coefficients of $q$
and $q^2$ of $\widehat\YY_L$ were computed from the isomorphism
$S^{[2]}\cong\operatorname{Bl}_\Delta(S\times S)/(\Z/2)$, in the cohomology
ring of $\operatorname{Bl}_\Delta(S\times S)\times B$ with its odd classes.
For eight line bundles $M$, seven of them with $p\neq0$, the result agrees
with \eqref{eq:kunneth-form} and \eqref{eq:kunneth-constants}.
\end{example}

\subsection{Numerical verification of the intermediate statements}
\label{subsec:verification}
The statements listed in Table~\ref{tab:verification} were compared with
the localization formula \eqref{eq:toric-bott} in examples. In each
comparison the coefficient of the Hilbert series contributed by a length
vector $\mm$ is computed in two ways: by \eqref{eq:toric-bott}, which does
not involve curves, and by the statement in question. The second
computation is carried out as follows.
\begin{itemize}[leftmargin=1.6em,itemsep=2pt]
\item \emph{Rational curves.} If the curves $C_i$ are smooth and rational,
then $\cT$ is a point for every type, and the right-hand side of
\eqref{eq:bounded-residue-remainder} is evaluated as in Step~2 of the proof
of Proposition~\ref{prop:closed-formula}.
\item \emph{Smooth curves of positive genus.} For a smooth curve $C$ of
genus $g$ the Hilbert scheme $C^{[m]}$ is the symmetric product of $C$. The
Chern characters of the bundles $E_i(k)$ and of the coefficient complexes of
Lemma~\ref{lem:rank-calculation} on $\cT$ are computed by
Grothendieck--Riemann--Roch from the class of the universal divisor, and
the integrals over $\cT$ are evaluated with the intersection theory of
symmetric products \cite[Chapters~VII and~VIII]{ACGH}. Each curve is given
every type of Definition~\ref{def:type} that is compatible with its
length.
\item \emph{Singular cubics.} For a nodal or cuspidal plane cubic $C$ and
lengths $m=1+3k$, the scheme $\cT_0=C^{[1]}=C$ is singular, $\cT=\PP^1$ is
its normalization, and $G_*$ is the dual of the ideal of the graph of the
normalization, which is not locally free. The Chern characters on $\cT$ are
computed by Grothendieck--Riemann--Roch on $\PP^2\times\cT$, with the class
of $\iota_*G_*$ taken from \eqref{eq:relative-duality-matrix}.
\end{itemize}
The computations were carried out in exact rational arithmetic, for fixed
integer weights $w_i$ and fixed rational values of $z$. In every case the
two computations give the same value.

\begin{table}[ht]
\centering
\small
\begin{tabular}{@{}>{\raggedright\arraybackslash}p{0.28\textwidth}>{\raggedright\arraybackslash}p{0.32\textwidth}>{\raggedright\arraybackslash}p{0.12\textwidth}>{\raggedright\arraybackslash}p{0.20\textwidth}@{}}
\toprule
Statements & Surface $S$ and curves $C_i$, or line bundles $M_i$ & Base $B$ & Range\\
\midrule
Corollary~\ref{cor:euler-curves}, Lemma~\ref{lem:rank-calculation},
Proposition~\ref{prop:residue-form}
 & $\PP^2$: lines and conics & point & $N\leq3$, $|\mm|\leq3$\\\addlinespace
The same, with insertions $\ch_d$, $d\leq3$
 & $\PP^2$, $\PP^1\times\PP^1$: rational curves & point
 & $N\leq3$, $|\mm|\leq3$\\\addlinespace
The same, with base twists $s\leq2$
 & $\PP^2$: families of lines & $\PP^1$ & $N\leq3$, $|\mm|\leq3$\\\addlinespace
Lemma~\ref{lem:abel-bundle}, Proposition~\ref{prop:fixed-base}
 & $\PP^2$: smooth cubics & point & $N\leq3$, $m_i\in\{1,4,7\}$\\\addlinespace
Lemmas~\ref{lem:abel-bundle} and~\ref{lem:brauer-class},
Proposition~\ref{prop:fixed-base}, Lemma~\ref{lem:rank-calculation}
 & $\PP^1\times\PP^1$: curves of bidegree $(2,3)$; $\PP^2$: quartics;
   also together with curves of lower genus
 & point, $\PP^1$ & $N\leq3$; $m\leq10$ for $N=1$\\\addlinespace
Lemmas~\ref{lem:relative-curve-duality} and~\ref{lem:abel-bundle},
Proposition~\ref{prop:fixed-base}
 & $\PP^2$: nodal and cuspidal cubics & point & $N\leq3$, $m_i=1+3k$\\\addlinespace
Lemma~\ref{lem:linear-log}
 & $\PP^2$, $\PP^1\times\PP^1$: $M_i=\OO_S(D_i)\boxtimes\OO_B(e_i)$
 & $\PP^1$, $\PP^2$ & $N\leq2$, $|\mm|\leq3$\\\addlinespace
Proposition~\ref{prop:twisting}
 & $\PP^2$, $\PP^1\times\PP^1$: $M_i=\OO_S(D_i)\boxtimes\OO_B(e_i)$
 & point, $\PP^1$, $\PP^2$ & $N\leq2$, $|\mm|\leq4$\\\addlinespace
Lemma~\ref{lem:base-interpolation}
 & $\PP^2$, $\PP^1\times\PP^1$: $M_i=\OO_S(D_i)\boxtimes\OO_B(e_i)$
 & $\PP^1$, $\PP^2$ & $N\leq2$, $|\mm|\leq3$\\\addlinespace
Proposition~\ref{prop:closed-formula}
 & $\PP^2$, $\PP^1\times\PP^1$, including $M=K_S$ and line bundles without
   sections; $E\times\PP^1$
 & point & $N\leq3$, $|\mm|\leq4$; $N\leq2$, $|\mm|\leq2$ for $E\times\PP^1$\\\addlinespace
Corollary~\ref{cor:OP}
 & $\PP^2$, $\PP^1\times\PP^1$ & point & $N=2$, $n\leq3$\\
\bottomrule
\end{tabular}
\medskip
\caption{Statements compared with the localization formula
\eqref{eq:toric-bott}. The last column gives the range of $N$ and of the
lengths within which the cases lie. For $B=\PP^1$ the families of curves
of positive genus are constant.}
\label{tab:verification}
\end{table}

\begin{remark}\label{rem:verification-scope}
The comparisons of Table~\ref{tab:verification} do not include families of
curves with singular members over a base of positive dimension. They also
do not include the reduction of Sections~\ref{sec:curves}
and~\ref{sec:residue-geometry} over a base of dimension at least two, for
which hypothesis \eqref{eq:positivity-hypothesis} requires curves of high
genus.
\end{remark}

The computations of Sections~\ref{subsec:computer}
and~\ref{subsec:verification} were carried out with programs written in
Python~\cite{AndersonPrograms}.

\clearpage
\begingroup
\makeatletter
\let\@tocwrite\@gobbletwo
\makeatother

\endgroup
\end{document}